\documentclass[reqno, 11pt]{amsart}
\usepackage[top=3.75cm, bottom=3.75cm, left=2.5cm, right=2.5cm, headsep=0.75cm]{geometry}            
\usepackage{amsmath,amsrefs,amssymb}
\usepackage{pgfplots}
 \usepackage{tikz,float}
\pgfplotsset{compat = newest}

\usepackage{graphicx}
\usepackage[colorlinks=true, pdfborder={0 0 0}]{hyperref}
\usepackage{subcaption}
\usepackage[font=scriptsize]{caption}
\usepackage{mwe}
\usepackage{dutchcal}

\newcommand{\R}{\mathbb{R}}
\newcommand{\N}{\mathbb{N}}

\newcommand{\B}{\mathbb{B}}
\newcommand{\Be}{\mathsf{B}}

\DeclareMathOperator{\bigO}{O}

\newcommand{\eps}{\varepsilon}

\newtheorem{theorem}{Theorem}[section]
\theoremstyle{plain}

\newtheorem{corollary}{Corollary}[section]

\newtheorem{definition}{Definition}[section]

\newtheorem{lemma}{Lemma}[section]

\newtheorem{proposition}{Proposition}[section]
\newtheorem{remark}{Remark}[section]

\usepackage{thmtools}

\newtheorem{theorema}{Theorem}[section]

\newcommand{\la} {\lambda}
\newcommand{\rn}{\mathbb{R}^{n}}
\newcommand{\bn}{\mathbb{B}^{n}}

\newcommand{\calu}{\mathcal{U}}
\newcommand{\dvg}{{\rm d}v_{\bn}}
\newcommand{\tstar}{2^{\star}}

\newcommand{\dx}{{\rm d}x}
\newcommand{\cn}{c(n, \lambda)}

\numberwithin{equation}{section}

\allowdisplaybreaks

\begin{document}

\title[Sharp quantitative stability of Struwe's decomposition in the hyperbolic space]{Sharp quantitative stability of Struwe's decomposition of the Poincar\'e-Sobolev inequalities on the hyperbolic space: Part I}

\author{Mousomi Bhakta}
\address{Mousomi Bhakta, Department of Mathematics\\
Indian Institute of Science Education and Research Pune (IISER-Pune)\\
Dr Homi Bhabha Road, Pune-411008, India}
\email{mousomi@iiserpune.ac.in}
\author{Debdip Ganguly}
\address{Debdip Ganguly, Department of Mathematics\\
Indian Institute of Technology Delhi\\
IIT Campus, Hauz Khas, New Delhi, Delhi 110016, India}
\email{debdipmath@gmail.com, debdip@maths.iitd.ac.in}
\author{Debabrata Karmakar}
\address{Debabrata Karmakar, Tata Institute of Fundamental Research\\
Centre For Applicable Mathematics\\
Post Bag No 6503, GKVK Post Office, Sharada Nagar, Chikkabommsandra, Bangalore 560065, India}
\email{debabrata@tifrbng.res.in}
\author{Saikat Mazumdar}
\address{Saikat Mazumdar, Department of Mathematics\\
Indian Institute of Technology Bombay\\
Mumbai 400076, India}
\email{saikat.mazumdar@iitb.ac.in}

\date{\today}


\begin{abstract}
A classical result owing to Mancini and Sandeep \cite{MS} asserts that all positive solutions of the Poincar\'e-Sobolev equation on the hyperbolic space
$$
-\Delta_{\mathbb{B}^n} u \, - \, \lambda u \, = \, |u|^{p-1}u, \quad u\in H^1(\mathbb{B}^n),
$$
are unique up to hyperbolic isometries where $n \geq 3,$ $1 < p \leq \frac{n+2}{n-2} $  and $\lambda \leq \frac{(n-1)^2}{4}.$
 
In the spirit of Struwe, 
Bhakta-Sandeep \cite{BS} proved the following non-quantitative stability result:  if $u_m \geq 0,$ and 
$
\|\Delta_{\mathbb{B}^n} u_m+ \la u_m +  u_m^{p}\|_{H^{-1}} \rightarrow 0, 
$
then $  \delta(u_m) := \mbox{dist}(u_m, \mathcal{M_{\lambda}}) \rightarrow 0,$ where $\mbox{dist}(u_m, \mathcal{M_{\lambda}})$ denotes the $H^1$-distance of $u_m$ from the manifold of 
sums of superpositions of  \it hyperbolic bubbles \rm and (localized) Aubin-Talenti bubbles. 

\medskip

In this article, we study the quantitative stability 
of Struwe decomposition. We prove under certain bounds on $\|\nabla u\|_{L^2(\mathbb{B}^n)}$ the inequality
$$
\delta(u) \lesssim    \|\Delta_{\mathbb{B}^n} u+ \la u +  u^{p}\|_{H^{-1}},
$$
 holds whenever $p >2$ and hence forcing the dimensional restriction $3 \leq n \leq 5.$ Moreover, it fails for any $n \geq 3$ and $p \in (1,2].$ This strengthens the phenomenon observed in the Euclidean case that the (linear) quantitative stability estimate depends only on whether the exponent $p$ is $>2$ or $\leq 2$. In the critical case, \it our dimensional constraint \rm coincides with the seminal result of Figalli and Glaudo \cite{FG} but we notice a striking dependence on the exponent $p$ in the subcritical regime as well which is not present in the flat case.
 
  Our technique is an amalgamation of Figalli and Glaudo's method and builds upon a series of new and novel estimates on the interaction of \it hyperbolic bubbles \rm and their derivatives and improved eigenfunction integrability estimates. Since the conformal group coincides with the isometry group of the hyperbolic space, we perceive a remarkable distinction in arguments and techniques to achieve our main results compared to that of the Euclidean case.

 \end{abstract}
\maketitle
\setcounter{tocdepth}{1}
\tableofcontents

\medskip

\noindent
\section{Introduction}

In this article, we continue our study on the stability of Poincar\'e-Sobolev inequalities in the hyperbolic space. The focus of this article is on the multi-bubble case. Let us first introduce to the reader the Sobolev embedding, generally refer to as Poincar\'e-Sobolev inequality on the hyperbolic space. Throughout this article, $\bn$ denotes the Poincar\'e ball model of the hyperbolic space equipped with the metric $g:= \left(\frac{2}{1-|x|^2}\right)^2 \; {\rm d}x^2$ and $\Delta_{\bn}, \nabla_{\bn}$ denote the Laplace-Beltrami operator and the gradient operator respectively, and $\dvg$ denotes the volume element.

\subsection{Poincar\'e-Sobolev inequality on the hyperbolic space.} Let $n \geq 3$ and $\lambda \leq \frac{(n-1)^2}{4}$ and $1 < p \leq \tstar -1.$ Then  there exists a best constant $S_{\la,p}:=S_{\lambda,p}(\B^{n})>0$ such that 
\begin{align}\label{S-inq}
S_{\lambda,p}\left(~\int \limits_{\B^{n}}|u|^{p+1} \, \dvg \right)^{\frac{2}{p+1}}\leq\int \limits_{\B^{n}}\left(|\nabla_{\B^{n}}u|^{2}-\lambda u^{2}\right)\, \dvg
\end{align}
holds for all  $u\in C_c^{\infty}(\B^n),$ where $\tstar = \frac{2n}{n-2}.$ 

\medskip

By density \eqref{S-inq} continues to hold for every $u$ belonging to the the closure of $C_c^{\infty}(\B^n)$ with respect to the norm
\begin{align}
\|u \|_{\lambda} = \left(\int \limits_{\B^{n}}\left(|\nabla_{\B^{n}}u|^{2}-\lambda u^{2}\right)\, \dvg\right)^{\frac{1}{2}}.
\end{align}
The quantity $\la_1(\bn):= \frac{(n-1)^2}{4}$ is the $L^2$-bottom of the spectrum of $-\Delta_{\bn}$ defined by 
\begin{align*}
\lambda_1(\bn) = \inf_{\|u\|_{L^2(\bn)} = 1} \|\nabla_{\bn} u\|_{L^2(\bn)}^2.
\end{align*}

As a result, if $\lambda < \frac{(n-1)^2}{4},$ then the closure is the classical Sobolev space $H^1(\bn).$ Otherwise, the closure is a larger space that we will denote by $\mathcal{H}^1(\bn)$ and which contains non-square-integrable functions. We refer to Section \ref{basic section} for a precise definition of $H^1(\bn).$  

\medskip

The inequality \eqref{S-inq} was first obtained by Mancini and Sandeep in \cite{MS} and in the same article, they also proved the existence of optimizers under appropriate assumptions on $n,\lambda$ and $p.$ In particular, they showed that under the assumption that

\begin{align*}
\mbox{{(\bf H1)}} \ \ \ \ \ \ \ \ \ \ \ \ \ \ \ \ \ \ \
\begin{cases}
\lambda < \frac{(n-1)^2}{4}, \ \ \ \mbox{when} \ 1 < p < \frac{n+2}{n-2}, \ \mbox{and} \ n \geq 3, \\ \\
\frac{n(n-2)}{4} < \lambda < \frac{(n-1)^2}{4}, \ \ \ \mbox{when} \  p = \frac{n+2}{n-2}, \ \mbox{and} \ n \geq 4, \ \ \ \ \ \ \ \
\end{cases}
\end{align*}

\medskip

 there always exists a strictly positive, radially symmetric and decreasing extremizer $\calu:= \calu_{n,\la,p}$ in $H^1(\bn).$ For simplicity of notations we will always denote the radially symmetric solution by $\calu$ and the dependence of $n,\la$ and $p$ will be implicitly assumed.  It is straightforward to verify that subject to the normalization 
 
 \begin{align} \label{calu}
 \bullet \ \|\calu\|_{\lambda}^2 = S_{\lambda,p}^{\frac{p+1}{p-1}}, \ \mbox{in addition to} \ \calu \in H^1(\bn), \calu > 0 \ \mbox{radially symmetric and decreasing},
\end{align}

\medskip

the obtained extremizer is a positive solution to the Euler-Lagrange equation 

\begin{align}\label{eq1}
 -\Delta_{\B^{n}}u-\lambda u=|u|^{p-1}u~ \ \ \ \  u \in H^{1}(\B^{n}).
 \end{align}
 
 \medskip
 
The equation \eqref{eq1} as well as the inequality \eqref{S-inq} is invariant under the conformal group of the ball model, which in this case coincides with the isometry group of the ball model and is generated by the hyperbolic translations $\tau_b, b\in \bn.$ In analogy with the Euclidean case $\tau_b$ can be thought of as the translation with takes $0$ to $b.$ In \cite{MS} Mancini and Sandeep also classified the positive solutions of \eqref{eq1} and which in turn provides the classification of the extremizers of \eqref{S-inq}. Their results are as follows: Under the assumptions {\bf (H1)} the $n$-dimensional manifold defined by
\begin{align*}
\mathcal{Z}_{0}:=\{\calu[b] := \mathcal{U}\circ\tau_{b}: b \in \B^{n}\}
\end{align*}
consists of all the positive solutions to \eqref{eq1} and $c\mathcal{Z}_0, \ c \in \mathbb{R}\backslash \{0\}$ consists of all the nontrivial extremizers of \eqref{S-inq}. Henceforth we shall call the elements of $\mathcal{Z}_0$ a {\it hyperbolic bubble}. 

\medskip

Thanks to \eqref{S-inq} the solutions of \eqref{eq1} are the critical points of the energy functional
\begin{align} \label{ilambda}
I_{\lambda}(u)= \frac{1}{2}\int \limits_{\B^{n}}\left(|\nabla_{\B^{n}}u|^{2}-\lambda u^{2}\right)~\dvg-\frac{1}{p+1} \int \limits_{\B^{n}}|u|^{p+1}~\dvg.
\end{align}
Using the conformal invariance of the norms stated in Section \ref{basic section}, in particular, Lemma \ref{lemma1}, it is easy to see that all the hyperbolic bubbles have the same energy 
$$I_{\la}(\calu \circ \tau_b) = \frac{p-1}{2(p+1)}S_{\la,p}^{\frac{p+1}{p-1}}.$$

\subsection{The Euclidean Sobolev inequality.} The classical Sobolev inequality in $\mathbb{R}^n,$ $n \geq 3,$ asserts that there exists a best constant $S := S(\R^n)$ such that
\begin{align}\label{Sobolev}
S \left(~\int \limits_{\R^{n}}|u|^{2^{\star}}{\rm d}x\right)^{2/2^{\star}} \leq \int \limits_{\R^{n}}|\nabla u|^{2} \ {\rm d}x
\end{align}
holds for all $u \in C_c^{\infty}(\R^n),$ where $2^{\star} = \frac{2n}{n-2}.$ By density argument, the inequality \eqref{Sobolev} continues to hold for all $u$  
satisfying $\|\nabla u\|_{L^2(\R^n)} <\infty,$ and $\mathcal{L}^n(\{|u| >t\}) < \infty$ for every $t>0,$ where $\|\cdot \|_{L^2(\R^n)}$ denotes the $L^2$-norm and $\mathcal{L}^n$ denotes the Lebesgue measure on $\R^n.$  Henceforth we shall denote the closure by $H^1(\rn).$
The value of $S$ is known and the equality cases in \eqref{Sobolev} have been well studied. The inequality \eqref{Sobolev} is invariant under the action of the conformal group of $\rn$ composed of the translations, dilations and inversions: for $z \in \rn, \mu > 0,$  and $u \in C_{c}^{\infty}(\rn)$ define $T_{z,\mu}(u) = \mu^{\frac{n-2}{2}}u(\mu(\cdot - z)),$ then both the norms in \eqref{Sobolev} are preserved. It is well known that \eqref{Sobolev}  is achieved if and only if $u$ is a constant multiple of the {\it Aubin-Talenti} bubbles \cite{AT, TAL}
\begin{align}\label{AT bubbles}
U[z,\mu](x) = (n(n-2))^{\frac{n-2}{4}}\mu^{\frac{n-2}{2}} \frac{1}{(1 + \mu^2|x-z|^2)^{\frac{n-2}{2}}}, \quad z \in \R^n,\, \mu>0. 
\end{align}
Therefore the set of equality cases in \eqref{Sobolev} forms a $(n+2)$ dimensional manifold. The choice of the dimensional constant in the definition of $U[z,\mu] = T_{z,\mu}(U[0,1])$ ensures that $U$ is a positive solution of the corresponding Euler-Lagrange equation 
\begin{align}\label{eleu}
-\Delta u  = |u|^{2^{\star} - 2}u \ \mbox{ in } \ \R^n.
\end{align}

Thanks to the Sobolev inequality, all solutions to \eqref{eleu} are the critical points of 
\begin{align*}
J(v) = \frac{1}{2}\int_{\rn}|\nabla v|^2 \ dx - \frac{1}{\tstar} \int_{\rn} |v|^{\tstar} \  dx, \ \ \ \ v \in H^{1}(\rn).
\end{align*}
and moreover all the Aubin-Talenti bubbles have the same energy $J(U[z,\la]) = \frac{1}{n}S(\rn)^{\frac{n}{2}}.$

\subsection{The quantitative stability problem}
In connection with Aubin and Talenti's results on the equality cases of the Euclidean Sobolev inequality, Brezis and Lieb \cite{BL} asked the following: if $u$ almost optimizes \eqref{Sobolev} i.e. if the deficit 
\begin{align*}
\delta_{Eucl}(u) = \|\nabla u\|^{2}_{L^2(\rn)}-S \|u\|^{2}_{L^{\tstar}(\rn)}
\end{align*}
is small then is it true that $u$ is close, in some appropriate sense, to the family of Aubin-Talenti bubbles? 
This question is positively answered by Bianchi and Egnell \cite{BE} nearly thirty years ago in the following quantitative form
\begin{align*}
\inf_{\mu \in \R, z \in \rn, c \in \R} \left ( \frac{\|\nabla u - cU[z,\mu]\|_{L^2(\rn)}}{\|\nabla u\|_{L^2(\rn)}}\right) \leq C(n)\delta_{Eucl}(u)^{\frac{1}{2}},
\end{align*}
strengthening the original results of Brezis and Lieb in bounded domains.
Here it is important to remark that the norm $\|\nabla u\|_{L^2(\rn)}$ is the strongest possible norm and the exponent $\frac{1}{2}$ is sharp, in the sense that if we replace $\frac{1}{2}$ by any other exponent strictly bigger than $\frac{1}{2}$ then the inequality fails as $\delta_{Eucl}(u) \to 0.$

\medskip

The stability problem has been a centre of attraction for several authors for quite some time.  Many possible generalizations and improvements in several directions have appeared in the literature. A notable generalization in this direction is the stability for $p$-Sobolev inequality for $p \neq 2.$  In \cite{EPaT} the authors 
obtained results similar to Brezis and Lieb but the Bianchi-Egnell type result was open for a long time. After some remarkable results of Cianchi \cite{CI}, Cianchi-Fusco-Maggi-Pratelli \cite{CFMP}, Figalli-Neumayer \cite{FN} finally Figalli and Zhang \cite{FZp} obtained sharp quantitative stability results for $p \neq 2$. Our goal for this article is somewhat different what we are going to describe now.

\medskip

Another remarkable direction is the stability problem from the Euler-Lagrange point of view initiated by Cirallo-Figalli and Maggi \cite{CFM}. The stability question can also be posed as follows: if $u$ almost solves the corresponding Euler-Lagrange equation then is it true that $u$ is close to the family of Aubin-Talenti bubbles in a quantitative way? The problem is far more challenging and generally false unless some more hypotheses are imposed upon it. To see this note that $u = U[-Re_1, 1] + U[Re_1, 1]$ almost solves \eqref{eleu} i.e. $\|\Delta u + |u|^{p-1}u\|_{H^{-1}}$ is small but clearly $u$ is not close to the family of Aubin-Talenti bubbles. Here it is important to remark that  the interaction between the bubbles 
\begin{align*}
\int_{\rn} \nabla U[-Re_1,1] \cdot \nabla U[Re_1,1] \ dx = \int_{\rn} U[-Re_1,1]^p U[Re_1,1] \ dx \approx R^{2-n}=o(1) \ \ \mbox{as} \ \ R \rightarrow \infty,
\end{align*}
which makes $\|\Delta u + |u|^{p-1}u\|_{H^{-1}}$ small.

\medskip

 This shortcoming  can easily be addressed by rephrasing the question to `Does $u$ close to a sum of Aubin-Talenti bubbles?' Yet this is not true as there are examples of non-trivial sign-changing solutions which are not the sum of Aubin-Talenti bubbles \cite{DingY}. However, the classification results of Gidas, Ni and Nirenberg ensure all the positive solutions to \eqref{eleu}
 are the Aubin-Talenti bubbles \cite{GNN}.
 
 \medskip
 
 A seminal work of Struwe \cite{MS2} (also see \cite{MS1}) which dates back to the results of Bianchi and Egnell asserts that imposing a non-negativity assumption on $u$ one can get a non-quantitative stability result.
\begin{theorema}[Struwe, 1984]
 Let $n\geq 3$ and $N \geq 1$ be positive integers. Let $\{u_m\} \subset H^1(\rn)$ be a sequence of nonnegative functions such that $(N - \frac{1}{2})S^n \leq \|\nabla u_m\|_{L^2(\rn)} \leq (N - \frac{1}{2})S^n,$ and assume that
 \begin{align*}
 \|\Delta u_m + u_m^{\tstar - 1}\|_{H^{-1}} \rightarrow 0 \ \ \mbox{as} \ \ m \rightarrow \infty.
 \end{align*}
Then there exists a sequence $(z_1^m,\ldots, z_N^m)_{m \in \mathbb{N}}$ of $N$-tuple of points in $\rn$ and sequence $(\la_1^m, \ldots, \la_N^m)$ of $N$-tuple of positive real numbers such that 
\begin{align*}
\Big \|\nabla \left( u_m - \sum_{i=1}^N U[z_i^m, \la_i^m]\right) \Big \|_{L^2} \rightarrow 0 \ \ \mbox{as} \ \ m \rightarrow \infty.
\end{align*}
Moreover the Aubin-Talenti bubbles  $U[z_i^m, \la_i^m]$, $U[z_j^m, \la_j^m]$, for $i\neq j$ do not interact with each other at the $H^{1}$-level
\begin{align*}
\int_{\rn} \nabla U[z_i^{m},\la_i^{m}] \cdot \nabla U[z_j^{m},\la_j^{m}] \ dx &= \int_{\rn} U[z_i^{m},\la_i^{m}]^{\tstar-1} U[z_j^{m},\la_j^{m}] \ dx \\ 
&\approx \min \left(\frac{\la_i^{m}}{\la_j^{m}}, \frac{\la_j^{m}}{\la_i^{m}}, \frac{1}{\la_i^{m}\la_j^{m}|z_i^{m}-z_j^{m}|^2}\right)^{\frac{n-2}{2}}\rightarrow 0 \ \ \mbox{as} \ \ m \rightarrow \infty. 
\end{align*}
\end{theorema}

 A sharp quantitative form of Struwe decomposition has recently been obtained by Figalli and Glaudo \cite{FG}. A clever way of getting around the problem of sign-changing $u$ is by phrasing the problem locally around the sum of Aubin-Talenti bubbles as nicely done by Cirallo-Figalli-Maggi \cite{CFM} for single bubble case and Figalli-Galudo \cite{FG} for the multi-bubble case in dim $3 \leq n \leq 5$, and more recently by Deng-Sun and Wei \cite{Weietal} in dimension $n \geq 6.$ It is remarkable that for multi-bubble cases the linear dependence of $\|\Delta u + |u|^{p-1}u\|$ heavily relies on the underlying dimension. In other words, unlike the single bubble case, the form of quantitative stability results for the multi-bubble case, the underlying dimension plays a significant role and it seems this difference results from the behaviour of the interacting bubbles \cite{FG, Weietal}.
To state their results we recall the definition of weak interaction among the bubbles. 


A family of Aubin-Talenti bubbles $\{U[z_i,\la_i]\}_{1 \leq i \leq N}$ is said to be $\delta$-interacting if the interaction between any two bubbles is less than $\delta$ i.e.
\begin{align*}
 \max_{1 \leq i \neq j \leq N}\min \left(\frac{\la_i}{\la_j}, \frac{\la_j}{\la_i}, \frac{1}{\la_i\la_j|z_i-z_j|^2}\right)^{\frac{n-2}{2}} \leq \delta.
\end{align*}
Note that the definition makes sense only if there is more than one bubble. The results obtained so far in the Euclidean case have been summarized in the following theorem.
\begin{theorema} \label{stability euclidean theorem}
 Let $n\geq 3$ and $N \geq 1$ be positive integers. There exist a small constant $\delta = \delta(n,N)$ and a large constant $C = C(n,N) > 0$ such that the following statement holds: let $u \in H^1(\rn)$ be a function such that 
 \begin{align*}
\Big \|\nabla u - \sum_{i=1}^N \nabla \tilde U\Big \|_{L^2} \leq \delta
\end{align*}
where $\{\tilde U_i\}_{1 \leq i \leq N}$ is a \underline{$\delta$-interaction family of }\footnote{when $N=1$ the underlined statement should be omitted.}  Aubin-Talenti bubbles.
Then there exists $N$ Aubin-Talenti bubbles $U_1,\ldots, U_N$ such that 
\begin{itemize}
\item[(a)] Cirallo-Figalli-Maggi \cite{CFM}. \ If $N=1$ then 
\begin{align*}
\Big \|\nabla  u -  \nabla U_1 \Big \|_{L^2} \leq C\|\Delta u + |u|^{\tstar -2}u\|_{H^{-1}}.
\end{align*}
\item[(b)] Figalli-Glaudo \cite{FG}. If $N>1$ and $3 \leq n \leq 5$ then 
\begin{align} \label{stability e}
\Big \|\nabla  u -  \sum_{i=1}^N\nabla U_i \Big \|_{L^2} \leq C\|\Delta u + |u|^{\tstar -2}u\|_{H^{-1}}.
\end{align}
Moreover, for any $i \neq k$, the interaction between the bubbles can be estimated
as
\begin{align*}
\int_{\rn} U_i^{\tstar -1}U_k \ dx \leq C\|\Delta u + |u|^{\tstar -2}u\|_{H^{-1}}.
\end{align*}
Furthermore, the dimensional restriction is optimal in the sense that for $n \geq 6$ and $N=2$ there exists a sequence $\{u_R\} \subset H^1(\rn)$ such that \eqref{stability e} fails to hold for any  $C= C(n)$ as $R \rightarrow \infty.$

\item[(c)]Ding-Sun-Wei \cite{Weietal}. For $n \geq 6$ the following optimal \footnote{optimality in the same sense as in (b).} inequality holds

\begin{align*}
\Big \|\nabla  u -  \sum_{i=1}^N\nabla U_i \Big \|_{L^2} \leq C
\begin{cases}
\|\Delta u + |u|^{\tstar -2}u\|_{H^{-1}}\big| \ln \|\Delta u + |u|^{\tstar -2}u\|_{H^{-1}}\big|^{\frac{1}{2}}, \ \ \mbox{if} \ \ n =6, \\ \\
\|\Delta u + |u|^{\tstar -2}u\|_{H^{-1}}^{\frac{n+2}{2(n-2)}}, \ \ \mbox{if} \ \ n \geq 7.
\end{cases}
\end{align*}

\end{itemize}
\end{theorema}

Before proceeding to the next section let us mention to the interested readers some existing stability results in the literature. We refer to the beautiful treatises \cite{ISO1, ISO2, ISO3, ISO4, CI, CFMP, CF, DE1, DT, FMP, FMP-1, FN, BDNS, DE,  Ngu, Ruff, FZp, CKNWei, SA, Neum, Seu} which covers, for example, the stability of isoperimetric inequality, $1$-Sobolev inequality, Caffarelli-Kohn-Nirenberg inequality, log-Sobolev inequality,  Gagliardo-Nirenberg-Sobolev inequality, Hardy-Littlewood-Sobolev, 
logarithmic Hardy-Littlewood-Sobolev inequality on the Euclidean space. On the values of optimal constant $C$ in the Bianchi-Egnell result we refer to \cite{DEFFL} (see also \cite{Kon}). On the stability of isoperimetric inequality on the sphere see \cite{isosphere} and on the hyperbolic space, we refer to \cite{isohyperbolic}.
See also the recent articles  \cite{Frank} for stability results on $\mathbb{S}^1(1/\sqrt{n-2}) \times S^1(1)$ and \cite{Max, NVo} for related stability on a  general Riemannian manifold. 

\medskip

For the critical point of view, other than the results \cite{CFM, FG,Weietal} mentioned above, one can consult the beautiful monographs \cite{CM17, KM17} for the Euclidean isoperimetric inequality. The beautiful survey article by Fusco
\cite{Fuscosurvey} gives a broad description of the stability results concerning several other related geometric and functional inequalities.

\subsection{Known results in the hyperbolic space}
In a recent article \cite{BGKM}, we studied the quantitative stability of the Poincar\'e-Sobolev inequality in the spirit of Bianchi and Egnell \cite{BE}. The bottom of the spectrum of $-\Delta_{\bn}$ being positive, the hyperbolic space admits sub-critical Sobolev inequality as well (i.e. $p < \tstar-1$) and we established the stability for both critical and subcritical Sobolev inequalities. 

\medskip

The non-quantitative form of the Poincar\'e Sobolev inequality on the hyperbolic space have been obtained nearly a decade ago by the first author jointly with Sandeep \cite{BS}. They establish the following Struwe-type profile decomposition both for the subcritical and the critical case.
Needless to say in the subcritical case there are no Aubin-Talenti bubbles and hence there are only hyperbolic bubbles present in the profile decomposition. However, in the critical case one can not exclude the possibility of the presence of suitably localized Aubin-Talenti bubbles as indicated below.

\medskip

For the critical case $p = \tstar-1,$ there are two types of bubbles. The hyperbolic bubbles which are of the form $\mathfrak{u}_k = \calu \circ \tau_{b_k}$ for some sequence $b_k \in \bn$ and localized Aubin-Talenti bubbles which are of the form 
\begin{align*}
\mathfrak{v}_k = \left(\frac{1 - |x|^2}{2}\right)^{\frac{n-2}{2}}\phi(x) U[x_0, \epsilon_k]
\end{align*}
where $\epsilon_k \rightarrow 0+$ and $\phi$ is a smooth cut-off function such that $0 \leq \phi \leq 1,$ and $\phi \equiv 1$ in a neighbourhood of $x_0,$ and $U[x_0,\epsilon_k]$ are the standard Aubin-Talenti bubbles defined above. Using the compact $L^2_{\tiny{loc}}$-embedding, it is easy to verify that $I_{\la}(\mathfrak{v}_k) = J(U[x_0, \epsilon_k]) + o(1)$ as $\epsilon_k \rightarrow 0.$ 
As a result the hyperbolic bubbles have energy $I_{\la}(\mathfrak{u}_k^j) = \frac{1}{n}S_{\la,p}^{\frac{n}{2}}$ for all $j$ while the Aubin-Talenti bubble has energy $I_{\la}(\mathfrak{v}_k^j) = \frac{1}{n}S^{\frac{n}{2}} + o(1)$ as $\epsilon_k \rightarrow 0+$ for all $j.$

\begin{theorema}[Bhakta and Sandeep \cite{BS}] \label{profile decomposition}
Let $n \geq 3, \lambda, p$ satisfies the hypothesis {\bf(H1)}. Let $\{u_m\} \subset H^1(\bn)$ be a sequence such that $I_{\la}(u_m) \rightarrow d$ and 
\begin{align*}
\|\Delta_{\bn} u_m+ \la u_m + |u_m|^{p-1}u_m\|_{H^{-1}} \rightarrow 0, \ \ \mbox{as} \ \ m \rightarrow \infty.
\end{align*}

 Then there exists $N_1,N_2 \in \mathbb{N} $ and functions $\mathfrak{u}_m^j \in H^1(\bn), 1 \leq j \leq N_1, \mathfrak{v}_m^j, 1 \leq j \leq N_2$ and $u \in H^1(\bn)$ such that up to a subsequence 
\begin{align*}
u_m = u + \sum_{j=1}^{N_1} \mathfrak{u}_m^j + \sum_{j=1}^{N_2} \mathfrak{v}_m^j + o(1) \ \ in \ H^1(\bn),
\end{align*}
where $I_{\la}^{\prime}(u) = 0$ in $H^{-1}(\bn)$ and $\mathfrak{u}_m^j, \mathfrak{v}_m^j$ are defined as above. Furthermore, we have 
\begin{align*}
d = 
\begin{cases}
 I_{\la}(u)  + N_1\frac{p-1}{2(p+1)}S_{\la,p}^{\frac{p+1}{p-1}}, \ \ \ \ \ \ \ \ \ \ \ \ \ \mbox{if} \ p < \tstar-1, \\ \\
 I_{\la}(u) + \frac{N_1}{n}S_{\la,p}^{\frac{n}{2}} + \frac{N_2}{n}S^{\frac{n}{2}}, \ \ \ \ \ \ \ \ \ \ \ \mbox{if} \ p =\tstar-1.
\end{cases}
\end{align*}
\end{theorema}

\begin{remark}
{\rm
If one or more kinds of bubbles are absent we will write $N_i = 0, i= 1,2.$ For example if the Aubin-Talenti bubbles are absent from the profile decomposition we simply write $N_2 = 0.$ Note that this is indeed the case if $p < \tstar-1$ or $d < \frac{1}{n}S^{\frac{n}{2}}$ for $p=\tstar-1.$ 
}
\end{remark}

\begin{remark} \label{rem2}
{\rm
It is worth noting that in the above theorem, we have imposed no sign assumption on the sequence $u_m$ and hence the $u$,  which is obtained by just taking the weak limit of $u_m$, may be different than the standard hyperbolic bubble. However, if we impose $u_m$ are non-negative, then from the uniqueness of the positive solutions to \eqref{eq1} we obtain either $u \equiv 0$ or $u = \calu \circ \tau_b$ for some $b \in \bn.$ In either case we have the energy quantization
\begin{align*}
d =  \frac{\tilde N_1}{n}S_{\la,p}^{\frac{n}{2}} + \frac{N_2}{n}S^{\frac{n}{2}}, \ \ \mbox{if} \ \ p = \tstar-1, \ \ \mbox{or} \ \ d= \tilde N_1\frac{p-1}{2(p+1)}S_{\la,p}^{\frac{p+1}{p-1}}, \ \ \mbox{if} \ \ p <\tstar-1,
\end{align*}
where $\tilde N_1 $ equal to either $N_1$ or $N_1 + 1.$ 
}
\end{remark}

\begin{remark}
{\rm 
In Section \ref{absence section} we show that if $p=\tstar-1$ and $d = \frac{ N}{n}S_{\la,p}^{\frac{n}{2}} $ for some $N\in \N$ then there exists an at most countable set $\Lambda$ such that if $\la \not\in \Lambda$ and $u_m \geq 0$ satisfies the assumptions of Theorem \ref{profile decomposition} then $N_2 = 0,$ i.e. the Aubin-Talenti bubbles are absent in the profile decomposition. This is the starting point of our result and quite remarkably the proof relies on a very simple observation of the strict decreasing property of the function $\la \mapsto S_{\la, \tstar-1.}$
}
\end{remark}

\subsection{Main results of the article}  

In Section \ref{interaction of bubbles section} we will show that 
\begin{align*}
\int_{\bn} \left(\langle \nabla_{\bn} \calu[z_i], \nabla_{\bn}\calu[z_k]\rangle_{\bn}  - \la \calu[z_i]\calu[z_k] \right)\ \dvg= \int_{\bn} \calu[z_i]^p \calu[z_k] \ \dvg \approx e^{-\cn d(z_i,z_k)},
\end{align*}
where 
\begin{equation}\label{cn-lambda}
\cn=\frac{n-1+\sqrt{(n-1)^2-4\lambda}}{2},
\end{equation}
is the rate of decay of the hyperbolic bubbles: $\calu(x) \approx e^{-\cn d(x,0)}.$
 This estimate led us to the following definition. 

\begin{definition}\label{d:delta-int}
Let $\calu_1 = \calu[z_1], \ldots, \calu_N = \calu[z_N]$ be a family of hyperbolic bubbles. We say that the family is $\delta$-interacting for some $\delta >0$ if 
\begin{align} \label{delta interaction}
\max_{ 1 \leq i \neq k \leq N} e^{-\cn d(z_i,z_k)} \leq \delta. 
\end{align}
\end{definition}
If we consider for some $\alpha_1,\ldots,\alpha_N \in \mathbb{R}$ the family $\alpha_i\calu_i$,  $1 \leq i \leq N,$ then 
the family is said to be $\delta$-interacting if in addition to \eqref{delta interaction}
\begin{align*}
\max_{1\leq i \leq N} \ |\alpha_i - 1| \leq \delta
\end{align*}
holds. 

\medskip
Therefore $\calu_i, \calu_k$ belongs to a $\delta$-interacting family if their $H^1$-scalar product is bounded by a constant multiple of
$\delta.$  Throughout the article, we shall use the following notations:
\begin{align}\label{the Q}
Q_{ik} = e^{-\cn d(z_i,z_k)}, \ \ Q_k = \max_{i \neq k} \ Q_{ik} \ \ \mbox{and} \ \ Q = \max_k \ Q_k
\end{align}
where $1 \leq i,\, k \leq N.$
\medskip

The main result of this article is the following 

\begin{theorem} \label{main theorem}
Let $3 \leq n \leq 5$  and $p > 2$ such that  $p$ and $\la$ satisfy {\bf (H1)}. For any $N \in \mathbb{N}$ there exists a small constant $\delta = \delta(n,\la,p,N)>0$ and a large constant $C = C(n,p,\la,N) >0$ such that the following statement holds: let $u \in H^1(\bn)$ be a function such that 
\begin{align*}
\Big\|u - \sum_{i=1}^N \tilde \calu_i\Big\|_{\la} \leq \delta,
\end{align*}
where $(\tilde \calu_i)_{1 \leq i \leq N}$ is a $\delta$-interacting  family of hyperbolic bubbles. Then there exists $N$ hyperbolic bubbles $(\calu_i)_{1 \leq i \leq N}$ such that 
\begin{align*}
\Big\|u - \sum_{i=1}^N  \calu_i\Big\|_{\la} \leq C \| \Delta_{\bn} u + \la u + |u|^{p-1}u\|_{H^{-1}}.
\end{align*}
Moreover, for any $i \neq k,$ the interaction between the bubbles can be estimated as
\begin{align*}
\int_{\bn} \calu_i^p \calu_k \ \dvg \leq C \| \Delta_{\bn}u + \la u + |u|^{p-1}u\|_{H^{-1}}.
\end{align*}
 \end{theorem}

As an application of Theorem \ref{main theorem} and the profile decomposition Theorem \ref{absence theorem} proved in Section \ref{absence section} we have the following corollary.

\begin{corollary} \label{main corollary}
Let $3 \leq n \leq 5$  and $p > 2$ such that  $p$ and $\la$ satisfy {\bf (H1)} and $N \in \mathbb{N}.$ Then there exists a small constant $\delta = \delta(n,\la,p,N)>0$ and a large constant $C = C(n,p,\la,N) >0$ and a countable subset $\Lambda$ of $\mathbb{R}$ if $p = \tstar-1$ such that for $\la \not\in \Lambda$ the following statement holds: For any 
non-negative function $u \in H^1(\bn)$ satisfying
\begin{align*}
\left(N-\delta \right)S_{\la,p}^\frac{p+1}{p-1} \leq \|u\|^2_{\la} \leq \left(N-\delta \right)S_{\la,p}^\frac{p+1}{p-1},
\end{align*}
there exists $N$ hyperbolic bubbles $(\calu_i)_{1 \leq i \leq N}$ such that 
\begin{align*}
\Big\|u - \sum_{i=1}^N  \calu_i\Big\|_{\la} \leq C \| \Delta_{\bn}u + \la u + u^p\|_{H^{-1}}.
\end{align*}
Moreover, for any $i \neq k,$ the interaction between the bubbles can be estimated as
\begin{align*}
\int_{\bn} \calu_i^p \calu_k \ \dvg \leq C \| \Delta_{\bn} u+ \la u + |u|^p\|_{H^{-1}}.
\end{align*}
\end{corollary}

After proving the stability result we look for the counter-example when $n \geq 3$ and $1 < p \leq 2.$ We prove that there exists a one-parameter family $u_R \in H^1(\bn), R \in (0,1)$ such that the following estimate holds:
 
 \begin{align*}
 \| \Delta_{\bn} u_{R} + \lambda u_R + |u_R|^{p-1} u_R \|_{L^{(p+1)'}} \lesssim
 \begin{cases}
\mathcal{D}|\log \mathcal{D}|^{-\frac{1}{2}}, \quad \mbox{if} \ p =2 \\
\medskip
\mathcal{D}^{1 + \alpha_0}, \qquad \quad \ \mbox{if} \ 1< p < 2.
 \end{cases}
\end{align*}
where $\mathcal{D}$ is the distance of $u_R$ from the manifold of sum of hyperbolic bubbles and $\alpha_0 >0$ is a constant depending on $n,\la$ and $p.$ When we let $R\uparrow 1,$ then $\mathcal{D} \rightarrow 0.$ This proves that the stability estimate of Theorem \ref{main theorem} fails whenever $1 < p \leq 2,$ in particular, in dimension $n \geq 6.$ Regarding Struwe decomposition in the hyperbolic space, we prove that the positive part of $u_R, u_R^+:= \max \{u_R,0\}$ also satisfies similar estimates and hence strengthening the counter-example for non-negative functions. The main result concerning the counterexample is stated in Section \ref{counter example section 1}, Theorem \ref{main counter example} and proved in subsequent sections.

\begin{remark}
{\rm 
We remark that the case $N=1$ has already been studied in \cite{BGKM} and the (linear stability) result holds in any dimension $n \geq 3$ and $p \in (1,\tstar-1].$ For $N \geq 2$,
note that in the statement of Corollary \ref{main corollary} for the critical case $p = \tstar -1$ we claim to prove the absence of Aubin-Talenti bubbles which significantly improves the profile decomposition result of \cite{BS} stated in Theorem \ref{profile decomposition}. However, we have to pay the price by excluding a countable set which is still, in our humble opinion, a nice result.  
Our study reveals that the stability of the hyperbolic space in low dimension $3 \leq n \leq 5$ bear resembles that of the Euclidean space even though the kernel of the linearized operator in the hyperbolic space is $n$-dimensional.  The stability result of \cite{Weietal}  in dimension $n\geq 6$ for the hyperbolic space is still unknown. However, we expect similar results should hold in the hyperbolic space as well which we plan to study in a forthcoming article.}
\end{remark}

\medskip

\subsection{Main hurdles, Novelty and Strategy of the proof} 

We shall now discuss the main difficulties that we encounter while proving Theorem~\ref{main theorem} and we briefly describe the new tools and the strategy of the proof. 

\begin{itemize}

\item As discussed above, while considering the multi-bubble scenario in the {\it critical case} we cannot ignore the existence of Aubin-Talenti bubbles. Hence we must take into account the interaction between the hyperbolic bubbles and the Aubin-Talenti bubbles, which is a  very challenging task. In this article as a first step, we look for the situation where the Aubin-Talenti bubbles are absent. To our surprise, and even more surprising the simplicity of the proof, we found that at energy level integer multiple of $\frac{1}{n}S_{\la,p}^{\frac{n}{2}}$, $p=\tstar-1$ and outside a countable set $\Lambda$ of $\lambda$'s there exist only hyperbolic bubbles. The set is characterized by 

\begin{align}
\Lambda = \left \{\la \in \left[\frac{n(n-2)}{4}, \frac{(n-1)^2}{4}\right] \ : \ \left(\frac{S_{\lambda,p}}{S}\right)^{\frac{n}{2}} \in [0,1] \cap \mathbb{Q}\right\}.
\end{align}

We prove in Section \ref{absence section} that $\Lambda$ is at most countable. As a result, we could make the situation slightly simpler by considering $\la$ in the complement of $\Lambda.$ Then we have to only consider the interaction between the \it hyperbolic bubbles \rm. However, working with hyperbolic bubbles poses many non-trivial and new difficulties which we try to explain briefly. 

\medskip

\item To our knowledge, the interaction integrals of hyperbolic bubbles are not known. As a first step, we obtained all the necessary and sharp interaction estimates among the bubbles. We obtain a series of new and novel estimates on the interactions of bubbles and their derivatives. We want to stress that these estimates are quite geometric in nature and very explicit, i.e., it depends largely on the geometry of the hyperbolic space.  Another novelty of this article is that we used explicit formulas of the hyperbolic distance and we were able to reduce the computations to a manageable integral evaluation. We found that all the interaction estimates are exponentially decaying in the hyperbolic distance between the points where the bubbles are most concentrated. More precisely, we prove that if $\calu[z_1], \calu[z_2]$ are two bubbles then 
\begin{align*}
\int_{\bn} \calu[z_1]^{\alpha}\calu[z_2]^{\beta} \ \dvg \approx
\begin{cases}
e^{-\cn \min\{\alpha, \beta\}d(z_1,z_2)} \ \ \ \ \ \ \ \ \ \ \ \ \ \ \ \mbox{if} \ \alpha \neq \beta,\\
d(z_1,z_2)e^{-\cn \min\{\alpha, \beta\}d(z_1,z_2)} \ \ \ \ \mbox{if} \ \alpha = \beta,
 \end{cases}
\end{align*}
whenever, $\alpha + \beta \geq 2,$ and where $d(z_1,z_2)$ is the hyperbolic distance between $z_1$ and $z_2$ and the constant in $\approx$ depends on $|\alpha-\beta|$ and  other natural parameters.

\medskip

\item A significant challenge comes while estimating the interaction between a bubble $ \calu[z_1]$ and the spacial derivative $V_j(\calu[z_2])$ of another bubble $ \calu[z_2]$ (see Appendix for the definition of $V_j(\calu[z_2])$) where $j$-denoted the direction (say) $e_j.$ Without loss of generality we may assume $z_1 = z$ and $z_2 = 0.$ It turns out that the interaction 
\begin{align*}
\int_{\bn} \calu[z]^p V_j(\calu[0]) \ \dvg
\end{align*}
is very sensitive to the position of the point $z.$  Some simple computations lead to the fact that if $z$ lies on the hyperplane $\{x | \ x \cdot e_j = 0\}$ then the interaction vanishes. More generally we can prove: if $z$ lies in the negative half $\{x | \ x \cdot e_j < 0\}$ then the interaction is $\approx_z e^{- \cn d(z,0)}$ and if $z$  lies in the positive half $\{x | \ x \cdot e_j > 0\}$ then the interaction is $\approx_z -e^{- \cn d(z,0)}.$ This is not a good situation to be in because constant $C$ in Theorem~\ref{main theorem} should depend only on the weak interaction strength $\delta$, not on the location of the bubbles on which we do not have any information whatsoever.
However, we manage to prove that if $z$ lies $\kappa >0$ distance away from the hyperplane i.e. $z \in \{x | \ |x \cdot e_j| \geq \kappa\}$ then the interaction is $\approx_{\kappa} \pm \ e^{- \cn d(z,0)},$ where the dependence of the constant $\approx_{\kappa}$ on $n,p,\la$ is implicitly assumed. As a result, the interaction of two bubbles and the interaction of a bubble and the derivative of another bubble has the same decay estimates up to a sign.

\medskip

\item The hyperbolic space does not admit any substitute of the conformal group of dilations as in the case of Euclidean geometry. This has both advantages and disadvantages. The advantage is that we have only bubble clusters and not bubble towers in the language of \cite{Weietal}. As a result, the computation of the interaction in the hyperbolic space turns out to be simpler than that of the Euclidean case. There is also a caveat while proving Theorem \ref{main theorem} as indicated in the next few bullet points.

\medskip

\item The proof of the Theorem \ref{main theorem} follows the same strategy as in \cite{FG}. Given $u$ we find a $\delta$-interacting family of hyperbolic bubbles $(\alpha_i, \calu_i)_{1 \leq i \leq N}$ by the minimization process 
$\inf_{\tilde\alpha_i \in \mathbb{R}, \tilde z_i \in \bn} \|u - \sum_{i=1}^N  \tilde\alpha_i\calu[\tilde z_i]\|_{\la}$ and we put $\rho = u - \sum_{i=1}^N \alpha_i\calu_i = u - \sigma.$ We test $\Delta_{\bn} u + \la u + |u|^{p-1}u$ against $\rho$ and performing some similar computations as in \cite[Theorem 3.3]{FG} the proof boils down to the justification of the following two inequalities:

\medskip

\begin{itemize}
\item[(i)] {\it Improved spectral inequality.}
Provided the weakly interaction $\delta$ is sufficiently small, it holds that
\begin{align*}
\int_{\bn}\sigma^{p-1}\rho^2\, \dvg\leq \ \frac{\tilde c}{p}  \|\rho\|^2_{\la}, \ \ \mbox{where} \quad \tilde c =  \tilde c(n, N, \lambda,p) <1.
\end{align*}
\item[(ii)] {\it Interaction integral estimates.}
Given $\eps>0$, if $\delta$ is sufficiently small then
\begin{align*}
\max_{i\neq k}\int_{\bn}\calu_i^p\calu_k \ \dvg \lesssim \|(\Delta_{\bn} + \lambda)u + |u|^{p-1}u\|_{H^{-1}} + \epsilon\|\rho\|_{\la}+\|\rho\|_{\la}^2.
\end{align*}
\end{itemize}
\medskip

\item The proof of (i) uses a localization argument as in \cite{FG}. The main difference is that the $L^n$-norm of the gradient is conformal and hence one can not expect the $L^n$-norm of the gradient of the localized function to be small. The substitute in the hyperbolic space is the $L^{\infty}$-norm of the gradient of the localized function. Thanks to the Poincar\'{e} inequality this poses no threat in our case.

\medskip

\item The proof of $(ii)$ in the Euclidean case relies heavily upon the interaction among the bubbles and the $\ mu$ derivatives of another bubble. Since $u = \sigma+\rho$ one can express $\Delta_{\bn} u + \la u + |u|^{p-1}u$ in the following manner
\begin{align*}
\Delta_{\bn} \rho + \lambda \rho + p\sigma^{p-1}\rho + I_1 + I_2 + I_3  = \Delta_{\bn} u + \la u + |u|^{p-1}u,
\end{align*}
where $I_1 = \sigma^p - \sum_{i=1}^N \alpha_i^p\calu_i^p$ contains all the interactions between the bubbles.
In the Euclidean case one test the above equation against $\partial_{\mu} U_k$ and derive that 
$\big|\int_{\rn} I_1\partial_{\mu} U_k \ dx \big|\approx \max_{i \neq k} \int_{\rn} U_i^pU_k \ dx.$ 

\medskip

The main reason for testing by $\partial_{\mu} U_k$ over testing by a spacial derivative (i.e. $\partial_{x_j}U_k$) is that spatial derivative leads to much weaker interaction estimates $<< \max_{i \neq k} \int_{\rn} U_i^pU_k$ in case the heights of the Aubin-Talenti bubbles are comparable i.e. if $U_i = U[z_i,\mu_i], 1 \leq i \leq N,$ then $\mu_i \approx \mu_j$ (see A. Bahri \cite[Estimate (F11)]{Bahri}). Note that the latter situation is precisely the case of the hyperbolic bubbles and the only possibility is to test by  $V_j(\calu_k)$. Thankfully our spacial derivative interaction almost behaves like the interaction of bubbles, but there is one caveat regarding the sign. There might be cancellation and $\int_{\bn} I_i V_j(\calu_k)$ might bring 
very weak interaction. 

\medskip 

\item We overcome this issue with the help of a geometric lemma \ref{geometric lemma}. The lemma suggests that if the family $(\calu_i)_{1 \leq i \leq N}$ sufficiently small and weakly interacting, then there exists a direction $e_j$ such that up to hyperbolic translations, orthogonal transformations and rearrangement of indices  $z_N = 0$ and all other $z_i$ lies in the negative half of the plane $\{x | \ x \cdot e_j < 0\}.$ Moreover, all $z_i, i \neq N$ lie 
at a positive distance away from the plane $\{x | \ x \cdot e_j = 0\}$. The proof of the lemma is purely geometric and makes the most use of the properties of the reflection in plane and inversion in spheres.

\medskip

\item In the next step we look for the counter-example of linear stability for $1 < p \leq 2.$ We proceed as in \cite{FG}. Let us denote $U = \calu[-Re_1], V = \calu[Re_1],$  where $R \in (0,1)$ and put $u_R = U + V + \rho.$
Then we can estimate
\begin{align*}
\Delta_{\bn} u_R + \la u_R + |u_R|^{p-1}u_R  = \Delta_{\bn} \rho + \lambda \rho + p (U+V)^{p-1}\rho + f + O(|\rho|^p)
\end{align*}
where $f = (U+V)^p - U^p - V^p.$ Of course, solving $\Delta_{\bn} \rho + \lambda \rho + p (U+V)^{p-1}\rho + f = 0$ and 
satisfying the orthogonality condition is not possible and therefore we have to look for a perturbed problem 
\begin{align*}
\Delta_{\bn} \rho + \lambda \rho + p (U+V)^{p-1}\rho + \tilde f = 0,
\end{align*}
where $\tilde f$ is such that $f - \tilde f$ belongs to the finite-dimensional space spanned by the eigenvectors 
of the operator $(-\Delta - \la)/(U+V)^{p-1}.$ As a result by Poincar\'e-Sobolev inequality we get
\begin{align*}
\|\Delta_{\bn} u_R + \la u_R + |u_R|^{p-1}u_R\|_{L^{\frac{p+1}{p}}} \lesssim \|f - \tilde f\|_{L^{\frac{p+1}{p}}} + \|\rho\|_{\la}^p,
\end{align*}
where $\|\cdot\|_{\la}$ is an equivalent $H^1$-norm defined in Section \ref{basic section}.
Moreover we can estimate $\|\rho\|_{\la} \approx \|\tilde f\|_{H^{-1}}$ and $\|f - \tilde f\|_{L^{\frac{p+1}{p}}} \lesssim \sum_{\Psi} |\int_{\bn} f \Psi \dvg|,$ where $\Psi$ solves
\begin{align}\label{egeintro}
-\Delta_{\bn} \Psi- \la\Psi = \mu (U+V)^{p-1}\Psi.
\end{align}

In \cite{FG}, the authors showed that $\|f\|_{L^2} = o(\|f\|_{H^{-1}})$ and hence one obtain $\|\Delta_{\bn} u_R + \la u_R + |u_R|^{p-1}u_R\|_{L^{\frac{p+1}{p}}} = o(\|\rho\|_{\la}).$ Then they proceed to show that $\|\rho\|_{\la}$ approximates the distance of $u_R$ from the manifold of the sum of two Aubin-Talenti bubbles and concluding the construction of counter-example in dimension $n \geq 6.$

\medskip

\item In the hyperbolic space, thanks to Poincar\'e inequality $\|f\|_{H^{-1}} \lesssim \|f\|_{L^2}$ and hence we need a major diversion from the line of reasoning from that of the Euclidean case. We achieve this in two steps:
\begin{itemize}
\item[(a)] A delicate eigenfunction estimates of \eqref{egeintro}.  We can show that for sufficiently small $\alpha_0$ the uniform bound
\begin{align*}
\int_{\bn} \frac{\Psi^2}{(U + V)^{\alpha_0}} \ \dvg \lesssim_{\mu, \alpha_0} 1,
\end{align*}
holds.

\item[(b)] And an improvement of the integral interaction estimates in half ball. Denote $\bn_+ := \{x\in \bn \ | \ x_1 >0\}$ then
\begin{align*}
\int_{\bn_+} U^{\alpha} V^{\beta(p-1)} \approx e^{-c(n,\la) \frac{\alpha + \beta(p-1)}{2} d(-Re_1, Re_1)}, \ \ 
\mbox{if} \ \alpha > \beta(p-1).
\end{align*}
\end{itemize}
With the help of these two results and a well-derived lower bound on $\|f\|_{H^{-1}}$, we conclude that 
\begin{align*}
\left|\int_{\bn_+}f \Psi \ \dvg\right| = o(\|f\|_{H^{-1}}).
\end{align*}
Once these estimates are achieved one can follow the ideas of Figalli and Glaudo to complete the proof. However, we emphasize that several technical difficulties arise in the hyperbolic space which need to be addressed separately. We derived these technical results either in the appendix or some of them whenever they arise.
\end{itemize}
\medskip

The organization of the article is as follows:

\medskip

\begin{itemize}
\item[Section 1:] Introduction section, contains a brief background on the stability problem in the Euclidean space, known results in the hyperbolic space, the main result of this article and a short subsection about major challenges and how we addressed it.   
\item[Section 2:] Contains the basics of the hyperbolic space. 
\item[Section 3:] This section is devoted to the improvement of the Struwe-type profile decomposition on the hyperbolic space. In particular, we obtained instances when Aubin-Talenti bubbles are absent in the profile decomposition.
\item[Section 4:] This is one of the core sections of this article. In this section, we proved all the interaction estimates namely (a) interaction between two bubbles, (b) interaction among three bubbles, and (c) interaction between a bubble and the spacial derivation of another bubble. This section is slightly technical but all the results obtained, to the best of the knowledge of the authors, are all new.
 \item[Section 5:] Contains the proof of the main theorem, namely, Theorem~\ref{main theorem} and Corollary~\ref{main corollary}
assuming (i) improved spectral inequality which is proved in Section 7 and (ii) Interaction integral estimates, 
which is proved in Section 8.
\item[Section 6:] Another slightly technical section containing the core of the localization argument.
\item[Section 9:] Set up and technical results concerning the counter-example in the case $1 < p \leq 2.$ The main result is Theorem \ref{main counter example}.
\item[Section 10:] Contains the proof of Theorem \ref{main counter example}.
\item [Appendix:] We recalled some technical, some known and some probably unknown, results needed for the article.
\end{itemize}

\medskip

\subsection{Notations}
\begin{itemize}
\item A point $x \in \rn$ will be denoted by $x = (x_1,\ldots, x_n).$ For a point $x \in \rn,$ $x_j$ will denote the $j$-th component of $x.$

\item $n=$ dimension, $N = $ number of bubbles, $\delta = $ interaction strength, $p+1 =$ exponent in Sobolev inequality and $\la=$ spectral parameter. 
\item The notation $ T \lesssim S$ (respectively $T \gtrsim S$) means $T \leq CS$ (respectively $T \geq CS$) for some constant $C$ that depends only on the given data
$n,N,p,\lambda.$  We denote $T \approx S$ if both $T \lesssim S$ and $T \gtrsim S$ hold.

\item The notation $T \lesssim_{p,q,r} S$ means $T \leq CS$ for some constant $C$ that depends on the parameters $p,q,r$ as well as on $n,N,\lambda,p.$

\item $x \cdot y = \sum_{j=1}^n x_jy_j$ denotes the standard dot product in $\rn$ and $| \cdot |$ the associated norm.

\item $\{e_j\}_{1 \leq j \leq n}$ denotes the standard basis of $\rn,$ i.e $l$-th component of $e_j$ is $\delta_{jl}.$
\item We denote by $B_E (a,r)$ the Euclidean ball $\{x \in \rn \ | \ |x-a| < r \}$ and $S(a,r)$ the Euclidean sphere $\{x \in \rn \ | \ |x-a| = r \}.$ The hyperbolic ball will be denoted by $B(a,r)$ which is defined in the next section.  

\item The notation $\|\cdot\|_{L^r}$ will denote the $L^r$-norm with respect to the measure $\dvg.$

\item Given dimension $n\geq 3$, $S$ denotes the best Sobolev constant in $\rn$ and given $\la,p$ as in \eqref{S-inq}, $S_{\la,\, p}$ denotes the best Poincar\'e-Sobolev inequality in $\bn.$ In the later case if $p = \tstar -1$ then we denote $S_{\la,p}$ by $S_{\la}.$

\end{itemize}

\section{Basics on the hyperbolic space} \label{basic section}

The Euclidean unit ball $\Be^n:= \{x \in \mathbb{R}^n: |x|^2<1\}$ equipped with the Riemannian metric
\begin{align*}
{\rm d}s^2 = \left(\frac{2}{1-|x|^2}\right)^2 \, {\rm d}x^2
\end{align*}
constitute the ball model for the hyperbolic $n$-space, where ${\rm d}x$ is the standard Euclidean metric and $|x|^2 = \sum_{i=1}^nx_i^2$ is the standard Euclidean length. By definition, the hyperbolic $n$-space is a $n$-dimensional complete, non-compact Riemannian manifold having constant sectional curvature equal to $-1$ and any two manifolds sharing the above properties are isometric \cite{RAT}. In this article, all our computations will involve only the ball model and will be denoted by $\bn$. We denote the inner product on the tangent space of $\mathbb{B}^n$ by $\langle \cdot, \cdot \rangle_{\bn}$ and
the volume element is given by $\dvg = \left(\frac{2}{1 - |x|^2}\right)^n {\rm d}x,$ where ${\rm d}x$ denotes the 
Lebesgue measure on $\mathbb{R}^n.$

Let $\nabla_{\bn}$ and $\Delta_{\bn}$ denote gradient vector field
and Laplace-Beltrami operator respectively.  Therefore, in terms of local (global) coordinates $\nabla_{\bn}$ and $\Delta_{\bn}$ takes the form:
\begin{align*} 
 \nabla_{\bn} = \left(\frac{1 - |x|^2}{2}\right)^2\nabla,  \quad 
 \Delta_{\bn} = \left(\frac{1 - |x|^2}{2}\right)^2 \Delta + (n - 2)\left(\frac{1 - |x|^2}{2}\right)  x \cdot \nabla,
\end{align*}
where $\nabla, \Delta$ is the standard Euclidean gradient vector field and Laplace operator respectively, and `$\cdot$' denotes the 
standard inner product in $\mathbb{R}^n.$

\medskip 

\noindent
 {\bf Hyperbolic distance on $\bn.$} The hyperbolic distance between two points $x$ and $y$ in $\bn$ will be denoted by $d(x, y).$ The hyperbolic distance between
$x$ and the origin can be computed explicitly  
\begin{align*}
\rho := \, d(x, 0) = \int_{0}^{|x|} \frac{2}{1 - s^2} \, {\rm d}s \, = \, \log \frac{1 + |x|}{1 - |x|},
\end{align*}
and therefore  $|x| = \tanh \frac{\rho}{2}.$ Moreover, the hyperbolic distance between $x, y \in \bn$ is given by 
\begin{align*}
d(x, y) = \cosh^{-1} \left( 1 + \dfrac{2|x - y|^2}{(1 - |x|^2)(1 - |y|^2)} \right).
\end{align*}

In this article, we will use several simplifications of the above formula which can be derived easily. For example, the following two formulas will be used:

\begin{align*}
\cosh \left(\frac{d(x,y)}{2} \right)= \frac{\sqrt{1-2 x \cdot y + |x|^2|y|^2}}{\sqrt{(1-|x|^2)(1-|y|^2)}} = 
\frac{|y||x-y^*|}{\sqrt{(1-|x|^2)(1-|y|^2)}}.
\end{align*}
where $y^* = y/|y|^2.$ Alternatively, using the formula  $\cosh^{-1}(s) = \ln (s + \sqrt{s^2-1}),$
$s>1$ one can deduce 
\begin{align*}
\frac{d(x,y)}{2} = \ln \left(\frac{|y||x-y^*| + |x-y|}{\sqrt{(1-|x|^2)(1-|y|^2)}}\right).
\end{align*}

Let $A: \rn \to \rn $ be an orthogonal transformation i.e., $Ax \cdot Ay = x \cdot y$ for all $x,y \in \bn.$ Then  it follows from the formula of hyperbolic distance that $d(Ax, Ay) = d(x,y)$ for all $x,y \in \bn.$ In other words, the hyperbolic distance is preserved under the action of the orthogonal group of $\rn.$


Geodesic balls in $\bn$ of radius $r$ centred at $x \in \bn$ will be denoted by 
$$
B(x,r) : = \{ y \in \bn : d(x, y) < r \}.
$$

Next, we introduce the concept of hyperbolic translation.
\subsection{Hyperbolic translation} We first recall some basic notions of reflections and inversions. 

\subsubsection{Reflections about a plane} Let $a$ be a non-zero vector in $\rn$ and  $t$ be a real number. Consider the hyperplane in $\rn$ 
\begin{align*}
P(a,t) = \{x \in \rn \ | \ \frac{a}{|a|} \cdot x = t \}.
\end{align*}
The reflection $\rho_{a,t}$ of $\rn$ about the plane $P(a,t)$ is defined by 
\begin{align*}
\rho_{a,t}(x) = x + 2\left(t - \frac{a}{|a|} \cdot x\right)\frac{a}{|a|}.
\end{align*}
When $t = 0$ we will simply denote the reflection by $\rho_a.$ It is easy to verify that $\rho_a$ sends the positive half plane $\{x \ | \ a \cdot x > 0\}$ to the negative half $\{ x\ | \ a \cdot x < 0\}$ and vice versa, and leaves the plane $P_{a,0}$ invariant.

\subsubsection{Inversion about a sphere} Let $a$ be a point in $\rn$ and let $r$ be a positive real number. The sphere of $\rn$ of radius $r$ centered at $a$ is defined by
\begin{align*}
S(a,r) = \{ x \in \rn \ | \ |x-a| = r\}.
\end{align*}

The inversion $\sigma_{a,r}$ of $\rn$ in the sphere $S(a,r)$ is defined by the formula 
\begin{align*}
\sigma_{a,r} (x) = a + \left(\frac{r}{|x-a|}\right)^2(x-a).
\end{align*}
It is easy to verify that 
\begin{align*}
|\sigma_{a,r}(x) - a||x-a| = r^2, \ \ \ \ \mbox{for all} \ x \in \rn.
\end{align*}
Hence $\sigma_{a,r}$ sends sphere $S(a,r_1)$ to the sphere $S(a,\frac{r^2}{r_1})$ for every $r_1>0.$ When $r$
is determined by $a$ (for example $r^2 = |a|^2-1$) we will simply use the notation $\sigma_a$ instead of $\sigma_{a,r}.$

\subsubsection{The hyperbolic translation}
Given a point $a \in \rn$ such that $|a|>1$  and $r>0,$ let $S(a,r):=\{x \in \rn \ | \ |x-a| = r\}$ be the sphere in $\rn$ with center $a$ and radius $r$ that intersects $S(0,1)$ orthogonally. It is known that it is the case if and only if $r^2 = |a|^2 -1,$ and hence $r$ is determined by $a.$ Let $\rho_{a}$ denotes the reflection with respect to the plane $P_a := \{x \in \rn \ | \ x\cdot a = 0\}$ and $\sigma_a$ denotes the inversion with respect to the sphere $S(a,r).$ Then $\sigma_a\rho_a$
leaves $\Be^n$ invariant (see \cite{RAT}).
 
 For $b \in \mathbb{B}^n,$ the hyperbolic translation $\tau_{b}: \mathbb{B}^n \rightarrow \mathbb{B}^n$ that takes $0$ to $b$ is defined by $\tau_b = \sigma_{b^*}\rho_{b^*}$ and can be expressed by the following formula
 
  \begin{align} \label{hyperbolictranslation}
  \tau_{b}(x) := \frac{(1 - |b|^2)x + (|x|^2 + 2 x \cdot b + 1)b}{|b|^2|x|^2 + 2 x\cdot b  + 1}
 \end{align}
where $b^* = \frac{b}{|b|^2}.$
 It turns out that $\tau_{b}$ is an isometry and forms the M\"obius group of $\Be^n$ (see \cite{RAT}, Theorem 4.4.6) for details and further discussions on isometries. Note that $\tau_{-b} = \sigma_{-b^*}\rho_{-b^*}$ is the hyperbolic translation that takes $b$ to the origin. In other words, the hyperbolic translation that takes $b$ to the origin is the composition of the reflection $\rho_{-b^*}$ and the inversion $\sigma_{-b^*}.$
  
For the convenience of the reader, we list several well-known properties of the hyperbolic translation in the next lemma. The proof follows by now standard properties of the M\"{o}bius transformations on the ball model which we skip for brevity and refer to the book \cite{Stoll}. 

\begin{lemma}\label{lemma1}
  For $b \in \B^n,$ let $\tau_b$ be the hyperbolic translation of $\mathbb{B}^n$ by $b.$ Then for every $u \in C_c^{\infty}(\mathbb{B}^n),$ there holds,

\begin{itemize}
  
\item[(i)] $\Delta_{\bn} (u \circ \tau_b) = (\Delta_{\bn} u) \circ \tau_b, \ \ \langle \nabla_{\bn} (u \circ \tau_b),
   \nabla_{\bn} (u \circ \tau_b)\rangle_{\bn} = \langle(\nabla_{\bn} u) \circ \tau_b, (\nabla_{\bn} u) \circ \tau_b\rangle_{\bn}.$
  
  \item[(ii)] For every open subset $U$ of $\mathbb{B}^n$
  \begin{align*}
  \int_{U} |u \circ \tau_b|^p \, \dvg = \int_{\tau_b(U)} |u|^p \, \dvg , \ \mbox{for all} \ 1 \leq p < \infty.
 \end{align*}
 
 \item[(iii)] For every $\phi, \psi \in C_c^{\infty}(\B^n)$,
 \begin{align*}
 \int_{\B^n} \phi(x) (\psi \circ \tau_b)(x) \, \dvg = \int_{\B^n} (\phi \circ \tau_{-b})(x) \psi(x) \, \dvg. 
 \end{align*}
\end{itemize}
\end{lemma}
By density, the above formulas hold as long as the integrals involved are finite.

\subsection{The Sobolev space $H^1(\mathbb{B}^n)$:} Thanks to the Poincar\'{e} inequality, we define the Sobolev space on $\bn$, denoted by $H^1(\mathbb{B}^n)$, as the completion of $C_c^\infty(\mathbb{B}^n)$ with respect to the norm
\begin{align*}
 \|u\|_{H^1(\mathbb{B}^n)} :=  \left(\int_{\mathbb{B}^n} 
 |\nabla_{\bn} u|^2  \, \dvg \right)^{\frac{1}{2}},
\end{align*}
where  $|\nabla_{\bn} u| $ is given by
 \begin{align*}
  |\nabla_{\bn} u| := \langle \nabla_{\bn} u, \nabla_{\bn} u \rangle^{\frac{1}{2}}_{\bn}.
\end{align*}
Using Poincar\'e inequality once again, $\|u\|_{\lambda}$ defined by 
\begin{align*}
\|u\|_{\lambda}^2 : = \int \limits_{\B^{n}}\left(|\nabla_{\bn} u|^{2}-\lambda u^{2}\right) \dvg~\hbox{ for all } u\in H(\B^{n}),
\end{align*}
is an equivalent $H^1(\mathbb{B}^n)$ norm for $\lambda < \frac{(n-1)^2}{4}.$ However, for $\lambda = \frac{(n-1)^2}{4},$ $\|u\|_{\lambda}$ is not an equivalent norm because, in general, $u$ may fail to be square integrable. Moreover, thanks to Lemma \ref{lemma1}, the norms $\|u\|_{H^1(\bn)}, \|u\|_{\lambda}$ are invariant under the action of isometries of the ball model. To distinguish between the norms we will call $\|\cdot\|_{\la}$ by $H^1_{\la}$-norm. Similarly, We will denote the
$H^1_{\la}$-inner product by $\langle \cdot, \cdot \rangle_{\la} : = \langle \cdot, \cdot \rangle_{H^1} - \la \langle \cdot, \cdot \rangle_{L^2}$  which generates the norm $\| \cdot \|_{\la}.$

\medskip 

We also define the subset of $H^1(\bn)$ consisting of radial functions
\begin{align*}
H^1_{\mbox{\tiny{rad}}}(\bn) := \{u \ \mbox{radial} \ | \ \int_{0}^{\infty} \left[u^2(\rho) + (u^{\prime}(\rho))^2\right] (\sinh \rho)^{n-1} \ d\rho < \infty \},
\end{align*}
where $\prime$ denotes the differentiation with respect $\rho.$

\medskip

We will also be needing the space $H^1_0(\Omega)$ for $\Omega \subset \bn$ which is defined by the closure of $C_c^{\infty}(\Omega)$ in $H^1(\bn).$

For a positive weight $w$ we define the weighted $L^2$-space by 
\begin{align*}
L^2_{w}(\bn):= \Big\{u \ | \ \int_{\bn} |u|^2w \ \dvg < \infty \Big\}.
\end{align*}

\section{On the absence of Aubin-Talenti bubbles} \label{absence section}

We investigate, in the critical case $p = \tstar -1$, under what conditions only the hyperbolic bubbles are present in Struwe-type profile decomposition. Because of the non-existence results of Sandeep and Mancini \cite{MS} in low dimensions and low spectrum regions,  in this section, we assume $n $ and $\lambda$ satisfy

\begin{align*}
\mbox{{(\bf H2)}} \ \ \ \ \ \ \ \ \ \ \ \ \ \ \ \ \ \ \
\ n \geq 4 \ \ \ \ \mbox{and} \ \ \frac{n(n-2)}{4} < \lambda < \frac{(n-1)^2}{4}.  \ \ \ \ \  \ \ \ \ \ \ \ \ \ \ \ \ 
\end{align*}

Since $p = \tstar-1$ is fixed in this section, for the simplicity of notations we denote $S_{p,\la}$ by $S_{\la},$ and the unique positive radial solution of \eqref{eq1} will be denoted by $\calu.$ To begin with we recall the uniqueness result of \cite{MS} in its simplest form.

\begin{theorema}[Theorem $1.3$ of \cite{MS}]\label{MS_uniq}
Assume {\bf (H2)}, then all the positive solutions of \eqref{eq1} are of the form $\calu\circ \tau_{b}$ for some $b \in \bn.$
\end{theorema}

Also recall that all the positive solutions to \eqref{eq1} have the same energy $I_{\lambda}(\calu \circ \tau_b) = \frac{1}{n}S_{\lambda}^{\frac{n}{2}}.$

\medskip

Below are a few applications of the above existence and uniqueness of solutions to \eqref{eq1}. 

\medskip

Define 
$\mathbcal{h} : [\frac{n(n-2)}{4}, \frac{(n-1)^2}{4}] \to (0,\infty)$ by $\mathbcal{h}(\lambda) = S_{\lambda}.$ Clearly $\mathbcal{h}$ is monotonically decreasing. We have

\begin{lemma} \label{strictly decreasing}
The function $\mathbcal{h}$ is strictly decreasing. 
\end{lemma}

\begin{proof}
It is well known that $S_{\frac{n(n-2)}{4}} = S$ where recall $S$ is the best Euclidean Sobolev constant and $S_{\lambda}< S$ for $\la > \frac{n(n-2)}{4}$ (see \cite{TT}). So, with out loss of generality we assume $\frac{n(n-2)}{4} < \la_1 < \la_2 \leq \frac{(n-1)^2}{4}$ and we claim to show that $S_{\lambda_1}> S_{\lambda_2}.$ 
On the contrary  we assume $S_{\lambda_1} = S_{\lambda_2}.$ Then $S_{\lambda_1}(\B^n)$ is also attained by $\calu_{\la_2}$ and hence by uniqueness Theorem \ref{MS_uniq}, $\calu_{\la_1} = C\calu_{\la_2}$ for some positive constant $C.$ Because of the assumed normalization \eqref{calu} and $S_{\lambda_1} = S_{\lambda_2},$ we conclude $C = 1.$ But $\calu_{\la_i}$ satisfies the Euler-Lagrange \eqref{eq1} with $\la = \la_i, i = 1,2.$ Subtracting the two equations satisfied by $\calu_{\la_i}$ we see that $(\la_1 - \la_2)\calu_{\la_1} = 0$ which implies $\calu_{\la_1} \equiv 0,$ a contradiction. Hence the proof follows.
\end{proof}

A direct application of Lemma \ref{strictly decreasing} is the following corollary. Set 

\begin{align}
\Lambda = \left \{\la \in \left[\frac{n(n-2)}{4}, \frac{(n-1)^2}{4}\right] \ : \ \left(\frac{S_{\lambda}}{S}\right)^{\frac{n}{2}} \in [0,1] \cap \mathbb{Q}\right\}.
\end{align}

\begin{corollary}\label{countable}
The set $\Lambda$ is at most countable.
\end{corollary}

\begin{proof}
By Lemma \ref{strictly decreasing}, the function $g(\la) = \left(\frac{S_{\lambda}}{S}\right)^{\frac{n}{2}} = 
\left(\frac{\mathbcal{h}(\la)}{S}\right)^{\frac{n}{2}}$
is strictly decreasing and hence there exists a one-to-one correspondence between $\left[\frac{n(n-2)}{4}, \frac{(n-1)^2}{4}\right]$ and the image of $g$ denoted by $Im(g).$ As a result, we can write
\begin{align*}
\Lambda = g^{-1}(\mathbb{Q} \cap Im(g))
\end{align*}
and since $\mathbb{Q} \cap Im(g)$ is countable and  $g$ is strictly decreasing $\Lambda$ must be countable.
\end{proof}

\medskip

\noindent
{\bf On the absence of Aubin-Talenti bubbles in the profile decomposition:} Recall that under the hypothesis {\bf (H2)} the profile decomposition may contain both Aubin-Talenti bubbles and the hyperbolic bubbles. The following simple consequence of 
Lemma~\ref{countable} and Theorem \ref{profile decomposition} isolate the situation preferable for us. 

\begin{theorem} \label{absence theorem}
Let $n \geq 4, p = \tstar-1$ and $\frac{n(n-2)}{4} < \la < \frac{(n-1)^2}{4}$ be such that $\la \not \in \Lambda$ and $N \in \mathbb{N}$ be given. 
Let $\{u_m\} \subset H^1(\bn)$ be a sequence such that $u_m \geq 0,$ $I_{\la}(u_m) \rightarrow d = \frac{N}{n} S_{\la}^{\frac{n}{2}}$ and 
\begin{align*}
\|\Delta_{\bn} u_m+ \la u_m + u_m^p\|_{H^{-1}} \rightarrow 0, \ \ \mbox{as} \ \ m \rightarrow \infty.
\end{align*}

 Then there exists $N$ sequences $\{b_m^j\} \subset \bn, 1 \leq j \leq N$ such that 
 \begin{align*}
u_m = \sum_{j=1}^N\calu \circ \tau_{b_m^j} + o(1) \ \ in \ H^1(\bn).
\end{align*} 
In other words, no Aubin-Talenti bubble exists in the profile decomposition of $u_m.$
\end{theorem}

\begin{proof}
By Theorem \ref{profile decomposition} and the subsequent Remark \ref{rem2}, we have 
\begin{align*}
\frac{N}{n} S_{\la}^{\frac{n}{2}}  = \frac{\tilde N_1}{n}S_{\la}^{\frac{n}{2}} + \frac{N_2}{n}S^{\frac{n}{2}}.
\end{align*}
If $N_2 \neq 0,$ (which corresponds to the existence of Aubin-Talenti bubbles) then $N \neq \tilde N_1$ and hence \begin{align*}
\left(\frac{S_{\lambda}}{S}\right)^{\frac{n}{2}} = \frac{N_2}{N - \tilde N_1} \in \mathbb{Q}.
\end{align*}
This contradicts the assumption that $\la \not \in \Lambda$ and completes the proof of the theorem.
\end{proof}

At this point, we are tempted to conjecture that Theorem \ref{absence theorem} is the best possible scenario we can get for $\frac{n(n-2)}{4} < \la < \frac{(n-1)^2}{4}.$ In other words, if $\mathbcal{h}$ is continuous then $Im(g)$ must contain both rational and irrational numbers and thence we can not expect to do better than Theorem \ref{absence theorem}. However, proving rigorously the existence of both Aubin-Talenti bubbles and hyperbolic bubbles at the prescribed energy level mentioned in Theorem \ref{absence theorem} is a formidable challenge. We don't have any positive or negative answer to it yet.

\section{Interaction of bubbles}\label{interaction of bubbles section}

\subsection{Interaction of two bubbles}\label{interaction of bubbles section (a)}
 In this subsection, we shall compute the interaction of two hyperbolic bubbles. We use the following sharp bounds on the hyperbolic bubble: there exist constants $\chi_i = \chi_i(n,p,\la),\,  i = 1,2$ such that 
\begin{align} \label{bubble decay}
\chi_1e^{-\cn d(x,z)} \leq \calu[z](x) \leq \chi_2 e^{- \cn d(x,z)}, \ \ \ \mbox{for all} \ \ x \in \bn,
\end{align}
and for every $z \in \bn$, where $\cn$ is as defined in \eqref{cn-lambda}. The case $\la = 0$ and $p < \tstar -1$ the decay estimate \eqref{bubble decay} have been obtained by Bonforte-Gazzola-Grillo and V\'azquez \cite{BGGV} some time ago which is a slight improvement of the seminal result of Mancini-Sandeep \cite{MS} concerning the asymptotic behaviour of radial solutions. 
We came to know about the validity of \eqref{bubble decay} for all admissible $\la$ from personal communication with Sandeep K and Ramya Dutta \cite{SRa}.  We remark that this article was completed only after knowing the results of Sandeep and Ramya, and we thank them for explaining and sharing the estimates to us. The result is going to appear in their forthcoming article. 

\begin{lemma}\label{lem-int-2}
Given $n \geq 3,$ let $\alpha, \beta \geq 0$ be such that $\alpha + \beta \geq 2.$ 
Let $\calu_i = \calu[z_i], i = 1,2$ be two hyperbolic bubbles such that $d(z_i, z_j) \geq \ln 3.$ Then 
\begin{align*}
\int_{\bn} \calu_1^{\alpha} \calu_2^{\beta}\ \dvg  \ 
\begin{cases}
\approx_{|\alpha-\beta|} \ e^{-\cn \min\{\alpha, \beta \} \, d(z_1, z_2)}, \ \ \ \ \ \ \ \ \ \  \ \ \ \mbox{if} \ \alpha \neq \beta \\
\medskip
\approx  \ \ \ \ \ \  d(z_1, z_2)e^{-\cn \beta  \, d(z_1, z_2)}, \ \ \ \ \ \ \ \ \ \ \ \mbox{if} \ \alpha = \beta. \\
\end{cases}
\end{align*}
Here we follow the convention that the constant in $\approx$ depends also on $n,p$ and $\la.$ 
\end{lemma}

\medskip

\begin{proof}

Before beginning the proof let us recall that for any two points $x, z \in \bn$ we have 
\begin{align*}
\cosh \frac{d(x, z)}{2} \, = \, \dfrac{\sqrt{1 - 2x.z + |x|^2 |z|^2}}{\sqrt{(1 - |z|^2) (1 - |x|^2)}}.
\end{align*}
So in particular for $z =0$ we have 
\begin{align*}
e^{d(x, 0)} \, \approx \, \dfrac{1}{(1 - |x|^2)}.
\end{align*}

By Lemma \ref{lemma1}(iii) it is enough to prove the result for $z_1=0, z_2= z.$ 

\medskip

Let us compute the following integral for $\alpha, \beta\in  \mathbb{R}^{+} \cup \{ 0 \}$ such that $\alpha + \beta \geq 2:$

\begin{align}\label{local-estimate}
I \, &= \, \int_{\bn} \calu[0]^{\alpha} \, \calu[z]^{\beta} \, \dvg \notag\\
& \approx \int_{\bn} e^{-\cn \alpha d(x, 0)} e^{-\cn \beta d(x, z)} \, \dvg  \notag\\
& \approx \int_{\Be^n} (1 - |x|^2)^{\cn \alpha} \left( \frac{(1 - |z|^2) (1 - |x|^2)}{1 - 2x.z + |x|^2 |z|^2} \right)^{\cn \beta} \, 
\left( \frac{2}{1 - |x|^2} \right)^n \, {\rm d}x \notag \\
& \approx (1 - |z|^2)^{\cn \beta} \int_{\Be^n} \frac{(1 - |x|^2)^{\cn (\alpha + \beta) - n}}{(1 - 2 x.z + |x|^2 |z|^2)^{\cn \beta}} \, {\rm d}x  \notag \\
& =  \frac{(1 - |z|^2)^{\cn \beta}}{|z|^{2 \cn \beta}} \int_{\Be^n} \frac{(1 - |x|^2)^{\cn (\alpha + \beta) - n}}{|x - z^*|^{2 \cn \beta}} \,  \dx.
\end{align}

Note that the the values of $\cn > \frac{(n-1)}{2}$ for all $\la < \frac{(n-1)^2}{4}$ and hence the integral \eqref{local-estimate} makes sense only if $\alpha + \beta \geq 2.$
Now without loss of generality, we may assume that $\alpha \geq \beta.$ We further subdivide our proof into two cases.  

\medskip 

{\bf Case-1: $\alpha > \beta.$} We notice that  $|z^*| \geq 1$ and $|x| < 1$ and therefore 

\begin{align*}
|x - z^*| \geq |z^*| - |x| \geq 1 - |x| \approx (1 - |x|^2).
\end{align*}
Since by our hypothesis $d(0,z) \geq \ln 3,$ we have $|z| \geq \frac{1}{2}$ and hence $\Be^n \subset B_E(z^*,3).$ Exploiting the fact that $\cn (\alpha - \beta) > 0$ we get

\begin{align}
\int_{\Be^n} \frac{(1 - |x|^2)^{\cn (\alpha + \beta) - n}}{|x - z^*|^{2 \cn \beta}} \,  
\notag \dx  &\lesssim \int_{\Be^n} |x - z^*|^{\cn (\alpha - \beta) - n} \, \dx \\
& \lesssim \int_{B_E(z^*,3)} |x - z^*|^{\cn (\alpha - \beta) - n} \, \dx  \lesssim_{|\alpha-\beta|} 1.
\end{align}
This proves the upper bound. 

\medskip 

For the lower bound, first note that for fixed $z_0$ with $\frac{1}{2} \leq |z_0| \leq 1$ the same argument leads to
\begin{align*}
0 < \int_{\Be^n} \frac{(1 - |x|^2)^{\cn (\alpha + \beta) - n}}{|x - z_0^*|^{2 \cn \beta}} \, \dx \lesssim_{|\alpha-\beta|} 1.
\end{align*}
Taking a minimizing sequence and using Fatou's lemma one can easily conclude that 
\begin{align*}
\inf_{\frac{1}{2} \leq |z| \leq 1} \int_{\Be^n} \frac{(1 - |x|^2)^{\cn (\alpha + \beta) - n}}{|x - z^*|^{2 \cn \beta}} \, \dx \approx_{|\alpha-\beta|} 1.
\end{align*}
This completes the proof lower bound. Hence combining the upper bound and lower bound estimates we conclude that 
\begin{align*}
\int_{\bn} \calu[0]^{\alpha} \, \calu[z]^{\beta} \, \dvg \approx_{|\alpha-\beta|} (1 - |z|^2)^{\cn \beta} \approx_{|\alpha-\beta|}  e^{-\cn \mbox{min}\{\alpha, \beta \} d(0, z)}.
\end{align*}

\medskip 

{\bf Case-2: $\alpha = \beta.$} {\it The upper estimate.}

\medskip

In this case, we have
\begin{align*}
\int_{\Be^n} \frac{(1 - |x|^2)^{\cn (\alpha + \beta) - n}}{|x - z^*|^{2 \cn \beta}} \, \dx \lesssim \int_{\Be^n} \frac{1}{|x - z^*|^n} \, \dx
\end{align*}
In the proof all constants depend on $n, \la$ and $p.$ With out loss of generality we may assume that $z^* = |z^*| e_n = t e_n,$ where $t = \frac{1}{|z|}.$ Now we transformation the integral in polar co-ordinates 
\begin{align*}
&x_1 = \varrho \, \sin \theta \sin \theta_1\ldots \sin \theta_{n-3}\sin \theta_{n-2} \\
& x_2 = \varrho \, \sin \theta \sin \theta_1 \ldots \sin  \theta_{n-3} \cos \theta_{n-2} \\
& x_3 = \varrho \, \sin \theta \sin \theta_1 \ldots \cos \theta_{n-3} \\
& \ldots \\
& x_n = \varrho \cos \theta.
\end{align*}
The Jacobian of the above transformation is $\rho^{n-1}\sin^{n-2}\theta \sin^{n-3}\theta_1\cdots \sin \theta_{n-3},$ 
where $\rho \in [0,1],$ \ $\theta, \theta_1, \ldots, \theta_{n-3} \in [0,\pi]$ and $ \theta_{n-2} \in [0,2\pi].$ With the above transformation, our integral now transforms to 

\begin{align}
\int_{\Be^n} \frac{1}{|x - z^*|^n} \, \dx \, &\approx \, \int_{0}^{1} \varrho^{n-1} 
\int_{0}^{\pi} \frac{ \sin^{n-2}\theta  \ {\rm d}\theta}{(t^2 - 2t \varrho \cos \theta + \varrho^2)^{n/2}} \, {\rm d}\varrho \notag \\
& = \int_{0}^{|z|} \varrho^{n-1} {\rm d}\varrho \int_{0}^{\pi} \frac{ \sin^{n-2}\theta \ {\rm d}\theta}{(1 - 2 \varrho \cos \theta + \varrho^2)^{n/2}}  \notag \\
& =  \int_{0}^{|z|} \frac{\varrho^{n-1}}{ \varrho^{n}} {\rm d} \varrho \int_{0}^{\pi} \frac{ \sin^{n-2}\theta \ {\rm d}\theta}{(1/\varrho^2 - 2 /\varrho \cos \theta + 1)^{n/2}}  \notag \\
& = \int_{0}^{|z|} \frac{\varrho^{n-1}}{ \varrho^{n}} \frac{\int_{0}^{\pi}  \sin^{n-2}\theta \ {\rm d}\theta}{(1/\varrho)^{n-2} (1/\varrho^2 - 1)} {\rm d}\varrho \notag \\
& = C \int_{0}^{|z|} \frac{\varrho^{n-1}}{ 1- \varrho^2} \, {\rm d}\varrho, \notag
\end{align}
where in the first line we have absorbed the term $2\pi\prod_{l=1}^{n-3}\int_0^{\pi} \sin^{l} \theta \ d\theta$ into $\approx$ sign and in the second line we have made the change of variable $\varrho \to t\varrho$. In the last line $C: = \int_{0}^{\pi}  \sin^{n-2} \theta \ {\rm d}\theta$ and in the fourth line we have used the well-known  formula 
\begin{align*}
\int_{0}^{\pi} \frac{ \sin^{n-2}\theta \ {\rm d}\theta}{(r^2 - 2r \cos \theta + 1)^{n/2}} \, {\rm d}\theta =  
\frac{1}{r^{n-2}(r^2 - 1)}\int_{0}^{\pi}  \sin^{n-2}\theta \ {\rm d}\theta, \ \ \mbox{for} \ r>1.
\end{align*}

Further, we estimate using $t  = \frac{1}{|z|} \approx 1$ as follows 

\begin{align}\label{eq-estimate}
\frac{C}{t} \int_{0}^{|z|} \frac{\varrho^{n-1}}{ 1- \varrho^2} \, {\rm d}\varrho  \lesssim |z|^{n-2} \int_{0}^{|z|}  \frac{\varrho }{1 - \varrho^2} \, {\rm d}\varrho 
 \lesssim  \log\left( \frac{1}{1 - |z|^2} \right) \approx  d(0, z). 
\end{align}
Substituting \eqref{eq-estimate} in \eqref{local-estimate} we obtain the desired upper bound
\begin{align*}
\int_{\bn} \calu[0]^{\alpha} \, \calu[z]^{\beta} \, \dvg \lesssim d(0, z) e^{-\cn d(0, z)}
= d(0, z) e^{-\cn \mbox{min}\{\alpha, \beta \} d(0, z)}.
\end{align*}

\medskip 

\noindent
{\it The lower estimate.} Now we shall compute the lower bound. We fix $\theta_0$ small to be determined later (actually, $\theta_0 = \pi/4$ will work).  We consider the following cone defined by the collection of all $x \in B_E(z,|z|)$ such that the angle between $(x-z)$ and $-z$ is less than or equal to $\theta_0.$
\begin{align*}
\mathcal{C} := \{ x \in B_E(z,|z|) : -  (x-z)\cdot  z \geq |x -z| |z| \, \cos \theta_0 \},
\end{align*}
where recall $\cdot$ denotes the standard Euclidean inner product. We emphasise that we shall choose $\theta_0$ independent of the point $z.$

\medskip

\begin{figure}[H]
\centering 
\begin{tikzpicture}[ scale=2.5]
  \draw[thick] (0,0) circle (1);
  \draw[thick] (0,1)--(0,-1);
    \draw[thick] (0.8,0)--(0.10717967697,-0.4);
  \draw[thick] (0.8,0)--(0.10717967697,0.4)  node[anchor=south west]{$\mathcal{C}$};
  \filldraw[black] (1.25,0) circle (1pt) node[anchor=south ]{z*};
  \draw (0.10717967697,-0.4) arc
    [
        start angle=30,
        end angle=-30,
        x radius=-0.8cm,
        y radius =-0.8cm
    ] ;
  \draw (0.530192378,0) arc
    [
        start angle=0,
        end angle=30,
        x radius=-0.3cm,
        y radius =-0.3cm
    ]  node[anchor=east]{$\tiny \theta_0$};
     \draw[dashed] (0.70710678118,0.70710678118)--(0.2,0.2) node[anchor=south east]{x};

       \filldraw[black] (0.2,0.2) circle (1pt);
       \draw[dashed] (1.25,0)--(0.2,0.2);
    \filldraw[black] (0.8,0) circle (1pt) node[anchor=south west]{z};
   \filldraw[black] (0,0) circle (1pt);
   \draw[dashed, <->] (-2,0)--(2,0);
 \end{tikzpicture}
 \caption{The above picture demonstrating the bounded cone $\mathcal{C}$ where the terms $(1-|x|^2)$ and $
 |x-z^*|$ are comparable for all $x \in \mathcal{C}.$ } 
\end{figure}
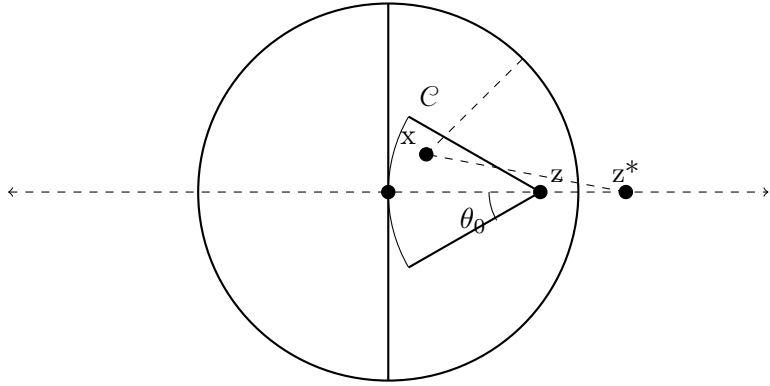

First we claim that if $|z| \sim 1,$ then 
\begin{align*}
|x - z^*| \lesssim (1 - |x|^2) \quad  \ \ \mbox{for all} \ x  \in \mathcal{C}.
\end{align*}

\noindent
{\it Proof of the claim:} By triangle inequality
\begin{align}\label{esti-cone-1}
|x - z^*| \leq |x - z| + |z - z^*| \leq |x - z| + \frac{1 - |z|^2}{|z|} \lesssim |x -z| + (1 - |z|^2).
\end{align}
On the other hand $|x|^2 = |x - z|^2 + 2  (x-z)\cdot z + |z|^2,$ and hence 
\begin{align} \label{esti-cone-11}
1 - |x|^2 &= 1 - |z|^2 - |x - z|^2 - 2 (x-z)\cdot z \geq (1 - |z|^2) - |x - z|^2 + 2|x-z||z| \cos \theta_0 \notag \\
& = (1 - |z|^2) + |x - z| \left(  2 |z| \cos \theta_0 - |x -z| \right).
\end{align}
In the $\mathcal{C},$ we have $|x -z| \leq |z|,$ therefore 
\begin{align*}
2 |z| \cos \theta_0 - |x -z| \geq  |z| (2\cos \theta_0 - 1)
 \geq (\sqrt{2}-1)|z| \\
 \end{align*}
if  $\theta_0  =\pi/4.$ We get  from \eqref{esti-cone-11}
\begin{align} \label{esti-cone-2}
1 - |x|^2 \geq (\sqrt{2} - 1) \left[  (1 - |z|^2) + |x -z| \right].
\end{align}
Hence from \eqref{esti-cone-1} and \eqref{esti-cone-2} the claim follows. Now given the above, we simplify the required integral as follows 
\begin{align*}
\int_{\mathcal{C}} \frac{(1 - |x|^2)^{\cn (\alpha + \beta) - n}}{|x - z^*|^{2 \cn \beta}} \, \dx \gtrsim \int_{\mathcal{C}} \frac{1}{|x - z^*|^n} \, \dx.
\end{align*}
Now we again make the following change of variable in polar co-ordinate 
\begin{align*}
&x_1 = -\varrho \, \sin \theta \sin \theta_1\ldots \sin \theta_{n-3}\sin \theta_{n-2} \\
& x_2 = -\varrho \, \sin \theta \sin \theta_1 \ldots \sin  \theta_{n-3} \cos \theta_{n-2} \\
& x_3 = -\varrho \, \sin \theta \sin \theta_1 \ldots \cos \theta_{n-3} \\
& \ldots \\
& x_n = |z| - \varrho \cos \theta, \quad z = |z|e_n, \ t = \frac{1}{|z|} \notag.
\end{align*}
With the above change of variable, we can write 
\begin{align}
|x - z^*|^2 & = |x|^2 + |z^*|^2 - 2 x_n z_n^* = \varrho^2 + |z|^2 - 2 |z| \varrho \cos \theta + \frac{1}{|z|^2} - 2 \left( |z| - \varrho \cos \theta \right) \frac{1}{|z|} \notag \\
& = \varrho^2 + 2 \varrho \cos \theta \left( \frac{1- |z|^2}{|z|} \right) + \frac{(1 - |z|^2)^2}{|z|^2} \notag.
\end{align}
This yields for $|z| \approx 1$
\begin{align}
\int_{\mathcal{C}} \frac{1}{|x - z^*|^n} \, \dx & \approx \int_{0}^{|z|} \varrho^{n-1} \, 
{\rm d}\varrho \int_{0}^{\theta_0} \frac{(\sin \theta)^{n-2} \, {\rm d}\theta}{\big[ \varrho^2 + 
2 \varrho \cos \theta \left( \frac{1- |z|^2}{|z|} \right) + \frac{(1 - |z|^2)^2}{|z|^2}\big]^{\frac{n}{2}}} \notag \\
& = \int_{0}^{\frac{|z|^2}{ 1 - |z|^2}} \dfrac{\left( \frac{ 1 - |z|^2}{|z|} \varrho \right)^{n-1}}{\left( \frac{1 - |z|^2}{|z|} \right)^{n-1}} \, {\rm d}\varrho \int_{0}^{\theta_0} 
\frac{(\sin \theta)^{n-2} \, {\rm d}\theta}{\big[ \varrho^2 + 
2 \varrho \cos \theta  + 1\big]^{\frac{n}{2}}} \notag\\
& =  \int_{0}^{\frac{|z|^2}{ 1 - |z|^2}}  \varrho^{n-1} {\rm d}\varrho \int_{0}^{\theta_0} 
\frac{(\sin \theta)^{n-2} \, {\rm d}\theta}{\big[ \varrho^2 + 
2 \varrho \cos \theta + 1\big]^{\frac{n}{2}}} \notag \\
& \gtrsim \int_{1}^{\frac{|z|^2}{ 1 - |z|^2}} \varrho^{n-1}   \dfrac{1}{\varrho^n} \int_{0}^{\theta_0} \sin^{n-2} \theta \ {\rm d}\theta \, {\rm d}\varrho\notag \\
& \gtrsim C_{\theta_0} \int_{1}^{\frac{|z|^2}{ 1 - |z|^2}} \frac{{\rm d}\varrho}{\varrho} \gtrsim \log \left(\frac{|z|^2}{1 - |z|^2} \right) \gtrsim \log \left(\frac{1}{1 - |z|^2} \right) \notag \\
& \approx d(0, z)
\end{align}
where in the second line we used the change of variable $\varrho \to \left( \frac{1- |z|^2}{|z|} \right)\varrho.$
Therefore from the above computation, we conclude 
\begin{align*}
\int_{\bn} \calu[0]^{\alpha} \, \calu[z]^{\beta} \, \dvg \gtrsim d(0, z) e^{-\cn \beta d(0, z)}.
\end{align*}
\end{proof}
\begin{remark}
{\rm 
The sharpness of the exponent in Lemma~\ref{lem-int-2} in case $\alpha \neq \beta$ can also be seen from the following heuristic explanation. Note that 
\begin{align*}
e^{d(x, z) - d(0,z)} \approx \frac{1 - 2 x.z + |x|^2 |z|^2}{1 - |x|^2}.
\end{align*}
Then as $|z| \rightarrow 1,$ i.e. $z \rightarrow e$ with $|e| =1,$ there holds 
\begin{align*}
e^{d(x, z) - d(0,z)}  \rightarrow \frac{1 - 2 x \cdot e + |x|^2}{1 - |x|^2} \approx |x - e|^2 e^{d(0, x)}
\end{align*}
Therefore for $\alpha > \beta,$ denoting $c = \cn$
\begin{align*}
e^{c \beta d(0, z)} \int_{\bn} e^{- c \alpha d(0, x)} e^{- c \beta d(x, z)} \, \dvg(x) \notag 
 \rightarrow & \int_{\bn} e^{- c \alpha d(0, x)} |x - e|^{-2 c \beta} e^{- c \beta d(0, x)} \, \dvg(x) \\
 \approx  &\int_{\Be^n} \frac{(1 - |x|^2)^{c (\alpha + \beta) - n}}{|x - e|^{2 c \beta}} \, \dvg(x) \approx 1.
\end{align*}

By replacing $\beta$ in $e^{c \beta d(0, z)}$ by $\nu,$ we observe that as  $|z| \rightarrow 1,$
\begin{align*}
  e^{c \nu d(0, z)} \int_{\bn} e^{- c \alpha d(0, x)} e^{- c \beta d(x, z)} \, \dvg(x) 
\rightarrow
\begin{cases}
0 \ \ \ \ \ \ \ \ \   \mbox{if} \ \nu < \beta \\
+ \infty  \ \ \ \ \ \mbox{if} \ \nu >  \beta. \\
\end{cases}
\end{align*}

This proves that $\nu = \beta$ is the precise asymptotic estimate for the interaction terms. 
}

\end{remark}


\subsection{Interaction of three bubbles}
In this section, we will derive the interaction of three hyperbolic bubbles. We denote $\calu_i = \calu[z_i], \ i = 1,2,3.$ The main focus of this section is to derive upper bounds on the integral of interactions. 
We recall 
\begin{align*}
Q_{ij} := e^{-\cn d(z_i,z_j)} \ \ \ \ \mbox{and} \ \ \ \ Q := \max_{i \neq j} \ Q_{ij}.
\end{align*}

We prove the following lemma.

\begin{lemma} \label{three bubble}
Let $\calu_i = \calu[z_i], i = 1,2,3$ be three hyperbolic bubbles and $p>1$ and assume $Q< <1$. Then 
\begin{align*}
\int_{\bn} \calu_1^{p-1}\calu_2\calu_3 \ \dvg  \ 
\begin{cases}
\ \lesssim \ Q^{\frac{3}{2}}\left(\ln \frac{1}{Q}\right)^{\frac{1}{3}}, \ \ \ \ \ \ \ \ \ \  \ \ \ \ \ \ \mbox{if} \ p-1>1, \\
\ \lesssim \ Q^{\frac{3}{2}}\ln \frac{1}{Q}, \ \ \ \ \ \ \ \ \ \ \ \ \ \ \ \ \ \ \ \ \   \mbox{if} \ p-1=1, \\
\ \lesssim_{\nu} \ Q^{\nu}, \ \mbox{for any} \ \nu < \frac{p+1}{2}, \ \ \ \ \mbox{if} \ p-1<1.
\end{cases}
\end{align*}
\end{lemma}

\begin{proof}
\noindent 
\item[Case (i).] $p - 1>1.$ We apply H\"older inequality as follows
\begin{align*}
\int_{\bn}  \calu_1^{p-1}\calu_2\calu_3 \ \dvg &= \int_{\bn}
 (\calu_1^{\frac{p-1}{2}}\calu_2^{\frac{1}{2}})  (\calu_2^{\frac{1}{2}}\calu_3^{\frac{1}{2}})  (\calu_1^{\frac{p-1}{2}}\calu_3^{\frac{1}{2}}) \ \dvg \\
 &\leq \left(\int_{\bn}  \calu_1^{\frac{3(p-1)}{2}}\calu_2^{\frac{3}{2}} \ \dvg\right)^{\frac{1}{3}} \left(\int_{\bn}  \calu_2^{\frac{3}{2}}\calu_3^{\frac{3}{2}} \ \dvg\right)^{\frac{1}{3}} \left(\int_{\bn}  \calu_1^{\frac{3(p-1)}{2}}\calu_3^{\frac{3}{2}} \ \dvg\right)^{\frac{1}{3}}\\
 &\lesssim Q_{12}^{\frac{1}{2}}\left(Q_{23}^{\frac{3}{2}} \ln \frac{1}{Q_{23}}\right)^{\frac{1}{3}}Q_{13}^{\frac{1}{2}}\\
 &\lesssim Q^{\frac{3}{2}}\left(\ln \frac{1}{Q}\right)^{\frac{1}{3}},
\end{align*}
where we have used $t \to t^{\epsilon \ln (1/t)}$ is increasing near $t = 0.$

\medskip

\noindent 
\item[Case (ii).] $p - 1=1.$ The idea is the same as in case (i). We apply H\"older inequality as follows
\begin{align*}
\int_{\bn}  \calu_1\calu_2\calu_3 \ \dvg &= \int_{\bn}
 (\calu_1^{\frac{1}{2}}\calu_2^{\frac{1}{2}})  (\calu_2^{\frac{1}{2}}\calu_3^{\frac{1}{2}})  (\calu_1^{\frac{1}{2}}\calu_3^{\frac{1}{2}}) \ \dvg \\
 &\leq \left(\int_{\bn}  \calu_1^{\frac{3}{2}}\calu_2^{\frac{3}{2}} \ \dvg\right)^{\frac{1}{3}} \left(\int_{\bn}  \calu_2^{\frac{3}{2}}\calu_3^{\frac{3}{2}} \ \dvg\right)^{\frac{1}{3}} \left(\int_{\bn}  \calu_1^{\frac{3}{2}}\calu_3^{\frac{3}{2}} \ \dvg\right)^{\frac{1}{3}}\\
 &\lesssim \left(Q_{12}^{\frac{3}{2}} \ln \frac{1}{Q_{12}}\right)^{\frac{1}{3}}\left(Q_{23}^{\frac{3}{2}} \ln \frac{1}{Q_{23}}\right)^{\frac{1}{3}}\left(Q_{13}^{\frac{3}{2}} \ln \frac{1}{Q_{13}}\right)^{\frac{1}{3}}\\
 &\lesssim Q^{\frac{3}{2}}\ln \frac{1}{Q}.
\end{align*}

\medskip

\noindent
\item[Case (iii).] We fix $\beta > \alpha > 1$ and $s_1,s_2 \in (0,1)$ that to be decided later and apply H\"older
\begin{align*}
\int_{\bn}  \calu_1^{p-1}\calu_2\calu_3 \ \dvg &  \leq \left(\int_{\bn}  \calu_1^{\alpha}\calu_2^{\beta} \ \dvg\right)^{s_1} \left(\int_{\bn}  \calu_1^{\alpha}\calu_3^{\beta} \ \dvg\right)^{s_1} \left(\int_{\bn}  \calu_2^{\frac{p+1}{2}}\calu_3^{\frac{p+1}{2}} \ \dvg\right)^{s_2}.
\end{align*}
A quick sanity check gives
\begin{align} \label{sanity check}
2\alpha s_1 = p-1, \ \ \ \beta s_1 + \frac{p+1}{2}s_2 = 1.
\end{align}
Since we assumed $\alpha < \beta,$ the interaction estimates of the previous section gives
\begin{align*}
\int_{\bn}  \calu_1^{p-1}\calu_2\calu_3 \ \dvg &  \leq Q_{12}^{\alpha s_1}Q_{13}^{\alpha s_1}\left(Q_{23}^{\frac{p+1}{2}} \ln \frac{1}{Q_{23}}\right)^{s_2} \\
& \leq Q^{2\alpha s_1 + \frac{p+1}{2}s_2}\left( \ln \frac{1}{Q}\right)^{s_2}\\
&= Q^{p-\beta s_1}\left( \ln \frac{1}{Q}\right)^{s_2}.
\end{align*}
If we want the interaction to be $o(Q)$ then we need $p - \beta s_1 >1.$ Now let us check that all the above requirements are feasible. Choose $\epsilon > 0$ small and set $\alpha = \frac{(p+1)(1-\epsilon)}{2}, \beta = \frac{(p+1)(1+\epsilon)}{2}.$ Then \eqref{sanity check} gives for sufficiently small $\epsilon$
\begin{align*}
&s_1 = \frac{p-1}{(p+1)(1-\epsilon)}  < 1, \ \ \beta s_1 = \frac{(1+\epsilon)(p-1)}{2(1-\epsilon)} < p-1, \\
&\frac{p+1}{2}s_1 = 1-\beta s_1 = \frac{3-\epsilon - (1+\epsilon)p}{2(1-\epsilon)} < \frac{p+1}{2}.
\end{align*}
Next we show that for any $\nu < \frac{p+1}{2}$ we can choose $\epsilon > 0$ such that $\nu < p - \beta s_1.$ Indeed, direct computation gives
\begin{align*}
p - \beta s_1 = p - \frac{(1+\epsilon)(p-1)}{2(1- \epsilon)} \rightarrow \frac{p+1}{2}, \ \ \mbox{as} \ \epsilon \rightarrow 0
\end{align*}
and hence we can choose such $\epsilon.$ As a result
\begin{align*}
Q^{p-\beta s_1}\left( \ln \frac{1}{Q}\right)^{s_2} = Q^{\nu}Q^{p-\beta s_1 - \nu}\left( \ln \frac{1}{Q}\right)^{s_2} 
\lesssim_{\nu} Q^{\nu}
\end{align*}
 if $Q < < 1.$ This completes the proof.

\end{proof}

\begin{remark}
The constant in the third case i.e $p-1 < 1$ depends on $\nu.$ In our first case if $p -1> \frac{3}{2}$ then we can refine the upper bound to $Q^{\min\{p-1,2\}}.$ Since our interest is in $p \leq \tstar - 1$ such assumptions are void in dimension $n \geq 5.$ The proof follows in the same manner by clubbing $\calu_1^{p-1}\calu_2\calu_3 = \left(\calu_1^{\frac{p-1}{2}}\calu_2\right) \left(\calu_1^{\frac{p-1}{2}}\calu_3 \right)$
and using H\"older inequality.
\end{remark}

\subsection{Interaction of bubbles and the derivative of bubbles}

In this section, we find the estimates on the interaction of bubbles and the space derivatives of other bubbles. Before proceeding let us first simplify the formula of the derivative. In coordinates $\frac{d}{dt}|_{t =0}(\calu \circ \tau_{te_j}) = V_j(\calu),$ where 
$V_j = (1+|x|^2)\frac{\partial}{\partial x_j} - 2x_j \sum_{l = 1}^n x_l\frac{\partial}{\partial x_l}$ and the derivatives involved are in the sense of Euclidean (see Appendix \ref{Appendix}).

\medskip

Recall the Euclidean norm $r = |x|$ and the hyperbolic distance $\rho = d(x,0)$ is related by $r = \tanh \frac{\rho}{2}.$ Since a hyperbolic ball with center $0$ is also a Euclidean ball with center $0$ but possibly different radius, we deduce $\calu$ is a function of $r.$ By abuse of notation we denote $\calu(x) = \calu(\rho) = U_1(r).$ Then direct computation gives
\begin{align}
V_j(\calu) = \frac{x_j}{r}(1-r^2)\frac{d}{dr}(U_{1}(r)) = 2 \frac{x_j}{|x|}\calu^{\prime}(\rho)
\end{align}
 where $\prime$ denotes the derivative with respect to $\rho$ and recall $\rho = d(x,0).$
It follows directly from the above expression and the bound $|\calu^{\prime}(\rho)| \lesssim \calu (\rho)$ that $|V_j(\calu)| \lesssim \calu$ on $\bn.$ Moreover, note that as $\calu^{\prime}<0$ for all $\rho > 0,$ $V_j(\calu)$ changes sign in $\bn.$

\medskip

We prove the following 

\begin{lemma} \label{interaction derivatives}
Let $z= (z_1,\ldots,z_n) \in \bn$ does not lie on the hyperplane $P_j:=\{x \in \rn \ | \ x_j = 0\}.$ Then 
\begin{align*}
\int_{\bn} \calu[z]^pV_j(\calu[0]) \ \dvg \approx_z
\begin{cases}
e^{-\cn d(z,0)}, \ \ \ \ \ \ \mbox{if} \ z_j < 0, \\
-e^{-\cn d(z,0)}, \ \ \ \ \mbox{if} \ z_j > 0, \\
\end{cases}
\end{align*}
where $\approx_z$ indicates the constant depends on (the position of) $z.$
\end{lemma}

\begin{remark}
Before proving the lemma let us first remark why it should be true. Also, this is a good time to remark that this is in sharp contrast to the Euclidean case. In the previous section, we observed that the interaction is an exponentially decaying function of the distance between the points where the bubbles are most concentrated. The derivative of an exponential function is an exponential function while the derivative of the distance function is of the absolute value of approximately 1 and that precisely plays the decisive role here. Formally,
\begin{align*}
\int_{\bn} \calu[z]^pV_j(\calu[0]) \ \dvg & = \int_{\bn} \calu[z]^p \frac{d}{dt}\Big|_{t=0}\calu[te_j] \ \dvg\\
&= \frac{d}{dt}\Big|_{t=0}\int_{\bn} \calu[z]^p \calu[te_j]\ \dvg\\
&\approx \frac{d}{dt}\Big|_{t=0}e^{-\cn d(z,te_j)} \\
&\approx e^{-\cn d(z,0)} \frac{d}{dt}|_{t=0}d(z,te_j) \\
&\approx_z \pm e^{-\cn d(z,0)}.
\end{align*}
On the other hand, in the Euclidean case,  the interaction of Aubin-Talenti bubbles has polynomial decay:
\begin{align*}
\int_{\rn} U[z,1]^pU[z_1,1] \ dx \approx |z-z_1|^{-(n-2)} 
\end{align*}
and hence 
\begin{align*}
\Big|\int_{\rn} U[z,1]^p\partial_{x_j}U[0,1] \ dx\Big| &=  \Big|\frac{d}{dt}\Big|_{t=0}\int_{\rn} U[z,1]^pU[te_j,1] \ dx \Big|\\
&\approx  \Big|\frac{d}{dt}\Big|_{t=0} |z-te_j|^{-(n-2)}\Big| \\
&\approx |z|^{-(n-1)} << |z|^{-(n-2)} \ \ \mbox{if} \ \ |z|>>1.
\end{align*}
See the treatise by A. Bahri \cite[Estimate (F11)]{Bahri}. However, the interaction of bubbles and the $\mu$-derivative of another bubble has the same decay as the interaction of bubbles if the height of the bubbles is comparable \cite[Estimate (F16)]{Bahri}. 
\end{remark}

\medskip

\noindent
{\bf Proof of Lemma \ref{interaction derivatives}.}
\begin{proof}
First, we simplify the expression as follows
\begin{align}\label{der1}
\int_{\bn} \calu[z]^pV_j(\calu[0]) \ \dvg &= \int_{\bn}(-\Delta_{\bn} \calu[z])V_j(\calu[0]) \ \dvg \notag \\
&=\int_{\bn}\calu[z] (-\Delta_{\bn} V_j(\calu[0])) \ \dvg \notag\\
&= p\int_{\bn}\calu[z] \calu[0]^{p-1}V_j(\calu[0]) \ \dvg \notag \\
&= 2p\int_{\bn}\calu[z] \calu[0]^{p-1} \frac{x_j}{|x|}\calu[0]^{\prime} \ \dvg.
\end{align}
Recall that $\calu[0]^{\prime} \leq 0$ on $\bn$ and radial and hence the term $\frac{x_j}{|x|}$ will play the central role in our estimate. Also, note that the upper bound follows directly from the bound $|\calu^{\prime}(\rho)| \lesssim \calu (\rho)$ and the estimates on the interaction of bubbles obtained in the last subsection.

\medskip

For $x \in \bn,$ we denote by $\hat x$ the reflection of $x$ with respect to the plane $P_j.$ 
We denote by $\bn_{-}$ the negative half of the ball $\{x \in \bn \ | \ x_j < 0\}.$ $\bn_+$ is similarly defined with $< 0$ replaced by $> 0.$
Without loss of generality, we assume that $z \in \bn_-.$
First, we claim that 

\medskip

\noindent
{\bf Claim:} $d(x,z) \leq d(\hat x,z)$ for all $x \in \bn_-.$

\medskip

The proof follows from the formula of the hyperbolic distance.
\begin{align*}
\cosh^2\left(\frac{d(\hat x,z)}{2}\right) - \cosh^2\left(\frac{d( x,z)}{2}\right) &= \frac{2(x - \hat x)\cdot z}{(1-|x|^2)(1-|z|^2)}\\
&=\frac{2x_jz_j}{(1-|x|^2)(1-|z|^2)} > 0
\end{align*}
as both $x,z \in \bn_-.$ Since $\cosh$ is strictly increasing the claim follows. Since $\calu[z]$ is strictly decreasing in $d(x,z)$ we deduce $\calu[z](x)> \calu[z](\hat x)$ for all $x \in \bn_-.$
We further simplify the integral 
\eqref{der1}
\begin{align}\label{der2}
\int_{\bn} \calu[z]^pV_j(\calu[0]) \ \dvg &= 2p\int_{\bn}\calu[z] \calu[0]^{p-1} \frac{x_j}{|x|}\calu[0]^{\prime} \ \dvg \notag \\
& = 2p\int_{\bn_-}\calu[z] \calu[0]^{p-1} \frac{x_j}{|x|}\calu[0]^{\prime} \ \dvg + 2p\int_{\bn_+}\calu[z] \calu[0]^{p-1} \frac{x_j}{|x|}\calu[0]^{\prime} \ \dvg \notag \\
&= 2p\int_{\bn_-}(\calu[z](x) - \calu[z](\hat x)) \calu[0]^{p-1}(x) \frac{x_j}{|x|}\calu[0]^{\prime}(x) \ \dvg(x).
\end{align}
Notice that the integrand in \eqref{der2} is non-negative. The proof can be finished if we show that $\calu[z](x) - \calu[z](\hat x) 
\gtrsim \calu[z](x)$ on a compact subdomain $B$ of $\bn_-.$ We achieve this by comparison principle.
Let us denote $w(x ) = \calu[z](x) - \calu[z](\hat x)$ and fix a compact domain $B \subset \bn_-.$ Then $w$ satisfies 
$$
(-\Delta_{\bn} - \lambda)w = \calu[z]^p(x) - \calu[z]^p(\hat x)
 \geq p \calu[z]^{p-1}(x)w(x)
 \geq \calu[z]^{p-1}(x)w(x)
$$
as $w \geq 0$ in $B.$ Let $c_1 = \inf_{x \in \partial B} w(x) >0$ and let $c_2$ be such that $c_2 \sup_{x \in \partial B} \calu[z](x) \leq c_1.$ Note that the constants $c_1,c_2$ depends on $B$  and hence on the location of $z.$ Since $c_2\calu[z]$ satisfies $(-\Delta_{\bn} - \lambda)(c_2\calu[z]) = \calu[z]^{p-1}(c_2\calu[z])$, weak comparison principle gives $w \geq c_2\calu[z]$ in $B.$ Indeed, denoting $v(x) = w(x) - c_2\calu[z](x),$ we see that $v$ satisfies 
\begin{align} \label{eqn-v}
\begin{cases}
-\Delta_{\bn} v - \lambda v \geq \calu[z]^{p-1}v, \ \ \ \ \ \mbox{in} \ B,\\
\ \ \ \ \ \ \ \ \ \ \ \  \ \ \  v \geq 0 \ \ \ \ \ \ \ \ \ \ \ \ \ \mbox{on} \ \partial B.
\end{cases}
\end{align}
Multiplying \eqref{eqn-v} by $v^- \in H_0^1(B)$ and integrating by parts gives
\begin{align} \label{inq-v-}
\int_B (|\nabla_{\bn} v^-|^2 - \lambda |v^-|^2) \ \dvg \leq \int_B \calu[z]^{p-1}|v^-|^2 \ \dvg.
\end{align}

But \eqref{inq-v-} is possible only when $v^- = 0$ because the first Dirichlet eigenvalue of the operator
$(-\Delta_{\bn} - \lambda)/\calu[z]^{p-1}$ on $\bn_-$ is $p>1$ (see Lemma \ref{half ev}) and so on $B$ it must be bigger than $1.$ Hence 
$v \geq 0$ on $B.$
 Moreover, if the compact set stays away from the hyperplane $\{x_j = 0\}$ then we have the lower bound $\frac{x_j}{|x|}\calu[0]^{\prime}(x) \gtrsim \calu[0]$ on $B.$ Therefore  
\begin{align*}
\int_{\bn} \calu[z]^pV_j(\calu[0])  &= 2p\int_{\bn_-}(\calu[z](x) - \calu[z](\hat x)) \calu[0]^{p-1}(x) \frac{x_j}{|x|}\calu[0]^{\prime}(x) \ \dvg(x)\\
&\geq 2pc_2 \int_{B} \calu[z] \calu[0]^{p-1}(x)\frac{x_j}{|x|}\calu[0]^{\prime}(x) \ \dvg(x)\\
&\gtrsim_z \int_{B} \calu[z] \calu[0]^p \ \dvg\\
&\gtrsim_z e^{-\cn d(z,0)}.
\end{align*}
This completes the proof.
\end{proof}

\medskip

 \begin{remark}
 {\rm 
 Note that we can not do better than $\approx_z,$ i.e., the constant must depend on $z.$ To view this note that 
 if $z_j =0$ then it follows from the formula of the hyperbolic distance that $d(x,z) = d(\hat x, z)$ where recall $\hat x$ is the reflection of $x$ with respect to the plane $\{x_j = 0\}.$ As a result, the interaction vanishes and hence the constant must depend on the position of $z.$ Indeed, in the next lemma we show that if $z_j$ stays away from 0, then we can make the bound uniform in the sense that if $|z_j| \geq \kappa >0 $ then the constant depends only on $\kappa.$ 
 }
\end{remark}

\begin{lemma} \label{interaction derivatives refined}
Let $z= (z_1,\ldots,z_n) \in \bn$ satisfy $|z_j| \geq \kappa$ for some constant $\kappa \in (0,1).$ Then there exists a $\delta_0 > 0$ such that for all $z$ satisfying $1- |z| < \delta_0$

\begin{align*}
\int_{\bn} \calu[z]^pV_j(\calu[0]) \ \dvg \approx_{k,\delta_0}
\begin{cases}
e^{-cd(z,0)}, \ \ \ \ \ \ \mbox{if} \ z_j < -\kappa, \\
-e^{-cd(z,0)}, \ \ \ \ \mbox{if} \ z_j > \kappa, \\
\end{cases}
\end{align*}
where $\approx_{k,\delta_0}$ indicates a constant that depends on $\kappa, \delta_0$ and the parameters $n, \lambda,p.$
\end{lemma}

\begin{proof}
As in Lemma \ref{interaction derivatives}, we may assume $z_j \leq -\kappa$ and we only need to obtain the lower bound. Given $z$ as in the statement of the lemma, we set $e = \frac{z}{|z|}.$ We fix $R_0$ large whose value will be decided later. We fix a compact set $B$ ($= B_e$, see Figure $2$) satisfying the following two conditions:

\begin{itemize}
\item[(i)] $|x-e| < \frac{1}{R_0}$ for all $x \in B,$
\item[(ii)] The Euclidean diameter of $B$ $\leq \frac{\kappa}{3}$ and $x_j \leq -\frac{\kappa}{2},$ for all $x \in B.$
\end{itemize}

\begin{figure}[H]
\centering 
\begin{tikzpicture}[ scale=2.5]
\draw[thick] (1,0) arc
    [
        start angle=0,
        end angle=180,
        x radius=1cm,
        y radius =1cm
    ] ;
    \draw[dashed] (-2,0)--(2,0);
     \draw[dashed, <->] (0,0)--(0,1.25)node[anchor=south west]{$-e_j$};
     \draw[dashed] (0,0)--(-0.95393920141, 0.3);
      \filldraw[black] (-0.76315136113,0.24) circle (0.5pt) node[anchor=south west]{$z$};
      \filldraw[black](-0.95393920141, 0.3)circle (0.5pt) node[scale = 0.5, anchor=south east]{$e = \frac{z}{|z|}$};
      \draw[thick] (-0.85,.3) circle (.065);
     \filldraw[black] (-.75,.35) circle (0.00000003pt) node[scale = 0.5, anchor=south east ]{$B_e$};
\end{tikzpicture}
 \caption{The above picture demonstrating the choice of the compact set $B_e$ which depends only on the direction $e$ of $z$.} 
\end{figure}
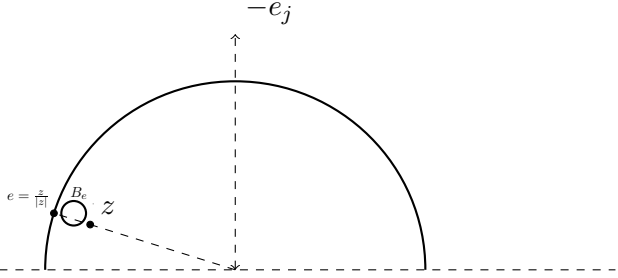

Note that $B$ depends only on the direction of $z.$
As before we denote by $\hat x$ the reflection of $x$ with respect to the plane $\{x_j = 0\}.$ We claim that if $R_0$ is sufficiently large and $\delta_0$ is sufficiently small then $d(\hat x, z) - d(x,z)$ is large. We compute
\begin{align*}
d(\hat x, z) - d(x,z) &= 2\left(  \ln \left(\frac{|z||\hat x-z^*| + |\hat x-z|}{\sqrt{(1-|x|^2)(1-|z|^2)}}\right) -  \ln \left(\frac{|z||x-z^*| + |x-z|}{\sqrt{(1-|x|^2)(1-|z|^2)}}\right)\right) \notag \\
&= 2  \ln \left(\frac{|z||\hat x-z^*| + |\hat x-z|}{|z||x-z^*| + |x-z|}\right)
\end{align*}
Recalling $e = \frac{z}{|z|}, z_k \leq -\kappa$ and $1-|z| < \delta_0$ we estimate
\begin{align*}
|z||x-z^*| + |x-z| &\leq |z||x-e| + |z||e - z^*| + |x-e| + |e-z| \\&\leq (1+ |z|)|x-e| + (1 + |z|^{-1})(1-|z|)\\
&\leq (1 + \kappa^{-1})(R_0^{-1} + \delta_0).
\end{align*}
On the other hand as $x_k \leq -\frac{\kappa}{2}$ we see that $|z||\hat x-z^*| + |\hat x-z| \gtrsim \kappa$ and hence 

\begin{align*}
d(\hat x, z) - d(x,z) \gtrsim 2 \ln \left( \frac{\kappa}{(1 + \kappa^{-1})(R_0^{-1} + \delta_0)}\right) =: R_1.
\end{align*}
Clearly, if $R_0$ is sufficiently large and if $\delta_0$ is sufficiently small then $R_1$ can be made large. Now recall that there exists universal positive constants $\chi_1, \chi_2$ such that $\chi_1 e^{-\cn d(x,z)} \leq \calu[z](x) \leq \chi_2 e^{-\cn d(x,z)}$ for all $x \in \bn.$ On the boundary of the fixed ball $B, x \in \partial B$ we have 
\begin{align*}
\calu[z](x) - \calu[z](\hat x) &\geq \chi_1 e^{-\cn d(x,z)} - \chi_2e^{-\cn d(\hat x, z)} \\
&\geq \chi_1 e^{-\cn d(x,z)} - \chi_2e^{-\cn (d( x, z) + R_1)} \\
&\geq (\chi_1 - \chi_2e^{-\cn R_1}) e^{-\cn d(x,z)} \\
&\geq \left(\frac{\chi_1 - \chi_2e^{- \cn R_1}}{\chi_2}\right)\calu[z](x).
\end{align*}
Hence the choice of $R_1$ such that $\chi_1 - \chi_2e^{-\cn R_1} \geq \chi_1/2$ would suffice to ensure uniform lower bound $\calu[z](x) - \calu[z](\hat x) \gtrsim_{\kappa,\delta_0} \calu[z](x)$ on $B.$ Moreover, as $x_k \leq -\frac{\kappa}{2}$ on $B$ we have $\frac{x_j}{|x|}\calu[0]^{\prime} \gtrsim_{\kappa} \calu[0]$ on $B.$ Hence 
\begin{align} \label{int-der}
\int_{\bn} \calu[z]^pV_j(\calu[0]) \ \dvg&= 2p\int_{\bn_-}(\calu[z](x) - \calu[z](\hat x)) \calu[0]^{p-1}(x) \frac{x_j}{|x|}\calu[0]^{\prime}(x) \ \dvg(x)  \notag \\
& \gtrsim_{\kappa,\delta_0} \int_B \calu[z]\calu[0]^p \ \dvg\notag \\
& \gtrsim_{\kappa,\delta_0} e^{- \cn d(z,0)} \int_B\frac{(1-|x|^2)^{\cn (p+1)-n}}{|x-z^*|^{2\cn }} \ dx.
\end{align}
Note that the same $B$ can be chosen for any points on the line segment $[z,e).$
Finally, to conclude the proof we note that given any two points $z_1$ and $z_2$ satisfying the hypothesis of the lemma and $|z_1| = |z_2|$, let $A$ be the orthogonal transformation that takes $z_1$ to $z_2.$ If $B_1$ is the ball that is chosen for the point $z_1$ according to the properties $(i)$ and $(ii)$ then $B_2 = A(B_1)$ would satisfy the same properties for the point $z_2.$  Hence for each such $z$ we can choose $B$ so that the integral involved in \eqref{int-der} can be bound from below uniformly with respect to $z.$ This completes the proof of the lemma.
\end{proof}

\section{Proof of the main stability theorem}
In this section, we prove the main theorem of this paper, namely Theorem~\ref{main theorem}.

\medskip

{\bf Proof of Theorem~\ref{main theorem}:} 
Let $\calu_i:=\calu[z_i]$, $i=1,\cdots, N$ and $\sigma=\sum_{i=1}^N\alpha_i\calu_i$ be the linear combination of hyperbolic bubbles that  is closest to $u$ in $\|\cdot\|_{\la}$, i.e.,  
$$\|u-\sigma\|_{\lambda}=\min_{\tilde\alpha_1,\cdots, \tilde\alpha_N\in\mathbb{R}, \\ \tilde z_1, \cdots,\tilde z_N\in\bn}\bigg(\int_{\bn}|\nabla_{\bn}(u-\tilde\sigma)|^2 \, \dvg-\lambda\int_{\bn}|u-\tilde\sigma|^2 \, \dvg\bigg)^\frac{1}{2},$$
and $\tilde \sigma = \sum_{i=1}^N \tilde \alpha_i \calu[\tilde z_i].$
Let $\rho:=u-\sigma$. From the hypothesis, it directly follows that $\|\rho\|_{\lambda}\leq \delta$. Furthermore, since the bubbles $\tilde\calu_i$ are $\delta-$interacting, the family $(\alpha_i, \calu_i)_{1\leq i\leq N}$ is $\delta'-$interacting for some $\delta'$ that goes to $0$ as $\delta$ goes to $0$. 

\medskip

Summing up, we can say that {\it qualitatively} $\sigma$ is a sum of weakly-interacting hyperbolic bubbles and $\|\rho\|_{\lambda}$ is small.

\medskip

Since $\sigma$ minimizes $H^1_{\lambda}$ distance from $u$,  $\rho$ is $H^1_{\lambda}$ orthogonal to the manifold composed of linear combination of $N$ hyperbolic bubbles (see \cite[(5.1)]{BGKM}), namely the following orthogonality conditions hold:
\begin{equation}\label{10-11-1}
 \int_{\bn} \left(\langle \nabla_{\bn} \rho , \nabla_{\bn} \calu_i \rangle_{\bn} \, - \, \lambda\rho\calu_i \right)\dvg=0
\end{equation}
\begin{equation}\label{10-11-2}
 \int_{\bn} \left(\langle \nabla_{\bn} \rho , \nabla_{\bn} V_j(\calu_i) \rangle_{\bn} \, - \, \lambda\rho V_j(\calu_i) \right)\dvg=0
\end{equation}
for all $i=1,\ldots,N, j = 1,\ldots,n,$ and where $V_j(\calu_i)=\frac{d}{dt}\calu_i\circ\tau_{te_j}\big|_{t=0}$.

Since $\calu_i$, $V_j(\calu_i)$ are eigenfunctions for $\frac{-\Delta_{\bn}-\lambda}{\calu_i^{p-1}}$ (see \cite[Proposition 3.1]{BGKM} or Appendix) the above orthogonality conditions are equivalent to 
\begin{equation}\label{10-11-3}
\int_{\bn}\calu_{i}^{p}\,\rho\,  \dvg=0,
\end{equation}
\begin{equation}\label{10-11-4}
\int_{\bn}V_j(\calu_i)\calu_i^{p-1}\,\rho \  \dvg =0.
\end{equation}
for all $i=1,\ldots,N, j = 1,\ldots,n.$
Our goal is to show that $\|\rho\|_{\lambda}$ is controlled by $\|\Delta_{\bn}u+\lambda u+|u|^{p-1}u\|_{H^{-1}(\bn)}$. To achieve that let us start testing $\Delta_{\bn}u+\lambda u+|u|^{p-1}u$ by $\rho$, exploiting the orthogonality condition \eqref{10-11-1} yields
\begin{align}\label{10-11-5}
\|\rho\|_{\lambda}^2 &=\int_{\bn}[\langle \nabla_{\bn}\rho, \nabla_{\bn}(u-\sigma) \rangle_{\bn}-\lambda \rho(u-\sigma)]\,\dvg\notag\\
&=\int_{\bn}(\langle \nabla_{\bn}\rho, \nabla_{\bn}u \rangle_{\bn} -\lambda \rho u)\,\dvg\notag\\
&=\int_{\bn}(-\Delta_{\bn}u-\lambda u-u|u|^{p-1})\rho\,\dvg+\int_{\bn}u|u|^{p-1}\rho\,\dvg\notag\\
&\leq\|-\Delta_{\bn}u-\lambda u-u|u|^{p-1}\|_{H^{-1}(\bn)}\|\rho\|_{H^1(\bn)}+\int_{\bn}u|u|^{p-1}\rho\, \dvg\notag\\
&\lesssim\|\Delta_{\bn}u+\lambda u+u|u|^{p-1}\|_{H^{-1}(\bn)}\|\rho\|_{\la}+\int_{\bn}u|u|^{p-1}\rho\, \dvg.
\end{align}
To control the last term in \eqref{10-11-5}, we use the following elementary estimates
\begin{equation}\label{10-11-6}
\bigg|(a+b)|a+b|^{p-1}-a|a|^{p-1}\bigg|\leq p|a|^{p-1}|b|+C_p\bigg(|a|^{p-2}|b|^2+|b|^p\bigg),
\end{equation}
\begin{equation}\label{10-11-7}
\bigg|\bigg(\sum_{i=1}^N a_i\bigg)\bigg|\sum_{i=1}^Na_i\bigg|^{p-1}-\sum_{i=1}^Na_i|a_i|^{p-1}\bigg|\lesssim\sum_{1\leq i\neq k\leq N}|a_i|^{p-1}|a_k|,
\end{equation}
that holds for any $a,b\in\R$ and for any $a_1,\cdots, a_N\in\R$. Applying \eqref{10-11-6} with $a=\sigma$ and $b=\rho$, and \eqref{10-11-7} with $a_i=\alpha_i\calu_i$ yields
$$\bigg|u|u|^{p-1}-|\sigma|^{p-1}\sigma\bigg|\leq p|\sigma|^{p-1}|\rho|+C(|\sigma|^{p-2}|\rho|^2+|\rho|^p)$$
where $C = C_{n,\lambda, N,p}$ and
$$\bigg| \sigma|\sigma|^{p-1}-\sum_{i=1}^N\alpha_i\calu_i|\alpha_i\calu_i|^{p-1}\bigg|\lesssim\sum_{1\leq i\neq k\leq N}|\alpha_i\calu_i|^{p-1}|\alpha_k\calu_k|.$$
Combining the above two estimates we deduce
$$\bigg|u|u|^{p-1}-\sum_{i=1}^N\alpha_i|\alpha_i|^{p-1}\calu_i^p\bigg|\leq p|\sigma|^{p-1}|\rho|+C\bigg(|\sigma|^{p-2}|\rho|^2+|\rho|^p+\sum_{1\leq i\neq k\leq N}\calu_i^{p-1}\calu_k\bigg).$$
Therefore, using \eqref{10-11-3}, we get
\begin{align}\label{10-11-8}
\int_{\bn}u|u|^{p-1}\rho \, \dvg\leq p\int_{\bn}\sigma^{p-1}\rho^2 \, \dvg \, + \, C\bigg(\int_{\bn} |\sigma|^{p-2}|\rho|^3 \, \dvg\nonumber\\
+\int_{\bn}|\rho|^{p+1} \, \dvg+\sum_{1\leq i\neq k\leq N}\int_{\bn}|\rho|\calu_i^{p-1}\calu_k \, \dvg \bigg).
\end{align}
Using Proposition~\ref{p:spec} from the forthcoming section, we have
$$p\int_{\bn}\sigma^{p-1}\rho^2 \, \dvg\leq\tilde c\|\rho\|_{\la}^2, \quad\mbox{with }\, \tilde c = \tilde c(n,\lambda, N,p)<1.$$ 
Using H\"{o}lder and Poincar\'e-Sobolev inequality we estimate the other terms on the RHS of \eqref{10-11-8} as follows
$$\int_{\bn} |\sigma|^{p-2}|\rho|^3 \, \dvg\leq \|\rho\|_{L^{p+1}}^3\|\sigma\|_{L^{p+1}}^{p-2}\lesssim\|\rho\|_{\lambda}^3 ,$$
 $$\int_{\bn}|\rho|^{p+1} \, \dvg\lesssim\|\rho\|_{\lambda}^{p+1},$$
$$\int_{\bn}|\rho|\calu_i^{p-1}\calu_k\dvg\leq \|\rho\|_{L^{p+1}}\|\calu_i^{p-1}\calu_k\|_{L^{(p+1)'}}\lesssim\|\rho\|_{\lambda}\|\calu_i^{p-1}\calu_k\|_{L^{(p+1)'}}.$$
Substituting these estimates into \eqref{10-11-8} yields 
\begin{align}\label{10-11-9}
\Big|\int_{\bn}u|u|^{p-1}\rho \, \dvg\Big|&\leq \tilde c(n,\lambda, N)\|\rho\|_{\lambda}^2\nonumber\\
&\qquad+C\bigg(\|\rho\|_{\lambda}^3+\|\rho\|_{\lambda}^{p+1}+\sum_{1\leq i\neq j\leq N}
\|\rho\|_{\lambda}\|\calu_i^{p-1}\calu_k\|_{L^{(p+1)'}} \bigg).
\end{align}
Next, we control $\|\calu_i^{p-1}\calu_k\|_{L^{(p+1)'}}$ for $i\neq k$. Since  $p>2$ implies $\min\{(p-1)(p+1)', (p+1)'\}=(p+1)'$. Therefore, by Lemma~\ref{lem-int-2}, we find
\begin{align}\label{10-11-10}
\|\calu_i^{p-1}\calu_k\|_{L^{(p+1)'}}=\bigg(\int_{\bn}\calu_i^{(p-1)(p+1)'}\calu_k^{(p+1)'} \, \dvg \bigg)^\frac{1}{(p+1)'}&\approx \big(e^{-\cn(p+1)'d(z_i,z_k)}\big)^\frac{1}{(p+1)'}\nonumber\\
&=e^{-\cn d(z_i,z_k)}\nonumber\\
&\approx\int_{\bn}\calu_i^{p}\calu_k\,\dvg.
\end{align}
Moreover, using Lemma~\ref{interaction lemma} we get
\begin{equation}\label{10-11-11}\max_{i\neq k}\int_{\bn}\calu_i^p\calu_k\dvg \lesssim \|(\Delta_{\bn} + \lambda)u + |u|^{p-1}u\|_{H^{-1}} + \epsilon \|\rho\|_{\la} +\|\rho\|_{\la}^2.\end{equation}
Hence, combining \eqref{10-11-9}, \eqref{10-11-10} and \eqref{10-11-11}, we find
\begin{align}\label{10-11-12}
\int_{\bn}u|u|^{p-1}\rho \, \dvg\leq \bigg(\tilde c(n,\lambda, N,p)+\eps N^2 C_{n,\lambda,N,p}\tilde C \bigg)\|\rho\|_{\lambda}^2
+C'_{n,\lambda,N,p}\bigg(\|\rho\|_{\lambda}^3+\|\rho\|_{\lambda}^{p+1}\nonumber\\
\qquad\qquad\qquad+ \|\Delta_{\bn}u+\lambda u+u|u|^{p-1}\|_{H^{-1}(\bn)}\|\rho\|_{\lambda} \bigg),
\end{align}
where $\tilde C$ is the constant hidden in \eqref{10-11-11}. We choose $\eps>0$ small such that \\
$\tilde c(n,\lambda, N,p)+\eps N^2 C_{n,\lambda,N,p}\tilde C <1$. Hence, substituting \eqref{10-11-12} in \eqref{10-11-5} yields
\begin{align*}
\bigg(1-\tilde c(n,\lambda, N,p)-\eps N^2 C_{n,\lambda,N,p}\tilde C\bigg)\|\rho\|_{\lambda}^2 &\lesssim\|\rho\|_{\lambda}^3+\|\rho\|_{\lambda}^{p+1}\\
&\qquad+ \|\Delta_{\bn}u+\lambda u+u|u|^{p-1}\|_{H^{-1}(\bn)}\|\rho\|_{\lambda}.
\end{align*}
Since we can assume $\|\rho\|_{\lambda}<<1$, the last inequality implies
\begin{equation}\label{10-11-13}
\|\rho\|_{\lambda}\lesssim  \|\Delta_{\bn}u+\lambda u+u|u|^{p-1}\|_{H^{-1}(\bn)}.\end{equation}
So, now we are just left to show that the value of all the $\alpha_i$ can be replaced by $1$; note that thanks to \eqref{10-11-13}, this fact follows from 
Lemma~\ref{interaction lemma}. More precisely, $|\alpha_i-1|\lesssim  \|\Delta_{\bn}u+\lambda u+u|u|^{p-1}\|_{H^{-1}(\bn)}$, so it suffices to consider $\sigma'=\sum_{i=1}^N\calu_i$ to get that $\sigma'$ satisfies all the desired conditions. 

\hfill{$\square$}

\noindent

{\bf Proof of Corollary \ref{main corollary}}

\begin{proof}
First assume $p < \tstar-1.$ Then the result immediately follows from the profile decomposition Theorem \ref{profile decomposition}, Remark \ref{rem2} and Theorem \ref{main theorem}.

Now assume $p = \tstar -1,$ and let $\la \not\in \Lambda.$ We know that that $S_{\la}$ is not a rational multiple of $S$ and hence there exists $\delta  = \delta(n,N,\la)>0$ such that the interval $((N-\delta)S_{\la}, (N+\delta)S_{\la})$ does not contain an integer multiple of $S.$ 
Should the statement of the corollary be false there must exist a sequence $\{u_m\}$ in $H^1(\bn)$ such that $u_m \geq 0, I_{\la}(u_m) \rightarrow \frac{N}{n}S_{\la}^{\frac{n}{2}}$  and $\|\Delta_{\bn}u_m + \la u_m + u_m^p\|_{H^{-1}} \lesssim \frac{1}{m} \rightarrow 0$
as $m \rightarrow \infty.$ Then by Theorem \ref{absence theorem} and Theorem \ref{main theorem} the desired bound follows.
\end{proof}

\section{localization of family of bubbles}

In this section, we collect some prerequisites to prove the improved spectral inequality stated in the next section.
\begin{lemma}\label{l:lip func}
Let $n\geq 1$. Given a point $x_0\in\bn$ and radii $0<r<R$, there exists a Lipschitz bump function $\varphi=\varphi_{r,R}:\bn\to[0,1]$ such that
$\varphi=1$ on the geodesic ball $B(x_0, r)$, $\varphi=0$ in $B(x_0, R)^\complement$ and
$$\sup_{x\in\bn}|\nabla_{\bn}\varphi(x)|_{\bn}\lesssim\frac{1}{R-r}.$$
\end{lemma}
\begin{proof}
Without loss of generality, we can assume that $x_0=0$. As before let $\rho(x)=d(x,0)$ denote the geodesic distance of $x$ to $0$. Define

\begin{align*}
\varphi(x) =
\begin{cases}
1, \qquad\qquad\qquad\qquad \mbox{ if } \ \  0<\rho(x)<r, \\
\sin\bigg[\frac{\pi}{2}\big(\frac{\rho(x)+R-2r}{R-r}\big)\bigg] \ \ \ \mbox{if } \ \ r<\rho(x)<R, \\
0, \qquad\qquad\qquad\qquad \mbox{ if } \ \ \rho(x)\geq R. \\
\end{cases}
\end{align*}
Since $\varphi$ is radial, by abuse of notation we can write $\varphi(x)=\varphi(\rho)$. Therefore, it only remains to show that
$$\sup_{x\in\bn}|\nabla_{\bn}\varphi(x)|_{\bn}=\sup_{\rho}|\varphi'(\rho)|\lesssim\frac{1}{R-r},$$
where $'$ denotes the derivative w.r.t. $\rho$. Clearly, for $r<\rho(x)<R$
$$|\varphi'(\rho)|=\frac{\pi}{2(R-r)}\bigg|\cos\bigg[\frac{\pi}{2}\big(\frac{\rho(x)+R-2r}{R-r}\big)\bigg]\bigg|\lesssim \frac{1}{R-r},$$
and $\varphi'(\rho)=0$ elsewhere. This completes the proof.

\end{proof}
As a corollary of Lemma~\ref{l:lip func}, choosing $R$ large enough and $r=R/2$, it follows $$|\nabla_{\bn}\varphi(x)|_{\bn}=o(1) \quad\mbox{as}\quad R\to\infty.$$

\begin{lemma}\label{l:localization}
Assume $3 \leq n \leq 5$, $p \in (2,\tstar-1]$ and $\la$ satisfies {\bf (H1)}. For $N\in\N$ and $\eps>0$, there exists $\delta=\delta(n,N,\lambda,p,\eps)>0$ such that if $\calu_i=\calu[z_i]$, $i=1,\cdots, N$ are $\delta-$interacting hyperbolic bubbles, then for any $1\leq i\leq N$, there exists a Lipschitz bump functions $\varphi_i:\bn\to[0,1]$ such that the followings hold:

(i) Almost all mass of $\calu_i^{p+1}$ is in the region $\{\varphi_i=1\}$, i.e., 
$$\int_{\{\varphi_i=1\}}\calu_i^{p+1}\dvg\geq (1-\eps)S_{\lambda,p}^\frac{p+1}{p-1}.$$

(ii) In the region $\{\varphi_i>0\}$, there holds
$$\eps \calu_i>\calu_k \quad\mbox{for all }\, k\neq i.$$

(iii) The $L^\infty$ norm of $|\nabla_{\bn}\varphi_i|_{\bn}$ is small, i.e., 
$$\sup_{x \in \bn} \ |\nabla_{\bn}\varphi_i(x)|_{\bn}\leq \eps.$$
 \end{lemma}
 
\begin{proof}
Without loss of generality, we may assume $i=N$ and applying hyperbolic translation we can also assume $\calu_N=\calu[0]$. Recall that, 
$$\calu_N(x)\approx e^{-\cn d(x, 0)}, \quad \calu_k(x)\approx e^{-\cn d(x, z_k)}, \, k\neq N.$$
Let  $R>1$ be a large parameter (we will fix it later depending on $\epsilon$). 
$$\int_{B(0,R/2)^\complement}\calu_N^{p+1}\dvg\lesssim\int_{R/2}^{\infty}e^{-(p+1)\cn\rho}e^{(n-1)\rho}d\rho\leq C \, \frac{e^{-\big((p+1)\cn-(n-1)\big)R/2}}{(p+1)\cn-(n-1)},$$
as $(p+1)\cn>n-1$ (follows from the fact that $\cn>\frac{n-1}{2}$ and $p>1$).
We choose $R$ such that
$$C \, \frac{e^{-\big((p+1)\cn-(n-1)\big)R/2}}{(p+1)\cn-(n-1)}\leq\eps \, S_{\lambda,p}^\frac{p+1}{p-1}$$
i.e., $$R\ \geq \ \frac{2}{(p+1)\cn-(n-1)}\ln\left({\frac{C}{\big((p+1)\cn-(n-1)\big)\eps S_{\lambda,p}^\frac{p+1}{p-1}}}\right).$$
Let $\varphi$ be the Lipchitz bump function constructed in Lemma~\ref{l:lip func} with $r=R/2$, i.e., $\varphi=1$ in $B(0, R/2)$, $\varphi=0$ in $B(0, R)^\complement$ and 
$|\nabla_{\bn}\varphi|_{\bn} \lesssim 2/R$.
We further enlarge $R$ (in addition to the previous choice), if required, to ensure that $2/R<\eps$. Then, 
 $$\int_{\{\varphi=1\}}\calu_N^{p+1}\dvg=\int_{\bn}\calu_N^{p+1}\dvg-\int_{B(0,R/2)^\complement}\calu_N^{p+1}\dvg\geq S_{\lambda,p}^\frac{p+1}{p-1}-\eps S_{\lambda,p}^\frac{p+1}{p-1}=(1-\eps)S_{\lambda,p}^\frac{p+1}{p-1}.$$
Moreover, $\sup_{\bn}|\nabla_{\bn}\varphi(x)|_{\bn} \lesssim 2/R< \eps$. Hence (i) and (iii) hold. 

Now we will determine $\delta$. Note that $\{\varphi>0\}=B(0,R)$. If $x\in B(0,R)$, using the decay estimate of $\calu_N$ and $\calu_k$, $k\neq N$, it follows
$$\calu_N(x)\gtrsim e^{-\cn d(x,0)}\gtrsim e^{-\cn R}$$
$$\calu_k(x)\lesssim e^{-\cn d(x,z_k)}\lesssim e^{\cn R}e^{-\cn d(0, z_k)},$$
as $d(x,z_k)\geq d(0,z_k)-d(x,0)$. To hold $\eps\calu_N>\calu_k$, $k\neq N$ in $B(0,R)$, we need to show that 
$$\eps e^{-\cn R}\geq C_1e^{\cn R}e^{-\cn d(0, z_k)},$$
for some constant $C_1>0$. In other words, we need $d(0,z_k)\geq \frac{1}{\cn}\ln\big(\frac{C_1e^{2\cn R}}{\eps}\big)$. Since $z_i$'s are $\delta-$interacting bubbles, from the definition we already have 
$$\max_{ i \neq k } e^{-\cn d(z_i,z_k)} \leq \delta,$$
i.e., $d(z_i, z_k)\geq \frac{\ln\delta^{-1}}{\cn}$ for all $k\neq i$. So we choose $\delta>0$ such that
$$\ln\delta^{-1}>  \ln\Big(\frac{C_1e^{2\cn R}}{\eps}\Big),$$
then clearly $\calu_i$s are $\delta-$interacting bubbles imply, 
$$\eps\calu_N>\calu_k, \quad k\neq N, $$
on the set $\{\varphi>0\}$. Note that $\delta \lesssim \epsilon.$ Taking $\varphi_N = \varphi$ completes the proof. 
\end{proof}

\section{Improved spectral inequality}
\begin{proposition}\label{p:spec}
Let $3 \leq n \leq 5$, $p \in (2,\tstar-1]$ and $\la$ satisfies {\bf (H1)} and $N\in\N$. There exists a positive constant $\delta=\delta(n, N,\lambda,p)$ such that if $\sigma=\sum_{i=1}^N\alpha_i\calu[z_i]$ is a linear combination of $\delta-$interacting hyperbolic bubbles and $\rho\in H^1(\bn)$ satisfy the orthogonality conditions
\begin{equation}\label{eq:ortho}
\int_{\bn}\rho\,\calu_i^{p}\ \dvg=0 \quad\mbox{and }\,  \int_{\bn}\rho\,\calu_i^{p-1} V_j(\calu_i)\ \dvg=0 \quad \mbox{for all }\, 1\leq i\leq N, \, \, 1\leq j\leq n,
\end{equation}
where  $\calu_i=\calu[z_i]$.
Then
$$\int_{\bn}\sigma^{p-1}\rho^2\, \dvg\leq \frac{\tilde c}{p}\|\rho\|_{\la}^2,$$ where $\tilde c=\tilde c(n, N, \lambda,p)$ is a positive constant strictly less than 1. 

\end{proposition}
\begin{proof}
Fix a parameter $\eps>0$ and in this proof, we will denote with $o(1)$ any quantity that goes to $0$ as $\eps\to 0$. By Lemma~\ref{l:localization}, there exists $\delta$, $\varphi_i$ satisfying (i), (ii) and (iii) of Lemma~\ref{l:localization}.
Note that (i) implies the sets $\{\varphi_i=1\}$, $i=1,\cdots, N$ are disjoint as almost all the masses of $\calu_i$ is concentrated on $\{\varphi_i=1\}$. Moreover, (ii) implies $\varphi_i$'s have disjoint support.  Therefore,
$$\bn\subset \cup_{i=1}^N\Big\{\varphi_i=1\Big\}\cup\Big\{\sum_{i=1}^N\varphi_i<1\Big\}.$$
By Lemma \ref{l:localization}(ii) and $\delta$-interaction of $\calu_i$, on the set $\{\varphi_i>0\}$, 
$$\sigma^{p-1}=( \alpha_i\calu_i+\sum_{k\neq i} \alpha_k\calu_k)^{p-1}\leq (1+ C N\eps)^{p-1}\calu_i^{p-1}=(1+o(1))\calu_i^{p-1}.$$
Therefore, 
\begin{align}\label{eq:9-3}
\int_{\bn}\sigma^{p-1}\rho^2\, \dvg&\leq\sum_{i=1}^N\int_{\{\varphi_i=1\}}\sigma^{p-1}\rho^2\dvg+\int_{\{\sum_{i=1}^N\varphi_i<1\}}\
\sigma^{p-1}\rho^2\, \dvg\nonumber\\
&\leq (1+o(1))\sum_{i=1}^N\int_{\{\varphi_i=1\}}\calu_i^{p-1}\rho^2\dvg+\bigg(\int_{\{\sum_{i=1}^N\varphi_i<1\}}\sigma^{p+1}\dvg\bigg)^\frac{p-1}{p+1}\|\rho\|^2_{L^{p+1}}\nonumber\\
&\leq (1+o(1))\sum_{i=1}^N\int_{\bn}(\rho\varphi_i)^{2}\calu_i^{p-1}\dvg+\bigg(\sum_{i=1}^N\int_{\{\varphi_i<1\}}\calu_i^{p+1}\dvg\bigg)^\frac{p-1}{p+1}\|\rho\|^2_{\la}\nonumber\\
&\leq (1+o(1))\sum_{i=1}^N\int_{\bn}(\rho\varphi_i)^{2}\calu_i^{p-1}\dvg+\eps^\frac{p-1}{p+1}\|\rho\|^2_{\la}
\end{align}
{\bf Claim:} $\displaystyle\int_{\bn}(\rho\varphi_i)^{2}\calu_i^{p-1}\dvg\leq \frac{1}{\Lambda}\|(\rho\varphi_i)\|_{\la}^2+o(1)\|\rho\|^2_{\la} \quad\mbox{for }\,  i=1,\cdots, N$,\\
 where $\Lambda$ is the smallest eigenvalue of 
$\frac{-\Delta_{\bn}-\lambda}{\calu_i^{p-1}}$, which is
bigger than $p$. 

To prove the claim, let $\psi$ be either $\calu_i$ or $V_j(\calu_i)$, $1\leq j\leq n$ normalized so that\\ $\int_{\bn}\psi^2\calu_i^{p-1}\dvg=1$. Hence, thanks to the orthogonality condition \eqref{eq:ortho}, it holds that
\begin{align*}
\bigg|\int_{\bn}\rho\,\varphi_i\,\psi\,\calu_i^{p-1}\,\dvg\bigg|&=\bigg|\int_{\bn}\rho\,\psi\,\calu_i^{p-1}(1-\varphi_i)\,\dvg\bigg|\leq \bigg|\int_{\{\varphi_i<1\}}\rho\,\psi\,\calu_i^{p-1}\,\dvg\bigg|\\
&\leq \bigg(\int_{\bn}|\rho|^{p+1}\dvg\bigg)^\frac{1}{p+1}\bigg(\int_{\bn}\psi^2\calu_i^{p-1}\dvg\bigg)^\frac{1}{2} 
\bigg(\int_{\{\varphi_i<1\}}\calu_i^{p+1}\dvg\bigg)^\frac{p-1}{2(p+1)}\\
&\leq o(1)\|\rho\|_{\la},
\end{align*}
where in the last inequality we have applied Lemma~\ref{l:localization}(i). 

This proves that $\rho\varphi_i$ is almost orthogonal to $\psi$. Hence, using \cite[Proposition 3.1]{BGKM}, it follows that if $\Lambda$ is the smallest eigenvalue of 
$(-\Delta_{\bn}-\lambda)/\calu_i^{p-1}$
bigger than $p$ then
\begin{equation}\label{eq:9-2}
\displaystyle\int_{\bn}(\rho\varphi_i)^{2}\calu_i^{p-1}\dvg\leq \frac{1}{\Lambda}\|(\rho\varphi_i)\|_{\la}^2+o(1)\|\rho\|^2_{\la}.
\end{equation}
This proves the claim. 

We now estimate RHS of \eqref{eq:9-2}.
\begin{equation}\label{eq:9-1}
\int_{\bn}|\nabla_{\bn}(\rho\varphi_i)|^2\, \dvg=\int_{\bn}|\nabla_{\bn}\rho|^2\varphi_i^2\, \dvg+\int_{\bn}\rho^2|\nabla_{\bn}\varphi_i|^2\, \dvg+
2\int_{\bn}\rho\varphi_i\nabla_{\bn}\rho\cdot\nabla_{\bn}\varphi_i\,\dvg.
\end{equation}
Using Lemma~\ref{l:localization}(iii) and Poincar\'e inequality,  $$\int_{\bn}\rho^2|\nabla_{\bn}\varphi_i|^2\, \dvg\leq \eps^2\int_{\bn}\rho^2\, \dvg=o(1)\|\rho\|^2_{\la}.$$
Next, 
\begin{align*}
\int_{\bn}\rho\varphi_i\nabla_{\bn}\rho\cdot\nabla_{\bn}\varphi_i\,\dvg&\leq \bigg(\int_{\bn}\rho^2|\nabla_{\bn}\varphi_i|^2 \dvg\bigg)^\frac{1}{2} \bigg(\int_{\bn}|\nabla_{\bn}\rho|^2\dvg\bigg)^\frac{1}{2}\\
&\leq\eps\|\rho\|_{\la}\|\nabla_{\bn}\rho\|_{L^2}\lesssim o(1)\|\rho\|_{\la}^2.
\end{align*}
Inserting the above two estimates into \eqref{eq:9-1}, we find
$$\|(\rho\varphi_i)\|_{\la}^2=\int_{\bn}\bigg(|\nabla_{\bn}(\rho\varphi_i)|^2-\la(\rho\varphi_i)^2\bigg)\, \dvg\leq \int_{\bn}\bigg(|\nabla_{\bn}\rho|^2-\la\rho^2\bigg)\varphi_i^2\, \dvg+o(1)\|\rho\|_{\la}^2.$$
Taking a sum of $i=1$ to $N$ and using the fact that $\varphi_i$s have disjoint support, it follows 
\begin{equation}\label{eq:9-4}
\sum_{i=1}^N\|(\rho\varphi_i)\|_{\la}^2\, \dvg\leq (1+o(1))\|\rho\|_{\la}^2.
\end{equation}
Thus combining \eqref{eq:9-3}, \eqref{eq:9-2} and \eqref{eq:9-4}, we deduce 
$$
\int_{\bn}\sigma^{p-1}\rho^2\, \dvg\leq \frac{1+o(1)}{\Lambda}\|\rho\|_{\la}^2\leq \frac{\tilde c}{p}\|\rho\|_{\la}^2,$$
with $\tilde c<1$ (since $\Lambda>p$).
\end{proof}

\section{Interaction integral estimates}

In this section, we prove the following interaction integral estimate:

\begin{lemma} \label{interaction lemma}
Let $3 \leq n \leq 5, \ N \in \mathbb{N}, p \in (2, \tstar-1]$ and $\lambda$ be as in {\bf (H1)}. For every $\epsilon>0,$ there exists a $\delta>0$
depending on $n,p,N,\lambda,\epsilon$ such that if $(\alpha_i,\calu[z_i])_{1 \leq i \leq N}$ is a family of $\delta$-interacting hyperbolic bubbles satisfying the orthogonality conditions \eqref{10-11-3} and \eqref{10-11-4} then setting $u = \sum_{i=1}^N \alpha_i\calu[z_i] + \rho$ the following estimates hold:
\begin{align*}
\max_{i\neq k}\int_{\bn}\calu_i^p\calu_k \ \dvg \lesssim \|(\Delta_{\bn} + \lambda)u + |u|^{p-1}u\|_{H^{-1}} + \epsilon \|\rho\|_{\la} + \|\rho\|_{\la}^2,
\end{align*}
and 
\begin{align*}
\max_{i}|\alpha_i-1|\lesssim \|(\Delta_{\bn} + \lambda)u + |u|^{p-1}u\|_{H^{-1}} + \epsilon \|\rho\|_{\la} +\|\rho\|_{\la}^2.
\end{align*}
\end{lemma}

\subsection{A geometric lemma} To prove the above lemma we take the help of the following geometrically intuitive lemma.

\begin{lemma} \label{geometric lemma}
Let $\calu_i = \calu[z_i], \ 1 \leq i \leq N$ be a family of $\delta$-interacting hyperbolic bubbles. Then there exists $\delta_0>0$ such that for every $\delta < \delta_0,$ there exists an indices $k \in \{1,\ldots, N\}$ and a direction $e_j$ such that after applications of hyperbolic translations and rotations, the followings hold:
\begin{itemize}
\item[(i)] $ z_k = 0,$
\item[(ii)] all $z_i, i \neq k$ lie in the negative half \ $\bn_- := \{x \in \bn \ | \ x_j < 0\}.$ 
\item[(iii)] There exists a constant $\kappa = \kappa(\delta) >0$ such that $(z_i)_j$ (the $j$-th component if $z_i$) satisfies 
$(z_i)_j \leq -\kappa$ for all $i \neq k.$
\end{itemize}
\end{lemma}

\begin{proof}
Since there are only $N$-many indices we can find two indices $k_1, k_2 \in \{1,\ldots, N\}$ such that 
\begin{align*}
d(z_{k_1}, z_{k_2}) = \max_{ i \neq k } \ d(z_i,z_k).
\end{align*}
There may be other choices of such indices, but we only need a pair of them.
By rearranging the indices and by applying a hyperbolic translation we may assume $z_1 = 0$ and $d(z_1, z_N) =  \max_{ i \neq k } d(z_i,z_k).$ Further applying an orthogonal transformation we may assume $z_N = |z_N|e_j.$
Recall that the hyperbolic distances are preserved under the orthogonal group of transformations.

\medskip

We will show that if $\calu[z_i]$ are sufficiently small $\delta$-interacting then $\tau_{-z_N}(z_i), 1 \leq i \leq N$ satisfies all the properties of the stated lemma.

\medskip

For the simplicity of notations we denote $z_N = z.$ Recall that $\tau_{-z} = \sigma_{-z^*}\rho_{-z^*}$ is the composition of reflection $\rho_{-z*}$ with respect to the plane $\{x \in \rn \ | \  x \cdot z^* = 0\}$ and the inversion with respect to the sphere 
$S(-z^*, r)$ where $r$ is determined by $r^2 = |z^*|^2-1 = \frac{1}{|z|^2} - 1.$ We divide $\bn$ into the regions $I_1,I_2$ and $I_3$ defined below (see Figure $3$):

\begin{figure}[H]
\centering 
\begin{tikzpicture}[ scale=2.5]
  \draw[thick] (0,0) circle (1) node[near end, anchor=north west]{$I_2$};
  \draw[very thick] (1.25,0) circle (0.5625)node[near end, anchor=south east]{$I_2$};
  \draw[very thick] (-1.25,0) circle (0.5625);
  \draw[thick] (0,1)--(0,-1);
  \filldraw[black] (1.25,0) circle (1pt) node[anchor=south]{z*};
   \filldraw[black] (-1.25,0) circle (1pt) node[anchor=south]{-z*};
    \filldraw[black] (0.8,0) circle (1pt) node[anchor=south]{z} node[ anchor=north west]{$I_1$};
    \filldraw[black] (-0.8,0) circle (1pt) node[anchor=south]{-z} node[ anchor=north east]{$I_3$};
   \filldraw[black] (0,0) circle (1pt);
   \draw[dashed, <->] (-2.5,0)--(2.5,0);
 \end{tikzpicture}
 \caption{The above picture demonstrating the regions $I_1,I_2$ and $I_3.$ The thick boundary portion of $B_E(z^*,r)$ in $\bn$ is included in $I_1,$ and similarly for $I_3.$ $I_2$ is the open region in $\bn$ complementing $I_1 \cup I_2$. All the $z_i$s which lie outside $I_1$ will map under the hyperbolic translation to the region $I_3.$ The boundary portion of $I_1$ will map to the boundary portion of $I_3.$} 
\end{figure}
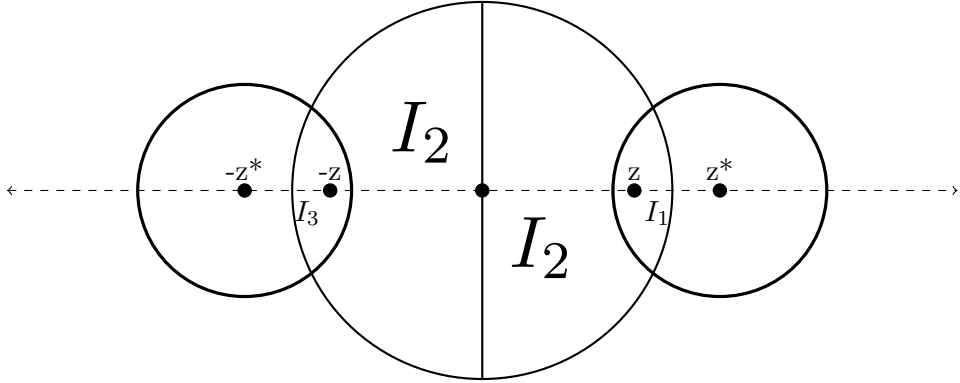
\begin{align*}
 I_1 = \bn \cap \overline {B_E(z^*,r)} \ \ \
 I_3 = \bn \cap \overline {B_E(-z^*,r)} \ \ \ \mbox{and} \ \ \
 I_2 = \bn \backslash (I_1 \cup I_3).
\end{align*}
The following observations are easy to verify:
\begin{itemize}
\item[(a)] $\sigma_{-z^*}$ leaves $S(-z^*, r)$ invariant.
\item[(b)] $z \in B_E(z^*, r)$ as direct computation shows $|z-z^*|^2 = r^2(1-|z|^2) < r^2.$ 
\item[(c)] $\rho_{-z^*}(I_1) = I_3, \rho_{-z^*}(I_3) = I_1$ and $\rho_{-z^*}$ leaves $I_2$ invariant.
\item[(d)] $\sigma_{-z^*}(I_1 \cup I_2) \subset I_3$ and $\sigma_{-z^*}(I_3 \backslash S(-z^*,r)) \subset I_1 \cup I_2.$
\item[(e)] $\sigma_{-z^*}$ sends spheres $S(-z^*, r_1)$ to spheres $S(-z^*, \frac{r^2}{r_1}).$
\end{itemize}

According to our choice all  $z_i$'s lie with in $\overline{B_E(0,|z|)}.$
The above observation indicates all points $z_i$ such that $z_i \not \in \overline{B_E(0,|z|)} \cap B_E(z^*, r)$ would translate  
under the hyperbolic translation $\tau_{-z}$ to the region $I_3.$ Clearly $\tau_{-z}(z) = 0.$ Now let   $z_i \in \overline{B_E(0,|z|)} \cap B_E(z^*, r).$ Then $\rho_{-z^*}(z_i) \in B_E(-z^*,r) \cap \overline{B_E(0,|z|)}.$ Let us now look at the action of $\sigma_{-z^*}$ on the set $B_E(-z^*,r) \cap \overline{B_E(0,|z|)}.$

\medskip

According to $(e),$ $\sigma_{-z^*}$ maps $B_E(-z^*,r) \cap \overline{B_E(0,|z|)}$ to the region $\overline{B_E(-z^*, \frac{1}{|z|})}.$ To see that we let $r_0 = |z - z^*| = |z|r^2.$ Then by $(e),$ $\sigma_{-z^*}$ maps $S(-z^*, r_0)$ to $S(-z^*, \frac{1}{|z|})$ and leaves $S(-z^*, r)$ invariant. Since all the $\rho_{-z^*}(z_i)$s are trapped with in the spheres $S(-z^*, r)$ and 
 $S(-z^*,r_0)$ we conclude $\tau_{-z}(z_i), i \neq N$ would lie with in $\overline{B(-z^*, \frac{1}{|z|})} \backslash \{0\}$ (see Figure $4$). This proves $(i)$ and $(ii)$.
 
 \medskip
 
 \begin{figure}[H]\centering
\begin{tikzpicture}[ scale=2.5]
  \draw[thick] (0,0) circle (1) node[anchor=north west]{0};
   \draw[thin] (0,0) circle (0.8);
   \draw[thick] (1.25,0) circle (0.5625) ;
  \draw[thick] (-1.25,0) circle (0.5625);
  \draw[thick] (0,1)--(0,-1);
  \draw[thick] (0.75,0)--(0.75,-1)--(1.1,-1) node[anchor=west]{$\mathcal{R}_0=B_E(0,|z|) \cap B_E(z^*, r)$};
   \draw[ultra thick] (-1.25,0) circle (1.25);
   \filldraw[black] (1.25,0) circle (1pt) node[anchor=south]{z*};
   \filldraw[black] (-.5,-.5) circle (1pt) node[anchor=south]{$x$};
   \filldraw[black] (-.5,-.5)--(0,-.5) node[scale = 0.5, anchor=north east]{distance at least $\kappa$};
   \filldraw[black] (-1.25,0) circle (1pt) node[anchor=south]{-z*};
   \filldraw[black] (0,0) circle (1pt);
    \draw[dashed, <->] (-3,0) node[anchor=east]{$-e_j$} --(2.5,0) node[anchor=west]{$e_j$};
     \filldraw[black] (0.8,0) circle (0.5 pt) node[anchor=south west]{z};
      \filldraw[black] (-0.8,0) circle (0.5 pt) node[anchor=south east]{-z};
 \end{tikzpicture}
 \caption{Under the action of hyperbolic translation $\tau_{-z},$ $z$ will map to $0$ and all other $z_i$s which lie in the region $\mathcal{R}_0$ will map to the largest ball centered at $-z^*$ and radius $\frac{1}{|z|}.$ After the translation if $x$ denotes the point $\tau_{-z}(z_i)$ then $x$ is at least $\kappa$ distance apart from the plane $\{x_j=0\}$.}
\end{figure}
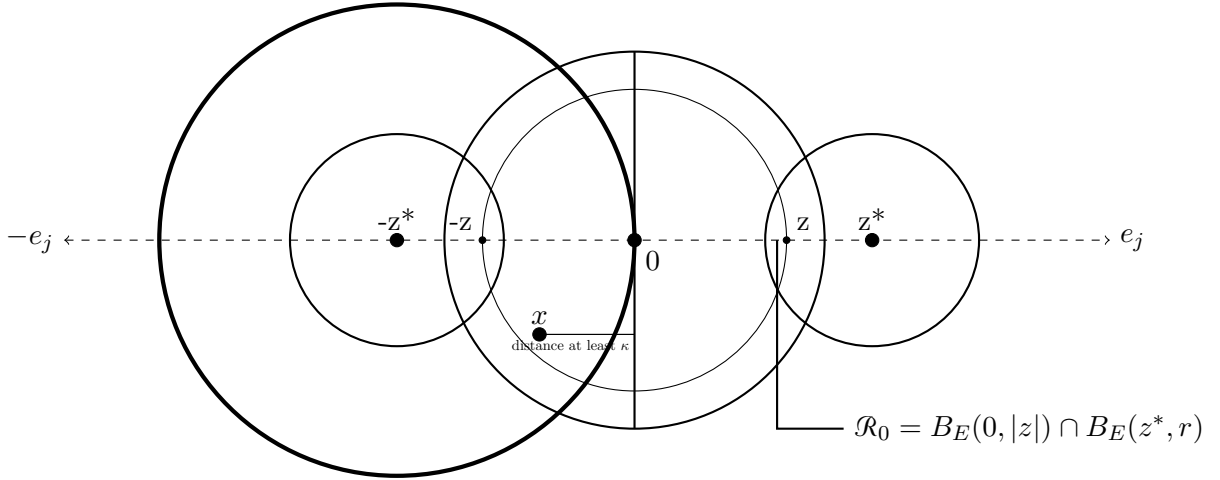

 \medskip
 
 Now we claim that if $\calu[z_i], 1 \leq i \leq N$ are sufficiently small $\delta$-interacting then $(iii)$ holds (cf. Figure $4$). 
 Fix $i \neq N$, for simplicity we denote $x = \tau_{-z}(z_i)$ and we write $x = (x^{\prime}, x_j)$ to indicate the $j$-th component. The above discussions yields  $x \in \overline{B(-z^*, \frac{1}{|z|})} \backslash \{0\}$ i.e,
 \begin{align} \label{qqq}
 \left|x_j + \frac{z_j}{|z|^2}\right|^2 + |x^{\prime}|^2 \leq \frac{1}{|z|^2}.
  \end{align}
If $\delta$ is small then $d(z_i,z_N) = d(x,0) = \ln \left(\frac{1+|x|}{1-|x|}\right)$ is sufficiently large. That is there exists small $\eta = \eta(\delta) >0$  such that $|x| > 1-\eta.$ Hence from \eqref{qqq} denoting $t = \frac{1}{|z|^2}$ we deduce 
\begin{align*}
& (x_j + tz_j)^2 + |x^{\prime}|^2 = |x|^2 + 2tx_jz_j + t^2z_j^2 \leq t \\
i.e. \ & x_j \leq \frac{t-t^2z_j^2-|x|^2}{2tz_j} \leq -\frac{(1-\eta)^3}{2}=: -\kappa,
\end{align*}
where we have used $z_j =|z| > 1-\eta$ as $d(0,z) = d(z_1,z_N) > > 1.$ This completes the proof.
\end{proof}

\begin{remark}
The point $(ii)$ is redundant given point $(iii)$ of the above lemma. We wanted to emphasise that $(i)$ and $(ii)$ are the properties of the hyperbolic translation while $(iii)$ is the property of the $\delta$-interaction.
\end{remark}

We are now in a position to prove the interaction integral estimates Lemma \ref{interaction lemma}.

\subsection{The proof of interaction integral estimates}
Given $u$ and $(\alpha_i, \calu[z_i])_{i=1}^N $ a family of $\delta$-interacting hyperbolic bubbles as in Lemma \ref{interaction lemma}. We put
\begin{align}
u = \sum_{i=1}^N \alpha_i\calu_i+ \rho := \sigma + \rho,
\end{align}
where $\calu_i = \calu[z_i]$ and $\sigma = \sum_{i=1}^N \alpha_i\calu_i.$ Recall that  the following orthogonality conditions hold
\begin{align}
&\int_{\bn} (\Delta_{\bn} + \lambda) \calu_i \rho \ \dvg= \int_{\bn} \calu_i^p \rho \ \dvg= 0 \label{org1}\\
&\int_{\bn} (\Delta_{\bn} + \lambda) V_j(\calu_i) \rho\ \dvg = p\int_{\bn} \calu_i^{p-1} V_j(\calu_i) \rho \ \dvg = 0, \ \ \mbox{for} \ 1 \leq j \leq n \label{org2},
\end{align}
for all $i = 1,\ldots, N.$ We compute
\begin{align}\label{eqn for u}
\Delta_{\bn} u +\lambda u + |u|^{p-1}u &= \Delta_{\bn} \rho + \lambda \rho -\sum_{i=1}^N \alpha_i\calu_i^p + |\sigma + \rho|^{p-1}(\sigma + \rho).
\end{align}

Putting $f = -\Delta_{\bn} u -\lambda u - |u|^{p-1}u $ we can rewrite the equation \eqref{eqn for u} as 
\begin{align} \label{main equation}
\Delta_{\bn} \rho + \lambda \rho + p\sigma^{p-1}\rho + I_1 + I_2 + I_3 + f = 0,
\end{align}
where 
\begin{align}
I_1 = \sigma^p - \sum_{i=1}^N \alpha_i^p\calu_i^p, \ \ I_2 =  |\sigma + \rho|^{p-1}(\sigma + \rho) - \sigma^p- p\sigma^{p-1}\rho, \ \ I_3 =  \sum_{i=1}^N (\alpha_i^p-\alpha_i)\calu_i^p
\end{align}

We follow the approach of \cite{Weietal} and divide the proof of Lemma \ref{interaction lemma} into several small lemmas. We point out that the next four lemmas hold in any dimension $n \geq 3.$

\begin{lemma}\label{ie1}
 We have
\begin{align} 
\int_{\bn} I_1 V_j(\calu_k) \ \dvg= \sum_{i \neq k} \alpha_k^{p-1}\alpha_i \int_{\bn} \calu_i^{p}V_j(\calu_k) \ \dvg + o(Q).
\end{align}
for all $1 \leq k \leq N$ and $j = 1,\ldots,n.$
\end{lemma}

\begin{proof}
We denote $\sigma_k = \sum_{i \neq k} \alpha _i \calu_i$ so that $\sigma = \sigma_k + \alpha_k \calu_k$ holds. We decompose the integral $\int_{\bn} I_1 V_j(\calu_k) \ \dvg$ into the following four integrals:
\begin{align*}
&J_1 := \int_{\{N\alpha_k\calu_k \geq \sigma_k\}} \left(p \alpha_k^{p-1}\calu_k^{p-1}\sigma_k - \sum_{i \neq k}\alpha_i^p\calu_i^p\right) V_j(\calu_k) \ \dvg \\
&J_2 = \int_{\{N\alpha_k\calu_k \geq \sigma_k\}} \left(\sigma^p - \alpha_k^p\calu_k^p - p \alpha_k^{p-1}\calu_k^{p-1}\sigma_k\right) V_j(\calu_k) \ \dvg \\
&J_3 = \int_{\{\sigma_k > N\alpha_k\calu_k\}}\left(p \alpha_k\calu_k\sigma_k^{p-1} + \sigma_k^p - \sum_{i=1}^N \alpha_i^p\calu_i^p\right)V_j(\calu_k) \ \dvg \\
&J_4 = \int_{\{\sigma_k > N\alpha_k\calu_k\}} \left(\sigma^p - \sigma_k^p - p\alpha_k\calu_k\sigma_k^{p-1}\right)V_j(\calu_k)\ \dvg.
\end{align*}
We now estimate each integral one by one. Using $|V_j(\calu_k)| \lesssim \calu_k$ and using the inequality $|(a+b)^p - a^p - pa^{p-1}b| \lesssim a^{p-2}b^2$ whenever $0\leq b \lesssim a$ we estimate
\begin{align}\label{J2}
|J_2| &\lesssim \int_{\{N\alpha_k\calu_k \geq \sigma_k\}} \calu_k^{p-2}\sigma_k^2\calu_k \ \dvg \notag \\
&\lesssim \int_{\{N\alpha_k\calu_k \geq \sigma_k\}} \calu_k^{p-\epsilon} \sigma_k^{1+\epsilon} \ \dvg \notag \\
&\lesssim Q_k^{1 + \epsilon}.
\end{align}
Similarly,
\begin{align}\label{J4}
|J_4| &\lesssim \int_{\{\sigma_k > N\alpha_k\calu_k\}} \sigma_k^{p-2}\calu_k^2\calu_k \ \dvg \notag \\
&\lesssim  \int_{\{\sigma_k > N\alpha_k\calu_k\}} \sigma_k^{p-\epsilon}\calu_k^{1+\epsilon} \ \dvg \notag \\
&\lesssim Q_k^{1+\epsilon}.
\end{align}
For $J_3$ we approach as follows.
\begin{align}\label{J3_1}
|J_3| &\lesssim \int_{\{\sigma_k > N\alpha_k\calu_k\}} \sigma_k^{p-1}\calu_k^2 \ \dvg + 
\int_{\{\sigma_k > N\alpha_k\calu_k\}}\Big|\sigma_k^p - \sum_{i \neq k} \alpha_i^p\calu_i^p\Big| \ \calu_k \ \dvg \notag \\
& \ + \int_{\{\sigma_k > N\alpha_k\calu_k\}} \calu_k^{p+1} \ \dvg.
\end{align}
The first term in \eqref{J3_1} can be estimated as in $(J2),(J4)$ to obtain $O (Q_k^{1+\epsilon}).$
For the second term we use the inequality $|(\sum_i a_i)^p - \sum_i a_i^p| \lesssim \sum_{i \neq j} |a_i|^{p-1}\min\{|a_i|,|a_j|\}$ to get
\begin{align*}
\int_{\{\sigma_k > N\alpha_k\calu_k\}}\Big|\sigma_k^p - \sum_{i \neq k} \alpha_i^p\calu_i^p\Big|\ \calu_k \ \dvg &\lesssim \sum_{i < j, i \neq k, j \neq k}\int_{\{\sigma_k > N\alpha_k\calu_k\}}\calu_i^{p-1}\calu_j\calu_k\notag \\
&\lesssim Q^{1+\epsilon},
\end{align*}
where in the last line we used Lemma \ref{three bubble}.
For the last term in \eqref{J3_1} we see that $\{\sigma_k > N\alpha_k\calu_k\} \subset 
\cup_{i \neq k} \{\alpha_i \calu_i > \alpha_k \calu_k\}.$ Hence
\begin{align*}
\int_{\{\sigma_k > N\alpha_k\calu_k\}} \calu_k^{p+1} \ \dvg &\leq \sum_{i \neq k}
\int_{\{\alpha_i\calu_i > \alpha_k\calu_k\}} \calu_k^{p+1} \ \dvg \\
&\lesssim \int_{\{\alpha_i\calu_i > \alpha_k\calu_k\}} \calu_k^{p-\epsilon}\calu_i^{1+\epsilon} \ \dvg \lesssim Q^{1+\epsilon}.
\end{align*}
Now for the estimation of $J_1$, we see that 
\begin{align}\label{J1_1}
&J_1 - \int_{\bn}p\alpha_k^{p-1}\calu_k^{p-1}\sigma_kV_j(\calu_k) \ \dvg \notag \\
 =& -\int_{\{N\alpha_k\calu_k < \sigma_k\}} p\alpha_k^{p-1}\calu_k^{p-1}\sigma_kV_j(\calu_k) \ \dvg - \sum_{i \neq k}\int_{\{N\alpha_k\calu_k \geq \sigma_k\}}\alpha_i^p\calu_i^pV_j(\calu_k)
\ \dvg = J_1^1+J_2^2.
\end{align}
$J_1^1, J_2^2$ can be estimated in a similar way
\begin{align*}
\Big|J_1^1\Big| \lesssim \int_{\{N\alpha_k\calu_k < \sigma_k\}} \calu_k^p\sigma_k \ \dvg 
\lesssim \int_{\{N\alpha_k\calu_k < \sigma_k\}} \calu_k^{p-\epsilon}\sigma_k^{1+\epsilon} \ \dvg \lesssim Q^{1+\epsilon},
\end{align*}
and since on $\{\sigma_k \leq N\alpha_k\calu_k\},$  $\calu_i \lesssim \calu_k$ holds for all $i \neq k$
\begin{align*}
\Big|J_1^2\Big| \lesssim \sum_{i \neq k}\int_{\{ \sigma_k \leq N\alpha_k\calu_k \}}\calu_i^p\calu_k \ \dvg \lesssim \sum_{i \neq k}\int_{\{ \sigma_k \leq N\alpha_k\calu_k \}}\calu_i^{p-\epsilon}\calu_k^{1+\epsilon} \ \dvg \lesssim Q^{1+\epsilon}. 
\end{align*}
Combining all the estimates and integrating by parts gives the result as indicated below
\begin{align*}
\int_{\bn} I_1 V_j(\calu_k) \ \dvg = \sum_{l=1}^4 J_l &= \int_{\bn}p\alpha_k^{p-1}\calu_k^{p-1}\sigma_kV_j(\calu_k) \ \dvg + o(Q) \\
&= \sum_{i\neq k}\alpha_i \int_{\bn}p\alpha_k^{p-1}\calu_k^{p-1}\calu_iV_j(\calu_k) \ \dvg + o(Q) \\
&= \sum_{i\neq k} \alpha_k^{p-1}\alpha_i\int_{\bn}\calu_i(-\Delta_{\bn} - \lambda)V_j(\calu_k) \ \dvg + o(Q)\\
&=\sum_{i\neq k} \alpha_k^{p-1}\alpha_i\int_{\bn}(-\Delta_{\bn} - \lambda)\calu_i \ V_j(\calu_k) \ \dvg + o(Q)\\
&=\sum_{i\neq k} \alpha_k^{p-1}\alpha_i\int_{\bn}\calu_i^p \ V_j(\calu_k) \ \dvg  + o(Q).
\end{align*}
This completes the proof.

\end{proof}

\begin{lemma}\label{ie2}
Denoting  $\xi =$ either $\calu_k$ or $V_j(\calu_k)$ for $1 \leq k \leq N,$ the following estimate holds.
\begin{align}
\Big|\int_{\bn} I_2 \xi \ \dvg \Big|  \lesssim
\|\rho\|_{\la}^{\min\{p,2\}}.
\end{align}
\end{lemma}

\begin{proof}
We again use the elementary inequality $|(a+b)|a+b|^{p-1} - a^p - p\sigma^{p-1}b| \leq |b|^p + |a|^{p-2}|b|^2$ and $|\xi| \lesssim \calu_k$ 
\begin{align*}
\Big|\int_{\bn} I_2\xi \ \dvg \Big| &\lesssim \int_{\bn}|\rho|^{p}\calu_k \ \dvg + \int_{\bn} \sigma^{p-2}|\rho|^2\calu_k \ \dvg \\
&\lesssim \|\rho\|_{L^{p+1}}^p\|\calu_k\|_{L^{p+1}} + \|\sigma\|_{L^{p+1}}^{p-2} \|\rho\|_{L^{p+1}}^2\|\calu_k\|_{L^{p+1}} \\
&\lesssim \|\nabla_{\bn} \rho\|_{L^2}^p + \|\nabla_{\bn} \rho\|_{L^2}^2 \\
&\lesssim   \| \rho\|_{\la}^p + \| \rho\|_{\la}^2\\
&\lesssim \|\rho\|_{\la}^{\min\{p,2\}} \quad\text{(as}\quad \|\rho\|_{\la} \mbox{ is small)}.
\end{align*}
This completes the proof.
\end{proof}

\begin{lemma} \label{ie3}
We have
\begin{align}
\int_{\bn} I_3 V_j(\calu_k) \ \dvg= \sum_{i \neq k} (\alpha_i^p-\alpha_i) \int_{\bn} \calu_i^{p}V_j(\calu_k) \ \dvg
\end{align}
for all $1 \leq k \leq N$ and $j = 1,\ldots,n.$
\end{lemma}

\begin{proof}
We note that 
\begin{align*}
(p+1)\int_{\bn} \calu_k^p V_j(\calu_k) \ \dvg&= (p+1) \int_{\bn} \calu_k^{p}\frac{d}{dt}\Big|_{t=0}(\calu_k\circ \tau_{te_j}) \ \dvg\\
&=\frac{d}{dt}\Big|_{t=0}\int_{\bn} (\calu_k\circ \tau_{te_j})^{p+1} \ \dvg = 0.
\end{align*}
Hence only the terms $i \neq k$ will survive 
\begin{align*}
\int_{\bn} I_3 V_j(\calu_k) \ \dvg= \sum_{i =1}^N (\alpha_i^p-\alpha_i) \int_{\bn} \calu_i^{p}V_j(\calu_k) \ \dvg = \sum_{i \neq k} (\alpha_i^p-\alpha_i) \int_{\bn} \calu_i^{p}V_j(\calu_k) \ \dvg.
\end{align*}

\end{proof}

\begin{lemma} \label{ie4}
Given $\epsilon >0$ there exists a $\delta > 0$ depending on $n,N,p,\la,\epsilon$ such that if $(\alpha_i, \calu_i)_{i=1}^N$ is a family of $\delta$-interacting hyperbolic bubbles satisfying the orthogonality conditions \eqref{org1},\eqref{org2} then
\begin{align}
\Big| \int_{\bn} p\sigma^{p-1} \rho \xi \ \dvg\Big| \lesssim (Q_k^{\min\{p-1,1\}} + \epsilon)\|\rho\|_{\la},
\end{align}
 where $\xi = $ either $\calu_k$ of $V_j(\calu_k),$ for every $1 \leq k \leq N, j = 1, \ldots, n,$ and $Q_k$ is defined by $Q_k = \max_{i \neq k} Q_{ik}.$ 
\end{lemma}

\begin{proof}
Given $\epsilon >0,$ we use the following elementary inequality
$|(a+b)^{p-1} - a^{p-1}| \leq C(\epsilon)b^{p-1} + \epsilon a^{p-1}$ for every $a,b>0, p>1.$ Using the orthogonality conditions and applying the above elementary inequality with $a = \alpha_k\calu_k$ and $b =  \sum_{i \neq k} \alpha_i\calu_i$ so that $a+b = \sigma$ we get
\begin{align*}
\Big| \int_{\bn}p \sigma^{p-1}\rho \xi \ \dvg\Big| &= \Big| \int_{\bn}p (\sigma^{p-1}-\alpha_k^{p-1}\calu_k^{p-1})\rho \xi \ \dvg\Big| \\
&\lesssim \sum_{i \neq k} \int_{\bn}\calu_i^{p-1}\calu_k|\rho| \ \dvg + \epsilon \int_{\bn} \calu_k^p|\rho| \ \dvg\\
&\lesssim \sum_{i\neq k} \|\calu_i^{p-1} \calu_k\|_{L^{\frac{p+1}{p}}}\|\rho\|_{L^{p+1}} + 
\epsilon \|\calu_k\|_{L^{p+1}}^p\|\rho\|_{L^{p+1}}\\
&\lesssim (Q_k^{\min\{p-1,1\}} + \epsilon)\|\rho\|_{\la}.
\end{align*}
This completes the proof.
\end{proof}

\medskip

\noindent
{\bf Proof of Lemma \ref{interaction lemma}.}
\begin{proof}
We proceed inductively. First note that in dimension $3 \leq n \leq 5$, $\tstar - 2 >1,$ and recall our assumption is $2 < p \leq \tstar-1,$ so that in either case $p-1>1.$ Let $k$ be the indices and $e_j$ be the direction as obtained in Lemma \ref{geometric lemma}. We test \eqref{main equation} against $V_j(\calu_k)$ and integrate and use the orthogonality conditions to get
\begin{align*}
 &\int_{\bn} I_1 V_j(\calu_k) \ \dvg + \int_{\bn}  I_3 V_j(\calu_k) \ \dvg \\ = \ &  -\int_{\bn} p \sigma^{p-1 }\rho V_j(\calu_k) \ \dvg-\int_{\bn} I_2 V_j(\calu_k) \ \dvg  - \int_{\bn} f V_j(\calu_k) \ \dvg .
\end{align*}
By Lemma~\ref{ie1}-\ref{ie4} and using $\min\{p-1,1\} = 1$ we get
\begin{align} \label{ieeqn1}
&\sum_{i \neq k}(\alpha_k^{p-1}\alpha_i + \alpha_i^p - \alpha_i) \int_{\bn} \calu_i^p V_j(\calu_k) \ \dvg \nonumber\\
\lesssim \ \ &\|f\|_{H^{-1}} + (Q_k + \epsilon)\|\rho\|_{\la} + \|\rho\|_{\la}^{2}+o(Q).
\end{align}

By Lemma \ref{interaction derivatives refined} we see that $\int_{\bn} \calu_i^p V_j(\calu_k) \ \dvg \approx Q_{ik}$
for all $i \neq k$ and thanks to Lemma \ref{geometric lemma} the constant in $\approx$ depends only on $n,N,\lambda,p$ and $\epsilon.$ Moreover, since the family $(\alpha_i, \calu_i)$ are $\delta$-interacting with $\delta$
sufficiently small we conclude $(\alpha_k^{p-1}\alpha_i + \alpha_i^p - \alpha_i) \approx 1$ for all $1 \leq i \leq N.$ Hence \eqref{ieeqn1} gives 
\begin{align}\label{ieeqn2}
Q_k = \max_{i \neq k} Q_{ik} &\lesssim \|f\|_{H^{-1}} + (Q_k + \epsilon)\|\rho\|_{\la} +\|\rho\|_{\la}^{2}+o(Q) \notag\\
&\lesssim \|f\|_{H^{-1}} +  \epsilon\|\rho\|_{\la} +\|\rho\|_{\la}^{2}+o(Q),
\end{align}
where we have used $\delta$ is sufficiently small to absorb the term $Q_k\|\rho\|_{\la}$. Now testing \eqref{main equation} against $\calu_k$ and using the orthogonality conditions we get
\begin{align} \label{ieeqn3}
\sum_i (\alpha_i^p - \alpha_i)\int_{\bn}\calu_i^p\calu_k \ \dvg= -\int_{\bn}(I_1+I_2 + p\sigma^{p-1}\rho +f)\calu_k \ \dvg.
\end{align}
It follows from Lemma \ref{ie2} and Lemma \ref{ie4} that if $\delta$ is small then
\begin{align} \label{ieeqn4}
\Big| \int_{\bn}(I_2 + p\sigma^{p-1}\rho +f)\calu_k \ \dvg\Big| \lesssim \|f\|_{H^{-1}} +  \epsilon\|\rho\|_{\la} +\|\rho\|_{\la}^2.
\end{align}
On the other hand,  the inequality $|(\sum_i a_i)^p - \sum_{i}a_i^p| \lesssim \sum_{i \neq j} a_i^{p-1}a_j$ and using the estimate \eqref{ieeqn2} we get
\begin{align} \label{ieeqn5}
\Big| \int_{\bn} I_1 \calu_k  \ \dvg\Big| &\lesssim \sum_{i \neq j} \int_{\bn} \calu_i^{p-1}\calu_j \calu_k \ \dvg \lesssim Q_k \, + \, o(Q) \notag\\
&\lesssim \|f\|_{H^{-1}} +  \epsilon\|\rho\|_{\la} + \|\rho\|_{\la}^2 \, + \, o(Q).
\end{align}
Combining \eqref{ieeqn3}, \eqref{ieeqn4}, \eqref{ieeqn5} and using \eqref{ieeqn2} we get
\begin{align} \label{ieeqn6}
|\alpha_k^p - \alpha_k|\int_{\bn} \calu_k^{p+1} \dvg &\leq \sum_{i\neq k} |\alpha_i^p - \alpha_i|\int_{\bn} \calu_i^p\calu_k +\Big|\int_{\bn}(I_1+I_2 + p\sigma^{p-1}\rho +f)\calu_k \ \dvg\Big| \notag \\
&\lesssim \|f\|_{H^{-1}} +  \epsilon\|\rho\|_{\la} + \|\rho\|_{\la}^2+o(Q).
\end{align}
This gives the desired estimate on $\alpha_k$ i.e.
\begin{align}\label{ieeqn7}
|\alpha_k - 1| \lesssim   \|f\|_{H^{-1}} + \epsilon\|\rho\|_{\la}+ \|\rho\|_{\la}^2+o(Q).
\end{align}
We proceed inductively as follows. Discard the indices $k$ from $\{1,\ldots, N\}$ and apply Lemma \ref{geometric lemma} to find indices $l$ ($\neq k$) and a direction which we again call $e_j.$ We proceed as before and get 
\begin{align} \label{ieeqn6}
\sum_{i \neq l}(\alpha_l^{p-1}\alpha_i + \alpha_i^p - \alpha_i) \int_{\bn} \calu_i^p V_j(\calu_l) \ \dvg
\lesssim  \|f\|_{H^{-1}} +  \epsilon\|\rho\|_{\la}+\|\rho\|_{\la}^2+o(Q).
\end{align}
 The terms corresponding to $i = k$ in \eqref{ieeqn6} can be estimated by $Q_k$ and hence 
 \begin{align*}
 &\max_{i \neq l, i \neq k} Q_{il} \lesssim \sum_{i \neq l, i \neq k}(\alpha_l^{p-1}\alpha_i + \alpha_i^p - \alpha_i) \int_{\bn} \calu_i^p V_j(\calu_k) \ \dvg \notag\\
\lesssim \ \ &  \Big|  (\alpha_l^{p-1}\alpha_k + \alpha_k^p - \alpha_k) \int_{\bn} \calu_k^p V_j(\calu_l) \ \dvg\Big| +\|f\|_{H^{-1}} +  \epsilon\|\rho\|_{\la} + \|\rho\|_{\la}^2+o(Q) \notag \\
\lesssim \ \ & \ Q_k +  \|f\|_{H^{-1}} +  \epsilon\|\rho\|_{\la} +\|\rho\|_{\la}^2+o(Q) \notag \\
\lesssim \ \ & \|f\|_{H^{-1}} +  \epsilon\|\rho\|_{\la} + \|\rho\|_{\la}^2+o(Q).
 \end{align*}
Similarly we obtain the estimate $|\alpha_l -1| \leq \|f\|_{H^{-1}} +  \epsilon\|\rho\|_{\la} + \|\rho\|_{\la}^2+o(Q).$ Proceeding in this way and by deleting the indices one by one we get
\begin{align*}
Q = \max_i Q_i \lesssim  \|f\|_{H^{-1}} +  \epsilon\|\rho\|_{\la} +\|\rho\|_{\la}^2+o(Q).
\end{align*}
Since $\delta$ is sufficiently small we get the desired estimate on $Q$ which in turn gives the desired estimate on 
$\max_i |\alpha_i - 1|$ as derived in \eqref{ieeqn5}-\eqref{ieeqn7}, completing the proof of the lemma.
\end{proof}

\section{Counter examples for $1 < p \leq 2$: main theorem and technical results} \label{counter example section 1}

In this section, we prove that the estimate 
\begin{align*}
\delta(u) \lesssim    \|\Delta_{\bn} u+ \la u +  u^{p}\|_{H^{-1}},
\end{align*}
fails whenever $1 <p \leq 2$ in any dimension $n \geq 3$ and $\la$ satisfying {\bf(H1)}. The main theorem of this section is as follows:

\begin{theorem}\label{main counter example}

Let $n \geq 3$ and $1 < p \leq 2$ such that $\lambda, p$ satisfy {\bf (H1)}. There exists a family $u_{R} \in H^1(\bn)$ parametrized by $R \in (0, 1)$ such that the following 
statements hold : 

\begin{enumerate}

\item[(a)] Let $\mathcal{D} := \inf_{c_1, c_2, z_1, z_2} \| u_{R} - c_1 \, \calu[z_1] - c_2 \calu[z_2]  \|_{\lambda},$ then for $R$ sufficiently close to 1,

\begin{align*}
 \| \Delta_{\bn} u_{R} + \lambda u_R + |u_R|^{p-1} u_R \|_{L^{(p+1)'}} \lesssim
 \begin{cases}
\mathcal{D}|\log \mathcal{D}|^{-\frac{1}{2}}, \quad \mbox{if} \ p =2 \\
\medskip
\mathcal{D}^{1 + \alpha_0}, \quad \ \ \ \ \ \ \mbox{if}  \ 1< p < 2,
 \end{cases}
\end{align*}
where $\alpha_0 >0$ is a constant depending on $n,\la$ and $p.$ 
\medskip 

\item[(b)] If we let $R \uparrow 1,$ then the family $u_R$ satisfies 
\begin{align*}
\| u_R - \calu[-Re_1] - \calu[Re_1]\|_{\lambda} \rightarrow 0,
\end{align*}

where $\calu$ is the standard hyperbolic bubbles and $e_1 = (1, 0, \ldots, 0).$

\medskip 

\item[(c)] Let $\mathcal{D^+} := \inf_{c_1, c_2, z_1, z_2} \| u_{R}^{+} - c_1 \, \calu[z_1] - c_2 \calu[z_2]  \|_{\lambda},$ where $u^+ = \mbox{max} \{ u, 0 \},$ then 

\begin{align*}
 \| \Delta_{\bn} u_{R}^+ + \lambda u_R + |u_R|^{p-1} u_R \|_{H^{-1}} \lesssim
 \begin{cases}
\mathcal{D^+}|(\log \mathcal{D}^+)|^{-\frac{1}{4}}, \quad \mbox{if} \ p =2 \\
\medskip
(\mathcal{D^+})^{1 + \frac{\alpha_0}{2}}, \quad \ \ \ \ \ \  \ \mbox{if} \ 1< p < 2,
 \end{cases}
\end{align*}
where $\alpha_0 >0$ is the positive constant in $(a).$ If we let  $R \uparrow 1,$ then the family $u_R^+$ satisfies 
\begin{align*}
\| u_R^+ - \calu[-Re_1] - \calu[Re_1]\|_{\lambda} \rightarrow 0.
\end{align*}

\medskip 

\end{enumerate}
\end{theorem}

\medskip 

\begin{remark}
In particular, Theorem~\ref{main theorem} does not hold when $1 < p \leq 2.$ Moreover, since $2^{\star} - 1 = \frac{n+2}{n-2} \leq 2$ if $n \geq 6,$ therefore Theorem \ref{main theorem} does not hold in dimension $n \geq 6.$
\end{remark}

\medskip

\subsection{Notations and preliminaries}\label{notations section 2}
 We know that the eigenvalues of $(-\Delta_{\bn} - \la)/\calu[z]^{p-1}$ are independent of $z.$ The first two eigenvalues are $1$ and $p$ respectively and the corresponding eigenspace is generated by $(n+1)$ vectors 
$\{\calu[z], V_1(\calu[z]),\ldots,V_n(\calu[z])\}.$ Let $\Theta$ be the smallest eigenvalue strictly bigger than $p$. We fix $\epsilon >0$ such that
\begin{align*}
(1-\epsilon)^2 = \frac{p}{\Theta}, \ \ \ \ i.e. \ \ \ \ \ (1-\epsilon) \Theta = \frac{p}{1-\epsilon}.
\end{align*}

 We denote  $U = \calu[-Re_1], V = \calu[Re_1],$ where $R \in (0,1)$ and define
 \begin{align} \label{deff}
 f = (U+V)^p - U^p-V^p.
 \end{align}
Note that $d(x, Re_1) \leq d(x,-Re_1)$ if $x \in \bn_+ :=\{x \in \bn \ :\ x_1 >0 \}$ and $d(x, Re_1) \geq d(x,-Re_1)$ in $\bn_- :=\{x \in \bn \ :\ x_1 <0 \}.$ Hence $U \geq V$ in $\bn_-$ and $U \leq V$ in $\bn_+.$ Since $1 < p \leq 2$  we deduce 
\begin{align*}
f \approx 
\begin{cases}
UV^{p-1}, \ \ \ \ \mbox{in} \ \bn_+\\
U^{p-1}V, \ \ \ \ \ \mbox{in} \ \bn_-.
\end{cases}
\end{align*}

\begin{itemize}
\item[$\bullet$] We denote by $\mathcal{F} $ the subspace generated by all the eigenfunctions of 
$(-\Delta_{\bn} - \la)/U^{p-1}$ and $(-\Delta_{\bn} - \la)/V^{p-1}$ with eigenvalues less than 
$(1-\epsilon)\Theta.$ That is 
\begin{align*}
\mathcal{F} = \mbox{span} \ B_{\mathcal{F}}, \ \mbox{where} \  B_{\mathcal{F}} = \{U,V,V_j(U),V_j(V), \ 1 \leq j \leq n \} 
\end{align*}
where we recall $V_j(U) = \frac{d}{dt}\Big|_{t=0} U \circ \tau_{te_j}$ and hence dim $\mathcal{F} = 2(n+1).$

\item[$\bullet$] Let $\mathcal{E} $ be the subspace generated by all the eigenfunctions of 
$(-\Delta_{\bn} - \la)/(U+V)^{p-1}$ with eigenvalues less than 
$(1-\epsilon)\Theta.$ We denote by $B_{\mathcal{E}}$ a basis for $\mathcal{E}.$

\item[$\bullet$] Let $\pi : L^2_{(U+V)^{p-1}}(\bn) \to \mathcal{E}^{\bot} \subset  L^2_{(U+V)^{p-1}}(\bn)$ be the orthogonal projection onto the subspace orthogonal to $\mathcal{E}.$ We recall 
\begin{align*}
\mathcal{E}^{\bot} := \Big \{ g \in L^2_{(U+V)^{p-1}} \ | \ \int_{\bn} g\psi(U+V)^{p-1} \ \dvg = 0, \ \ \mbox{for all} \ \psi \in \mathcal{E}\Big\}.
\end{align*}

\item[$\bullet$] We define 
\begin{align}\label{deftf}
\tilde f = (U+V)^{p-1} \pi \left(\frac{f}{(U+V)^{p-1}}\right).
\end{align}
\end{itemize}

The following lemma is a straightforward consequence of the definition.

\begin{lemma} There exists a unique $\rho \in L^2_{(U+V)^{p-1}}(\bn)$ solving the equation 
\begin{align} \label{perturbed problem}
\Delta_{\bn} \rho + \lambda \rho + p (U+V)^{p-1}\rho + \tilde f =0.
\end{align}
\end{lemma}
\begin{proof}
We know that the operator $T := \left(\frac{-\Delta_{\bn} - \la}{(U+V)^{p-1}}\right)^{-1} : L^2_{(U+V)^{p-1}}(\bn) \to L^2_{(U+V)^{p-1}}(\bn)$ is a well defined compact, self-adjoint operator. Therefore, solving \eqref{perturbed problem} is equivalent to solving $(I - pT) \rho = T\left(\pi \left(\frac{f}{(U+V)^{p-1}}\right)\right).$ Since the eigenvalues of $T$  restricted to the subspace $\mathcal{E}^{\bot}$ are strictly smaller than $\frac{1}{p},$ there exists a unique $\rho \in \mathcal{E}^{\bot}$ solving \eqref{perturbed problem}.
\end{proof}

\subsection{ Definition of $u_R$ and a few straightforward  estimates}
For simplicity of notations, we omit the subscript $R.$ We define
\begin{align}\label{defurho}
u = U + V + \rho.
\end{align}
Then 
\begin{align*}
\Delta_{\bn} u + \la u + |u|^{p-1}u = f- \tilde f + \mathcal{R}
\end{align*}
where 
\begin{align*}
\mathcal{R} = |U+V+\rho|^{p-1}(U+V+\rho) - (U+V)^p - p(U+V)^{p-1}\rho.
\end{align*}
For $1 < p \leq 2,$ we have $|\mathcal{R}| \leq |\rho|^p$ and hence
\begin{align}\label{ufrhola}
\|\Delta_{\bn} u + \la u + |u|^{p-1}u\|_{L^{\frac{p+1}{p}}(\bn)} &\leq \|f-\tilde f\|_{L^{\frac{p+1}{p}}(\bn)} + \|\rho^p\|_{L^{\frac{p+1}{p}}(\bn)} \notag\\
&\leq \|f-\tilde f\|_{L^{\frac{p+1}{p}}(\bn)} + \|\rho\|^p_{L^{p+1}(\bn)}\notag\\
&\lesssim \|f-\tilde f\|_{L^{\frac{p+1}{p}}(\bn)} + \|\rho\|^p_{\la}.
\end{align}
Thanks to Lemma \ref{equivalence of norms}(b) we note that $\|\rho\|_{\la} \approx \|\tilde f\|_{H^{-1}(\bn)}. $
In section \ref{E and F} we will show that $\|f-\tilde f\|_{L^{\frac{p+1}{p}}(\bn)} = o( \|f\|_{H^{-1}(\bn)}).$ Since $L^{\frac{p+1}{p}}(\bn)  \hookrightarrow H^{-1}(\bn)$ we deduce 
\begin{align} \label{rhofh-1}
\|\rho\|_{\la} \approx \|f\|_{H^{-1}(\bn)}.
\end{align}

However, the estimate $\|f-\tilde f\|_{L^{\frac{p+1}{p}}(\bn)} = o( \|f\|_{H^{-1}(\bn)}),$ is non-trivial. Poincar\'e inequality implies $L^2(\bn) \hookrightarrow H^{-1}(\bn),$ which creates substantial difficulty in the hyperbolic space. This is where we need a major 
diversion in respect of the ideas used in \cite{FG}. We will achieve this in the next two subsections, building upon a delicate improved integrability of the eigenfunctions of the operator $(-\Delta_{\bn} - \la)/(U+V)^{p-1},$ and  
improved interaction estimates. The latter one follows the idea of our interaction estimates obtained in section \ref{interaction of bubbles section (a)} and based on the observation that the interaction integral estimates of $U^{\alpha}V^{\beta}$ over a half ball can be improved. The former one is much more delicate and would be derived in the next subsection.


\subsection{Improved integrability of the eigenfunctions}

We recall that $U = \calu[-Re_1]$ and $V = \calu[Re_1].$ Let $\Psi \in H^1(\bn)$ solves the following equation 

\begin{equation}\label{EVP}
-\Delta_{\bn} \Psi \, - \, \lambda \Psi \, = \, \mu (U + V)^{p-1} \Psi,
\end{equation}
where $\mu < \frac{p}{1 - \varepsilon}$  for some $\varepsilon > 0,$  and we assume $\Psi$ is normalized i.e. $\int_{\bn} \Psi^{2} (U + V)^{p-1} \, \dvg =1.$ 
We emphasize that $\Psi$ in \eqref{EVP} depends on $R.$ For simplicity of notations we shall denote it by $\Psi.$
\medskip 

The main purpose of this section is to derive the integrability of $\Psi^2$ with respect to the weight $(U+V)^{-\alpha}$ for some $\alpha >0.$  Let us state the main result of this subsection. 

\begin{proposition}\label{weighted-prop}
Let $n \geq 3,$ and $\Psi$ satisfies \eqref{EVP} and normalized. Then there exists a small $\alpha >0$ such that

\begin{equation}
\int_{\bn} \dfrac{\Psi^2}{(U + V)^{2 \alpha}} \, \dvg \, \lesssim \, 1, \quad \mbox{as} \ R \rightarrow 1.
\end{equation}
\end{proposition}

\medskip 

The proof of the above proposition requires many intriguing steps.  For the reader's convenience, we break the proof into several small lemmas. 

\medskip 

We first collect some straightforward estimates. We remark that Lemma \ref{easylemma1}(a), and by standard (interior) elliptic estimates, $\Psi$ is smooth, and hence pointwise evaluation of $\Psi$ is passable.

\begin{lemma}\label{easylemma1}
Let $n \geq 3 ,$ and $\Psi$ satisfies \eqref{EVP}, then the following results hold

\begin{itemize}
\item[(a)] $ \| \Psi\|^2_{\lambda} : = \int_{\bn} \left( |\nabla_{\bn} \Psi|^2 \, - \, \lambda \Psi^2  \right) \, \dvg = \mu,$ \quad \mbox{and} \  $ \| \Psi\|_{L^2} \lesssim 1.$

\medskip 

\item[(b)]  $|\Delta_{\bn} \Psi| = |\Psi| \left( \lambda + \mu (U + V)^{p-1} \right) \approx |\Psi|. $ 

\medskip 

\item[(c)] 
\begin{align*}
\int_{\bn} \dfrac{|\Delta_{\bn} \Psi + \lambda \Psi |^2}{(U + V)^{2(p-1)}} \, \dvg \, = \, \mu \int_{\bn} \Psi^2 \, \dvg \lesssim 1.
\end{align*}

\medskip 

\item[(d)] For every $0 < \alpha \leq \frac{p-1}{2}$ the following equality holds 
\begin{align*}
\int_{\bn} \dfrac{|\nabla_{\bn} \Psi|^2}{W^{2\alpha}} \, \dvg - \lambda \int_{\bn} \dfrac{\Psi^2}{W^{2\alpha}} \, \dvg \, 
= \, 2 \alpha \int_{\bn} \dfrac{\Psi \langle \nabla_{\bn} \Psi, \nabla_{\bn} W\rangle}{W^{2 \alpha + 1}} \, \dvg \, + \, \bigO(1),
\end{align*}
where we denoted $W = U + V$ and $O(1)$ is a quantity that stays bounded as $R\uparrow 1.$
\end{itemize}
\end{lemma}

\begin{proof}
The proof of $(a) - (c)$ is straightforward and follows immediately. For part $(d),$ multiply \eqref{EVP} by $\frac{\Psi}{W^{2 \alpha}}$ and integrate by parts 
\begin{align*}
\int_{\bn} \Big\langle \nabla_{\bn} \Psi, \nabla_{\bn}\left( \dfrac{\Psi}{W^{2 \alpha}} \right) \Big\rangle_{\bn} \, \dvg - \lambda \int_{\bn} \dfrac{\Psi^2}{W^{2 \alpha}} \, \dvg
= \mu \int_{\bn} \Psi^2 W^{p-1 - 2\alpha} \, \dvg = \bigO(1).
\end{align*}
Now the desired conclusion can be reached after some simple computations.
\end{proof}

\medskip 

\begin{remark}
Note that in the above we are using $\int_{\bn} \frac{\Psi^2}{W^{2\alpha}} < +\infty$ which is not yet justified. 
One can proceed as follows: Let $\eta_k$ be a smooth approximation of the identity as $k \rightarrow \infty$ such that $\|\nabla_{\bn}\eta_k\|_{L^{\infty}} + \|\Delta_{\bn}\eta_k\|_{L^{\infty}} = o(1).$ Then we can argue as described below (we apply it to $\Psi \eta_k$ instead of $\Psi$ and we let $k \rightarrow \infty$ instead of varying $R \uparrow 1$). 
\end{remark}

Fix $\alpha$ small (for the moment we assume $2\alpha < p-1$). Assume on the contrary that $\int_{\bn} \frac{\Psi^2}{W^{2 \alpha}} \, \dvg \rightarrow \infty$ as $R \rightarrow 1.$ Let us define the following limits which will play a very essential role in the subsequent proof. 

\begin{align*}
\bullet & \ \ \ l = \limsup_{R \rightarrow 1} \dfrac{ \int_{\bn} \dfrac{|\nabla_{\bn} \Psi|^2}{W^{2 \alpha}} \, \dvg}{\int_{\bn} \dfrac{\Psi^2}{W^{2 \alpha}} \, \dvg}, \ \ \
\ \ \ m =  \limsup_{R \rightarrow 1} \dfrac{ \int_{\bn}  \dfrac{\Psi \langle \nabla_{\bn} \Psi, \nabla_{\bn} W \rangle}{W^{2 \alpha + 1}} \, \dvg}{\int_{\bn} \dfrac{\Psi^2}{W^{2 \alpha}} \, \dvg}\\
\bullet & \ \ \ \mbox{and} \ k = \limsup_{R \rightarrow 1} \dfrac{ \int_{\bn} \dfrac{\Psi^2 |\nabla_{\bn} W|^2}{W^{2 \alpha + 2}} \, \dvg}{\int_{\bn} \dfrac{\Psi^2}{W^{2 \alpha}} \, \dvg}.
\end{align*}
Using Lemma~\ref{easylemma1} and the estimate $|\nabla_{\bn} W| = |\nabla_{\bn} U + \nabla_{\bn} V| \lesssim  U + V = W,$ we conclude that the above limits are finite. 
Moreover, from Lemma \ref{easylemma1}(d) we obtain

$$
l - \lambda = 2 \alpha m,
$$
 and using Cauchy-Schwartz inequality we obtain 
 
 $$
 m \leq k^{\frac{1}{2}} l^{\frac{1}{2}}.
 $$

\begin{lemma}\label{lem-comput}
Let $2\alpha < p-1.$ Then the following relations hold

\medskip 
\begin{itemize}
\item[(a)] $l = \lambda + 2 \alpha m \geq \frac{(n-1)^2}{4} + \alpha \lambda + \alpha(\alpha + 1) k.$ 

\medskip 

\item[(b)] $\lambda = l - \alpha \lambda - \alpha(2\alpha + 1) k.$
\end{itemize}

\end{lemma}

\begin{proof}


For $\beta > 0$ we compute
\begin{align}\label{simple-1}
\Delta_{\bn} \left( \frac{1}{W^{\beta}} \right) &= \mbox{div}_{\bn} \left( - \beta W^{-\beta -1} \nabla_{\bn} W \right) \notag \\
& = - \beta W^{- \beta -1} \Delta_{\bn} W + \beta (\beta + 1) W^{-\beta -2} |\nabla_{\bn} W|^2 \notag \\
& = - \beta W^{- \beta -1}(\Delta_{\bn} (U + V)) + \beta (\beta + 1) W^{-\beta -2} |\nabla_{\bn} W|^2 \notag\\
& =  - \beta W^{- \beta -1} (- \lambda W - (U^p + V^p)) + \beta (\beta + 1) W^{-\beta -2} |\nabla_{\bn} W|^2 \notag \\
& = \lambda \beta W^{-\beta} + \beta W^{- \beta -1} (U^p + V^p) + \beta (\beta + 1) W^{-\beta -2} |\nabla_{\bn} W|^2.
\end{align}
Using Poincar\'e inequality and \eqref{simple-1} we get 
\begin{align}\label{simple-2}
\frac{(n-1)^2}{4} \int_{\bn} \left( \frac{\Psi}{W^{\alpha}} \right)^2 \, \dvg &\leq \int_{\bn} \left| \nabla_{\bn} \left( \frac{\Psi}{W^{\alpha}} \right) \right|^2 \, \dvg 
 = - \int_{\bn} \frac{\Psi}{W^{\alpha}} \Delta_{\bn} \left( \frac{\Psi}{W^{\alpha}} \right) \, \dvg \notag \\
 & = - \int_{\bn} \frac{\Psi}{W^{\alpha}} \left[ \frac{\Delta_{\bn} \Psi}{W^{\alpha}} + 
 2 \langle \nabla_{\bn} \Psi, \nabla_{\bn} \frac{1}{W^{\alpha}} \rangle + \Psi \Delta_{\bn} \left(\frac{1}{W^{\alpha}} \right) \right] \, \dvg \notag \\
 & = - \int_{\bn} \frac{\Psi \Delta_{\bn} \Psi}{W^{2 \alpha}} \, \dvg + 2 \alpha \int_{\bn} \frac{\Psi \langle \nabla_{\bn} \Psi, \nabla_{\bn} W \rangle}{W^{2\alpha + 1}} \, \dvg \notag \\
 & - \int_{\bn} \frac{\Psi^2}{W^{\alpha}} \left[ \alpha \lambda W^{-\alpha} + \alpha W^{-\alpha -1} (U^p + V^p) 
 + \alpha(\alpha + 1) \frac{|\nabla_{\bn} W|^2}{W^{\alpha + 2}}  \right] \, \dvg \notag \\
 & = - \int_{\bn} \frac{\Psi (\Delta_{\bn} \Psi + \lambda \Psi)}{W^{2 \alpha}} \, \dvg + \lambda \int_{\bn} \frac{\Psi^2}{W^{2 \alpha}} \, \dvg \notag \\
& + 2 \alpha \int_{\bn} \frac{\Psi \langle \nabla_{\bn} \Psi, \nabla_{\bn} W \rangle}{W^{2\alpha + 1}} \, \dvg - \alpha \lambda \int_{\bn} \frac{\Psi^2}{W^{2 \alpha}} \, \dvg \notag\\
& - \alpha(\alpha + 1) \int_{\bn} \frac{\Psi^2 |\nabla_{\bn} W|^2}{W^{2 \alpha + 2}} \, \dvg + \, \bigO(1),
\end{align}
where in the last line we have used the fact that $U^p + V^p \lesssim W^p$ and hence $W^{-2 \alpha -1} (U^p + V^p) \lesssim W^{p-1-2\alpha}.$ In the above expression we divide by 
$\int_{\bn} \frac{\Psi^2}{W^{2 \alpha}} \, \dvg$ and let $R \rightarrow 1,$ to complete the proof of part $(a).$

\medskip 

To obtain part $(b)$ we argue as follows:  we multiply \eqref{simple-1}, applied for $\beta = 2 \alpha,$  by $\Psi^2$ and integrate 

\begin{align}\label{simple-3}
- \int_{\bn} \frac{\Psi \Delta_{\bn} \Psi}{W^{2 \alpha}} \, \dvg = \int_{\bn} \frac{|\nabla_{\bn} \Psi|^2}{W^{2 \alpha}}\, \dvg - \alpha \lambda 
\int_{\bn} \frac{\Psi^2}{W^{2 \alpha}} \, \dvg \notag\\
 - \alpha(2 \alpha + 1) \int_{\bn} \frac{\Psi^2 |\nabla_{\bn} W|^2}{W^{2 \alpha + 2}} \, \dvg \, + \, \bigO(1).
\end{align}
Furthermore, using the equation 

\begin{equation}\label{simple-4}
\frac{- \Psi \Delta_{\bn} \Psi}{W^{2 \alpha}} = \lambda \frac{\Psi^2}{W^{2 \alpha}} + \Psi^2 W^{p-1 - 2 \alpha}.
\end{equation}
and substituting \eqref{simple-4} into \eqref{simple-3} and by letting $R \rightarrow 1,$ we obtain the desired result $(b).$ This completes the proof of the lemma. 
\end{proof}

\medskip 

The next lemma provides us with a control over $k$ in terms of $\alpha.$ Note that according to our assumption on $p\leq 2$ and hence $2\alpha <p-1<1.$

\begin{lemma}
We have 
$\frac{\lambda}{1 + 2 \alpha} \leq k \leq \frac{\lambda}{ 1 - 2 \alpha}.$

\end{lemma} 

\begin{proof}
A simple computation using Lemma~\ref{lem-comput} gives us $2  m = \frac{l - \lambda}{\alpha} =  \lambda +  (2 \alpha + 1) k.$ Then 

\begin{align*}
 (\lambda +  (2 \alpha + 1) k)^2 = 4 m^2 \leq 4 l k = 4 ( \lambda + \alpha \lambda + \alpha (2 \alpha + 1)k)k.
\end{align*}
Therefore $\lambda^2 \leq (2 \alpha + 1) (2 \alpha -1)k^2 + 2 \lambda k$ and we obtain the following 

$$
(1 - 4 \alpha^2) k^2 - 2 \lambda k + \lambda^2 \leq 0,
$$
here we note that $2 \alpha \leq p-1< 1,$ provided $p < 2.$ Hence the result follows immediately. 

\end{proof}

Now we are ready to prove the proposition~\ref{weighted-prop}

\medskip

\noindent
{\bf Proof of Proposition~\ref{weighted-prop}.} If the required expression blows up then we have 

$$
\frac{(n-1)^2}{4} + \alpha \lambda + \alpha (\alpha + 1) k \leq l = \lambda + \alpha \lambda + \alpha (2 \alpha +1)k.
$$
The above implies 

$$
\frac{(n-1)^2}{4} - \lambda - \alpha^2 k \leq 0.
$$
This is impossible since, $\lambda < \frac{(n-1)^2}{4}$ and $k = \lambda(1 + \circ(1)),$ as $\alpha \rightarrow 0,$ hence the proposition follows.

\medskip


\subsection{Improved interaction estimates} This section is devoted to the improvement of interaction estimates obtained in Section~\ref{interaction of bubbles section} in the sense described below. We know that for $p < 2,$ there holds 

\begin{align}\label{c1c2c3}
\int_{\bn} U^2 V^{2 (p-1)} \, \dvg \approx e^{-2c(n, \lambda) (p-1)d(-Re_1, Re_1)}.
\end{align}
and it is optimal. However, it turns out that we can improve the estimate if we restrict the domain to the half ball $\bn_{+}:= \{ x \in \bn: x_1 > 0 \}.$ 


\medskip

For simplicity of notations, in this subsection, we shall denote $z_1 = -Re_1$ and $z_2 = Re_1.$
\begin{proposition}\label{general-prop}
Let $U = \calu[- Re_1] = \calu[z_1]$ and $V = \calu[Re_1] = \calu[z_2],$ then 

\begin{equation}\label{iminest}
\int_{\bn_+} U^{\alpha} V^{\beta (p-1)} \, \dvg \approx 
\begin{cases}
e^{- \frac{c(n, \lambda)( \alpha + \beta (p-1))}{2} d(z_1,z_2)}, \quad \  \qquad \mbox{if}  \ \ \alpha > \beta (p-1),\\
d(z_1,z_2) e^{- c(n, \lambda) \beta(p-1) d(z_1,z_2)}, \quad \mbox{if} \quad \alpha = \beta(p-1).
\end{cases}
\end{equation}
\end{proposition}

\begin{remark}
Comparing with \eqref{c1c2c3} we see that the above proposition is a nontrivial improvement of the previous interaction estimates when $p$ is very close to $1.$
\end{remark}

\begin{proof}
Since $z_1 =- Re_1$ and $z_2 = Re_1,$ it is easy to verify that
\begin{align*}
d(z_1, z_2) = 2 d(0, z_2) = 2 \log \left(  \frac{1 + R}{1 -R}\right).
\end{align*}
Therefore $e^{d(z_1, z_2)} \approx \frac{1}{(1 - R^2)^2}$ if $R\approx 1.$ For simplicity of notations we denote $c = c(n,\la)$ then recall that 
\begin{align*}
U(x) \approx e^{- c(n, \lambda) d(x, z_1)} \approx 
 \left[ \dfrac{(1 -R^2)(1 - |x|^2)}{R^2 |x - z_1^{\star}|^2} \right]^{c},
 \end{align*}
  and similarly, 
 \begin{align*}
 V(x) \approx e^{- c(n, \lambda) d(x, z_2)} \approx 
 \left[ \dfrac{(1 -R^2)(1 - |x|^2)}{R^2 |x - z_2^{\star}|^2} \right]^{c}.
 \end{align*}
 
 We estimate 
 
 \begin{align*}
 \int_{\bn_+} U^{\alpha} V^{\beta(p-1)} \, \dvg &\approx \int_{\bn_+} \left[ \dfrac{(1 -R^2)(1 - |x|^2)}{R^2 |x - z_1^{\star}|^2} \right]^{c\alpha }
 \left[ \dfrac{(1 -R^2)(1 - |x|^2)}{R^2 |x - z_2^{\star}|^2} \right]^{c\beta(p-1)} \dfrac{1}{(1 - |x|^2)^n} \, {\rm d}x \\
 & \approx \dfrac{(1 - R^2)^{c(\alpha + \beta(p-1))}}{R^{2 c(\alpha + \beta(p-1))}} \int_{\bn} \dfrac{(1 - |x|^2)^{c(\alpha + \beta(p-1)) - n}}{|x - z_1^{\star}|^{2 c\alpha} |x - z_2^{\star}|^{2c\beta(p-1)}} \, {\rm d}x.
 \end{align*}
 Since $|x - z_1^{\star}| \approx 1$ in $\bn_+$ as $R \rightarrow 1$ we therefore can neglect the term $|x-z_1^*|.$ Hence  
 
 \begin{align*}
  \int_{\bn_+} U^{\alpha} V^{\beta(p-1)} \, \dvg &\approx (1 - R^2)^{c(\alpha + \beta(p-1))} \int_{\bn_+} \dfrac{(1 - |x|^2)^{c(\alpha + \beta(p-1)) -n}}{|x - z_2^{\star}|^{2 c\beta(p-1)}} \, {\rm d}x \\
  & \approx (1 - R^2)^{c(\alpha + \beta(p-1))} 
  \begin{cases}
  1, \ \ \ \ \ \ \ \ \ \ \ \ \mbox{if} \ \alpha > \beta(p-1), \\
  d(0,z_2), \ \ \ \ \mbox{if} \ \alpha = \beta(p-1),
  \end{cases}
 \end{align*} 

which is precisely \eqref{iminest} and hence the proposition follows. 
\end{proof}


\medskip 

The next result is devoted to finding a lower bound on $\| f\|_{H^{-1}},$ where $f = (U + V)^p - U^p - V^p \geq 0$ and $1 < p \leq 2.$ In particular, we prove the following lemma 

\begin{lemma}\label{estimate-f}
Let $z_1 = - Re_1$ and $z_2 = Re_1.$ Then 

\begin{equation}
\| f\|_{H^{-1}} \gtrsim
\begin{cases}
(d(z_1, z_2))^{\frac{1}{2}}e^{- c(n, \lambda) d(z_1, z_2)} , \quad \mbox{if} \ p =2, \\
\medskip
e^{-\frac{c(n, \lambda) p}{2} d(z_1, z_2)} \qquad \qquad \qquad \mbox{if} \  1< p < 2.
\end{cases}
\end{equation}
\end{lemma}

\begin{proof}
Let us first consider the special case $p =2$ for which $f$ simplifies to $f = 2 UV.$ Fix any  $\alpha \geq 2.$ Then we have 
\begin{align}\label{h-1}
d(z_1,z_2)e^{- c(n, \lambda) \alpha d(z_1, z_2)} \approx \int_{\bn} U^{\alpha} V^{\alpha} \, \dvg \approx \int_{\bn} f U^{\alpha - 1} V^{\alpha - 1} \, \dvg
\leq \|f\|_{H^{-1}} \| U^{\alpha -1} V^{\alpha -1}\|_{H^1}.
\end{align}
A direct computation of the $H^1$ norm yields 

\begin{align}\label{h-1-1}
\| U^{\alpha - 1} V^{\alpha -1}\|_{H^1} &= \| \nabla_{\bn} (U^{\alpha -1} V^{\alpha -1})\|_{L^2}\notag \\
&= \left( \int_{\bn} | U^{\alpha -1} (\alpha -1) (\nabla_{\bn} V) V^{\alpha -2} \, + \, V^{\alpha -1} (\alpha -1) U^{\alpha -2} \nabla_{\bn} U|^2 \, \dvg \right)^{\frac{1}{2}} \notag \\
& \lesssim \left( \int_{\bn} U^{2(\alpha -1)} V^{2(\alpha -1)}  \, \dvg\right)^{\frac{1}{2}} \approx e^{-c(n, \lambda) (\alpha - 1)d(z_1, z_2)} (d(z_1, z_2))^{\frac{1}{2}},
\end{align}
where in the last line we have used $\|\nabla_{\bn} \calu\|_{\bn} \lesssim \calu.$
Now substituting \eqref{h-1-1} into  \eqref{h-1} we obtain the desired result.  

\medskip 

Next, we shall consider the case $ 1 < p < 2.$ We use Green function estimates of $-\Delta_{\bn}$ in the hyperbolic space. Let $G(x, y)$ be the 
Green function of $-\Delta_{\bn}$ and is given by $G(x, y):= G(d(x, y)) := \int_{d(x, y)}^{\infty} (\sinh s)^{-(n-1)} \, {\rm d}s.$ Then $G$ satisfies 

$$
G(x, y) \geq C_n\, e^{-(n-1) d(x, y)},
$$
whenever $d(x, y) \geq 1$ and $C_n$ is a dimensional constant.  Let us define 

$$
u(x) := \int_{\bn} G(x, y) \, f(y) \, \dvg(y).
$$
Then $u$ satisfies,  $- \Delta_{\bn} u = f$ in $\bn.$ Further, multiplying this equation by $u$ and integration by parts gives  

$$
\| \nabla_{\bn} u\|^2_{L^2} = \int_{\bn} f \, u \, \dvg \leq \, \| f\|_{H^{-1}} \| \nabla_{\bn} u\|_{L^2},
$$
and hence we obtain $\|\nabla_{\bn} u\|_{L^2} \leq \|f\|_{H^{-1}}.$ Then the Poincar\'e-Sobolev inequality gives us 

$$
\|f\|_{H^{-1}} \geq \|\nabla_{\bn} u\|_{L^2} \geq \| u\|_{L^{p+1}}.
$$
Our aim is now to derive a point-wise lower bound of $u.$ Recall from subsection \ref{notations section 2},
\begin{equation}
f \approx 
\begin{cases}
UV^{p-1} \quad \mbox{in} \ \bn_+ \\
U^{p-1} V \quad \mbox{in} \ \bn_-.
\end{cases}
\end{equation}
Let us denote $\Omega = \{  x \in \bn : x_1 < -\frac{1}{2}\},$ then  for $x \in \Omega$ and  $y \in \bn_+,$  $d(x, y) \gtrsim 1,$  and hence 

\begin{align}\label{f-1-1}
u(x) &:= \int_{\bn} G(x, y) \, f(y) \, \dvg(y) \gtrsim \int_{\bn_+} G(x, y) \ UV^{p-1}(y) \, \dvg(y) \notag \\
& \approx  \int_{\bn_+} e^{-(n-1) d(x, y)} e^{- c(n, \lambda) d(y, z_1)} e^{-c(n, \lambda)(p-1) d(y, z_2)} \, \dvg(y) \notag \\
& \approx \int_{\bn_+} \left[ \left(\dfrac{(1 - |x|^2)(1 - |y|^2)}{|x|^2 |y - x^{\star}|^{2}} \right)^{n-1} 
\left(\dfrac{(1 - R^2) (1 - |y|^2)}{R^2 |y - z_1^{\star}|^2} \right)^{c(n, \lambda)} \times\right. \notag \\
& \qquad \qquad \qquad \qquad \qquad \ \ \  \left.  \left(\dfrac{(1 - R^2) (1 - |y|^2)}{R^2 |y - z_2^{\star}|^2} \right)^{c(n, \lambda)(p-1)} (1 - |y|^2)^{-n} \right]\, {\rm d}y \notag\\
& \approx (1 - R^2)^{c(n, \lambda) p} \dfrac{(1 - |x|^{2})^{n-1}}{|x|^{2(n-1)}} \int_{\bn_+} 
\dfrac{(1 - |y|^2)^{c(n, \lambda) p -1}}{|y - z_2^{\star}|^{2 c(n, \lambda) (p-1)} |y - x^{\star}|^{2(n-1)}} \, {\rm d}y.
\end{align}
In the last line we have used the fact that $|y - z_1^{\star}| \approx 1$ if $y \in \bn_+.$ Also by our assumption $|y - x^{\star}| \approx 1,$ whenever $y \in \bn_+$
 and $x \in \Omega.$ Moreover, $2c(n, \lambda) (p-1) - c(n, \lambda) p + 1 = c(n, \lambda)p - 2 c(n, \lambda) + 1 = c(n, \lambda) (p-2) + 1 < n$ and hence 
 
 \begin{align*}
 \int_{\bn_+} \dfrac{(1 - |y|^2)^{c(n, \lambda) p -1}}{|y - z_2^{\star}|^{2 c(n, \lambda) (p-1)} } \, {\rm d}y \gtrsim 1
 \end{align*}
 
 Therefore we deduce from \eqref{f-1-1} that 
 \begin{align*}
 u(x) \gtrsim (1 - R^2)^{c(n, \lambda) p} (1 - |x|^2)^{n-1} \quad \mbox{in} \ \Omega,
 \end{align*}
and consequently,  
\begin{align*}
\|u\|_{L^{p+1}} \gtrsim (1 - R^2)^{c(n, \lambda) p} \left( \int_{\Omega} (1 - |x|^2)^{(n-1)(p+1) - n} \, {\rm d}x \right)^{\frac{1}{p+1}} \approx (1 - R^2)^{c(n, \lambda) p}.
\end{align*}
Since $e^{-\frac{d(z_1,z_2)}{2}} \approx (1 - R^2),$ the proof of the proposition has been completed.
\end{proof}

\medskip

As a consequence of the above results we prove 

\begin{proposition} \label{intfpsi}
Let $\Psi$ satisfies \eqref{EVP}. Then there exists a constant $\alpha_0>0$ depending on $n,\la$ and $p$ such that 

\begin{equation}
\left|  \int_{\bn} f \, \Psi \ \dvg\right| \lesssim
\begin{cases}
 \| f\|_{H^{-1}}^{1+ \alpha_0}, \quad \ \ \ \ \qquad \quad \  \ \ \ \ \ \ \ \mbox{if} \ 1 < p < 2,\\
\medskip
\| f\|_{H^{-1}} \Big|\ln \|f\|_{H^{-1}}\Big|^{-\frac{1}{2}}, \quad \ \ \ \ \mbox{if} \ p =2.
\end{cases}
\end{equation}

\end{proposition}

\begin{proof}
We first consider $1 < p < 2.$ Let  $\alpha >0$ be the small quantity determined by Proposition ~\ref{weighted-prop}. We denote $W = U + V,$ then $W \approx V$ (respectively, $U$) in $\bn_+$ (respectively, $\bn_-$). By Cauchy-Schwartz inequality we deduce

\begin{align} \label{fpsiestimate}
\left|  \int_{\bn} f \, \Psi \, \dvg\right| &\lesssim \int_{\bn_+} U V^{p-1} |\Psi| \, \dvg \,+ \,   \int_{\bn_-} U^{p-1} V |\Psi| \, \dvg \notag \\
& \approx \int_{\bn_+} U V^{p-1 + \alpha}  \frac{|\Psi|}{W^{\alpha}} \, \dvg \, + \, \int_{\bn_-} U^{p-1 + \alpha} V  \frac{|\Psi|}{W^{\alpha}} \, \dvg \notag \\
& \lesssim \left( \int_{\bn_+} U^2 V^{2(p-1 + \alpha)} \, \dvg \right)^{\frac{1}{2}} \, + \, \left( \int_{\bn_-} U^{2(p-1 + \alpha)} V^2 \, \dvg \right)^{\frac{1}{2}} \notag \\
& \approx e^{-c(n, \lambda) \frac{2 + 2(p-1 + \alpha)}{4} d(z_1, z_2)} \ \ \ \mbox{if $p-1+\alpha <1$}\notag \\
& \approx \left( e^{- c(n, \lambda) \frac{p}{2} d(z_1, z_2)} \right)^{\frac{p+\alpha}{p}} \lesssim   \| f \|_{H^{-1}}^{1+\alpha_0},
\end{align}
where $\alpha_0 = \frac{\alpha}{p}.$ Note that if $p-1+ \alpha \geq 1,$ then \eqref{fpsiestimate} is still valid for any $\alpha_0 <\frac{2-p}{p}.$ The case, $p = 2$ can be tackled as follows 
 
 \begin{align}
\left|  \int_{\bn} f \, \Psi \, \dvg\right|  \lesssim  \int_{\bn} UV |\Psi| \, \dvg &\lesssim \left( \int_{\bn_+} U^2 V^{2(1 + \alpha)} \, \dvg \right)^{\frac{1}{2}} + \left( \int_{\bn_-} U^{2(1 + \alpha)} V^{2} \, \dvg \right)^{\frac{1}{2}}\notag \\
& \lesssim \left( \int_{\bn} U^2 V^{2(1 + \alpha)} \, \dvg \right)^{\frac{1}{2}} \notag\\ 
&\approx e^{-c(n, \lambda) d(z_1, z_2)} 
 \end{align}
 By Lemma \ref{estimate-f} and using the bound $\|f\|_{H^{-1}} \lesssim \|f\|_{L^2}$ we get $|\ln \|f\|_{H^{-1}}| \approx d(z_1,z_2)$ and hence $e^{-c(n, \lambda) d(z_1, z_2)} \lesssim \|f\|_{H^{-1}}|\ln \|f\|_{H^{-1}}|^{-\frac{1}{2}}.$ This completes the proof of the lemma.  

\end{proof}

Based on the above estimates, we can now follow the ideas of Figalli and Glaudo \cite{FG} and complete the proof of sharpness of the stability estimate of Theorem \ref{main theorem}. There are a few technical points that need to be addressed in the case of the hyperbolic space, however, the main arguments are almost along the lines of \cite{FG}. For the reader's convenience, we plan to provide as many details as possible, so that the article can be read without prior consultation of the Euclidean case.


\subsection{The subspaces $\mathcal{E}$ and $\mathcal{F}$ are close}\label{E and F}

We recall the definition of generalized Hausedorff distance between two subspaces of a Hilbert space $H$ 
\cite{Morris, FG}.

\begin{definition}
Let $X$ be a Hilbert space and let $\|\cdot\|_X$ be the associated norm. Let $E, F$ be two subspaces of $X.$ The distance $d_X(E,F)$ between them is defined as 
\begin{align*}
d_X(E,F) = d_{H}(\overline{B_X(0,1)}\cap E, \overline{B_X(0,1)}\cap F).
\end{align*}
where $d_H$ is the Hausedorff distance defined by
\begin{align*}
d_H(A,B) = \inf\{r>0 \ | \ A \subset B + \overline{B_X(0,r)}, \ \mbox{and} \ B \subset A + \overline{B_X(0,r)}\},
\end{align*}
and $B_X(0,r)$ denotes the open ball in $X$ of radius $r.$  It is easy to verify that the Hausedorff distance satisfies the scaling property $d_H(\kappa A, \kappa B) = \kappa \ d_H(A,B),$ for $\kappa>0.$
\end{definition}

 The main properties of $d_X$ used in this article are mentioned in the next few bullet points
\begin{itemize}
\item For every subspaces $E,F$ of $X$ it holds $d_X(E,F) \leq 1,$ and if two subspaces have different dimensions then $d_X(E,F) = 1.$

\medskip

\item If $E$ and $F$ are two closed subspaces of $X,$ then $d_X(E,F)$ is given by 
\begin{align*}
d_X(E,F) = \sup_{e \in E\backslash \{0\}} \frac{\|e - \pi_F(e)\|_X}{\|e\|_X} = \sup_{f \in F\backslash \{0\}} \frac{\|f - \pi_E(f)\|_X}{\|f\|_X},
\end{align*}
where $\pi_E$ and $\pi_F$ are the projection operators onto $E$ and $F$ respectively. Hence as a corollary
\begin{align}\label{generalized hdis}
\|e - \pi_F(e)\|_X \leq d_X(E,F)\|e\|_X, \ \ \|f - \pi_E(f)\|_X \leq d(E,F)\|f\|_F
\end{align}
holds for every $e \in E$ and $f \in F.$

\medskip

\item $d_X$ is invariant under orthogonal complement, i.e.,
\begin{align*}
d_X(E,F) = d_X(E^{\bot}, F^{\bot}).
\end{align*}

\end{itemize}

In this section we assume $X = L^2_{(U+V)^{p-1}}(\bn)$ and $E = \mathcal{E}, F = \mathcal{F}$ where $\mathcal{E}, \mathcal{F}$ are defined in section \ref{notations section 2}. The main purpose of this section is to prove that $d_X(\mathcal{E},\mathcal{F}) = o(1),$ which would ensure $\mbox{dim} \  \mathcal{E} = \mbox{dim} \ \mathcal{F}.$

\begin{lemma} \label{same dimension}
Let $X = L^2_{(U+V)^{p-1}}(\bn),$ then
$d_X(\mathcal{E}, \mathcal{F}) = o(1)$ as $R \uparrow 1.$ In particular,
\begin{align*}
\mbox{dim} \  \mathcal{E} = \mbox{dim} \ \mathcal{F} = 2(n+1).
\end{align*}
\end{lemma}

\begin{proof}
We divide the proof into two steps.

\medskip

\noindent
{\bf Step 1:} For any $\psi \in B_{\mathcal{F}}$ there exists $\psi^{\prime} \in \mathcal{E}$ such that 
\begin{align*}
\|\psi - \psi^{\prime}\|_{L^2_{(U+V)^{p-1}}(\bn)} = o(1), \ \ \mbox{as} \ R \uparrow 1.
\end{align*}
Without loss of generality, we may assume that $\psi$ satisfies
\begin{align*}
-\Delta_{\bn} \psi - \la \psi = \mu U^{p-1}\psi, \ \ \ \mbox{where} \ \ \mu = 1 \ \mbox{or} \ p.
\end{align*}
The idea is to project $\psi$ onto $\mathcal{E}.$ We rewrite the equation satisfied by $\psi$ as follows
\begin{align*}
-\Delta_{\bn} \psi - \la \psi - \mu (U+V)^{p-1}\psi= \mu \left(U^{p-1} - (U+V)^{p-1}\right)\psi =: \ f.
\end{align*}
We write $\psi = \sum_k \alpha_k \psi_k,$ where $\{\psi_k\}$ is the $L^2_{(U+V)^{p-1}}$-orthonormal  basis consisting of eigenvectors of $(-\Delta_{\bn} - \la)/((U+V)^{p-1}).$ Then by Lemma \ref{equivalence of norms}(a) we get
\begin{align*}
\sum_k \alpha_k^2 \left(1 - \frac{\mu}{\la_k}\right)^2 &= \|(-\Delta_{\bn} - \la)^{-1}f\|_{L^2_{(U+V)^{p-1}}(\bn)}^2 \\
&\lesssim \|f\|^2_{L^2(\bn)} \\
&\lesssim \int_{\bn} |U^{p-1} - (U+V)^{p-1}|^2\psi^2 \ \dvg \\
&\lesssim \left(\int_{\bn} |U^{p-1} - (U+V)^{p-1}|^{\frac{2(p+1)}{p-1}}\right)^{\frac{p-1}{p+1}}
\left(\int_{\bn} |\psi|^{p+1}\right)^{\frac{2}{p+1}}\\
&= o(1), \ \ \mbox{as} \ \ R \uparrow 1.
\end{align*}
We define $\psi^{\prime} = \pi_{\mathcal{E}}(\psi),$ the projection of $\psi$ on to the subspace
$\mathcal{E}.$ Then 
\begin{align*}
\|\psi - \psi^{\prime}\|_{L^2_{(U+V)^{p-1}}(\bn)}^2 &= \sum_{\la_k \geq p(1-\epsilon)^{-1}} \alpha_k^2 \\
&\leq \frac{1}{\epsilon^2}\sum_{\la_k \geq p(1-\epsilon)^{-1}} \alpha_k^2 \left(1- \frac{\mu}{\la_k}\right)^2 \\
&=o(1), \ \ \mbox{as} \ \ R \uparrow 1.
\end{align*}
 
 \noindent
 {\bf Step 2:} For any $\psi \in B_{\mathcal{E}},$ there exists $\psi^{\prime} \in \mathcal{F}$
 such that 
 \begin{align*}
 \|\psi - \psi^{\prime}\|_{L^2_{(U+V)^{p-1}}(\bn)} = o(1), \ \ \mbox{as} \ \ R \uparrow 1.
 \end{align*}
 The idea is to localize $\psi$ around the bubbles $U$ and $V$ respectively and approximate the localized function by the sum of elements of $\mathcal{F}.$ Here recall 
 $\mathcal{F}$ is composed of eigenvectors of $(-\Delta_{\bn}-\la)/U^{p-1}$ and $(-\Delta_{\bn} - \la)/V^{p-1}$ with eigenvalues $1$ and $p$ respectively.
 
 Let $\eta$ be a smooth cutoff function identically $1$ around $-Re_1.$ We will decide $\eta$ later. Set $\varphi_U = \psi \eta$ and we write $\varphi_U = \sum_k \beta_k\varphi_k,$ where
 $\{\varphi_k\}$ is an orthonormal basis of $L^2_{U^{p-1}}(\bn)$ consisting of eigenvectors of $(-\Delta_{\bn} - \la)/U^{p-1}.$ Here for clarity, we remark that orthonormal with respect to the $L^2_{U^{p-1}}$ inner product, and we denote the eigenvalues by $\{\tilde \la_k\}$. By Lemma \ref{equivalence of norms}(a), denoting $f = (-\Delta_{\bn} - \la - \mu U^{p-1})\varphi_U$ we get
 \begin{align} \label{identity beta1}
 \sum_{k} \beta_k^2\left(1 - \frac{\mu}{\tilde \la_k}\right)^2 &= \|(-\Delta_{\bn} - \la)^{-1}f\|_{L^2_{U^{p-1}}}^2 \notag\\
 &\lesssim \|f\|_{L^2(\bn)}^2 = \|(-\Delta_{\bn} - \la - \mu U^{p-1})\varphi_U\|_{L^2(\bn)}^2.
 \end{align}
 Since $(-\Delta_{\bn} - \la - \mu (U+V)^{p-1})\psi = 0,$ where $\mu < p(1-\epsilon)^{-1},$
 we compute
 \begin{align*}
 (-\Delta_{\bn} - \la - \mu U^{p-1})\varphi_U &= (-\Delta_{\bn} - \la - \mu U^{p-1})(\psi \eta)\\
 &= (-\Delta_{\bn} \eta) \psi - 2 \langle \nabla_{\bn} \psi, \nabla_{\bn} \eta \rangle_{\bn}
 + \mu((U+V)^{p-1} - U^{p-1})\psi\eta.
 \end{align*} 
 Let $d = d(-Re_1,Re_1).$ Let $\eta$ be the cut off function as in Lemma \ref{l:lip func}, such that $\eta \equiv 1$ on  $ B(-Re_1,\frac{d}{4})$ and support $\eta$ contained in $B(-Re_1, \frac{d}{2})$ so that 
 \begin{align*}
 \|\Delta_{\bn} \eta\|_{L^{\infty}(\bn)} + \|\nabla_{\bn}\eta\|_{L^{\infty}(\bn)} = o(1).
 \end{align*}
Then we have supp $\varphi_U \subset \bn_-$ and
\begin{align} \label{identity beta2}
\|(-\Delta_{\bn} - \la - \mu U^{p-1})\varphi_U\|_{L^2(\bn)} &\leq \|\psi \Delta_{\bn}\eta\|_{L^2(\bn)} + 
2\| \langle \nabla_{\bn} \psi, \nabla_{\bn} \eta \rangle_{\bn}\|_{L^2(\bn)} \notag\\
& \ \ \ \ + \|\mu((U+V)^{p-1} - U^{p-1})\psi\eta\|_{L^2(\bn)} \notag\\
 &= o(1), \ \ \mbox{as} \ \ R \uparrow 1.
\end{align}
From \eqref{identity beta1} and \eqref{identity beta2} we get $ \sum_{k} \beta_k^2\left(1 - \frac{\mu}{\tilde \la_k}\right)^2 = o(1)$ as $R \uparrow 1.$ We set $\psi_U = \sum_{\tilde \la_k < p(1-\epsilon)^{-1}} \beta_k\varphi_k \in \mathcal{F},$ then 
\begin{align*}
\|\psi_U - \varphi_U\|_{L^2_{U^{p-1}}(\bn)} = \sum_{\tilde \la_k \geq p(1-\epsilon)^{-1}} \beta_k^2 \leq \frac{1}{\epsilon^2} \sum_{\tilde \la_k \geq p(1-\epsilon)^{-1}} \beta_k^2\left(1 - \frac{\mu}{\tilde \la_k}\right)^2 = o(1), \ \ \mbox{as} \ \ R \uparrow 1.
\end{align*}
Similarly, we find $\varphi_V$ such that supp $(\varphi_V) \subset \bn_+$ and $\psi_V \in \mathcal{F}$ such that 
\begin{align*}
\|\psi_V - \varphi_V\|_{L^2_{V^{p-1}}(\bn)} = o(1), \ \ \mbox{as} \ \ R \uparrow 1.
\end{align*}
We recall that $U + V \approx U$ in $\bn_-$ and $\approx V$ in $\bn_+.$ Also, recall that 
$\psi_U$ is an element of $\mathcal{F}_U,$ $|\psi_U| \lesssim U$ holds.
With this information at hand, we refine the previous estimates as follows
\begin{align*}
\|\psi_U - \varphi_U\|_{L^2_{(U+V)^{p-1}}(\bn)} &\leq \|\psi_U - \varphi_U\|_{L^2_{(U+V)^{p-1}}(\bn_-)} + \|\psi_U\|_{L^2_{(U+V)^{p-1}}(\bn_+)} \\ 
&\lesssim \|\psi_U - \varphi_U\|_{L^2_{U^{p-1}}(\bn_-)} + \|\psi_U\|_{L^2_{V^{p-1}}(\bn_+)}\\
&= o(1), \ \ \mbox{as} \ \ R \uparrow 1.
\end{align*}
In the last line, we have used 
\begin{align*}
 \|\psi_U\|_{L^2_{V^{p-1}}(\bn_+)}^2 \lesssim \int_{\bn_+}U^2V^{p-1} \ \dvg  = o(1) \ \ \mbox{as} \ \ R \uparrow 1.
\end{align*}
Similarly, we get
\begin{align*}
\|\psi_V - \varphi_V\|_{L^2_{(U+V)^{p-1}}(\bn)} = o(1), \ \ \mbox{as} \ \ R \uparrow 1.
\end{align*}
We set $\psi^{\prime} = \psi_U + \psi_V,$ then 
\begin{align*}
\|\psi - \psi^{\prime}\|_{L^2_{(U+V)^{p-1}}(\bn)} &\leq \|\varphi_U - \psi_U\|_{L^2_{(U+V)^{p-1}}(\bn)} + \|\varphi_V - \psi_V\|_{L^2_{(U+V)^{p-1}}(\bn)} \\
& \ \ \ \ + \|\psi - \varphi_U - \varphi_V\|_{L^2_{(U+V)^{p-1}}(\bn)} \\
& = o(1) + \|\psi - \varphi_U - \varphi_V\|_{L^2_{(U+V)^{p-1}}(\bn)}, \ \ \mbox{as} \ \ R \uparrow 1.
\end{align*}
Since $\varphi_U = \psi$ on $B(-Re_1, \frac{d}{4})$ and $ \varphi_V = \psi$ on $B(Re_1, \frac{d}{4})$ we have
\begin{align*}
\|\psi - \varphi_U - \varphi_V\|_{L^2_{(U+V)^{p-1}}(\bn)}^2 &\lesssim 
\int_{\bn \backslash (B(-Re_1, \frac{d}{4}) \cup B(Re_1, \frac{d}{4}))} \psi^2(U+V)^{p-1} \ \dvg \\
&\lesssim o(1)\int_{\bn} \psi^2 \ \dvg \\
&=o(1), \ \ \mbox{as} \ \ R \uparrow 1.
\end{align*}
This completes the proof.
\end{proof}

As a corollary, we have the following lemma
\begin{lemma} \label{ffh-1}
Let $f$ and $\tilde f$ be defined as in \eqref{deff} and \eqref{deftf} respectively and $R \approx 1.$ 
\begin{itemize}
\item[(a)] If $p = 2,$ then 
\begin{align*}
\|f - \tilde f\|_{L^{\frac{p+1}{p}}(\bn)} \lesssim \|f\|_{H^{-1}(\bn)}\Big|\ln \|f\|_{H^{-1}(\bn)}\Big|^{-\frac{1}{2}}.
\end{align*}
\item[(b)] If $1 < p < 2,$ then there exists $\alpha_0 >0$ depending on $n,\la$ and $p$ such that 
\begin{align*}
\|f - \tilde f\|_{L^{\frac{p+1}{p}}(\bn)} \lesssim \|f\|_{H^{-1}(\bn)}^{1+\alpha_0}.
\end{align*}
\end{itemize}
In particular, $\|f - \tilde f\|_{H^{-1}(\bn)} \lesssim \|f - \tilde f\|_{L^{\frac{p+1}{p}}(\bn)} = o(\|f\|_{H^{-1}(\bn)}).$
\end{lemma}
\begin{proof}
By definition 
\begin{align*}
f - \tilde f &= (U+V)^{p-1}\left[\frac{f}{(U+V)^{p-1}} - \pi_{\mathcal{E}}\left(\frac{f}{(U+V)^{p-1}}\right)\right]\\
&=(U+V)^{p-1}\sum_{\psi \in B_{\mathcal{E}}} \langle \frac{f}{(U+V)^{p-1}}, \psi\rangle_{L^2_{(U+V)^{p-1}}(\bn)} \psi \\
& = (U+V)^{p-1}\sum_{\psi \in B_{\mathcal{E}}} \langle f, \psi \rangle_{L^2(\bn)} \psi.
\end{align*}
Hence, 
\begin{align*}
\|f - \tilde f\|_{L^{\frac{p+1}{p}}(\bn)}  &\leq \sum_{\psi \in B_{\mathcal{E}}} |\langle f, \psi \rangle_{L^2(\bn)}|\| \psi(U+V)^{p-1}\|_{L^{\frac{p+1}{p}}(\bn)}\\
&\leq \sum_{\psi \in B_{\mathcal{E}}} |\langle f, \psi \rangle_{L^2(\bn)}| \|\psi\|_{L^{p+1}(\bn)}
\|U + V\|_{L^{p+1}(\bn)} \\
&\lesssim \sup_{\psi \in B_{\mathcal{E}}}|\langle f, \psi \rangle_{L^2(\bn)}|.
\end{align*}
By Proposition \ref{intfpsi}, the lemma follows.
\end{proof}

Thanks to \eqref{ufrhola}, Lemma \ref{ffh-1} and \eqref{rhofh-1}  we have the following
\begin{lemma} \label{finallemma}
If $R \approx 1$, then 
\begin{align*}
\|\Delta_{\bn} u + \la u + |u|^{p-1}u\|_{L^{\frac{p+1}{p}}(\bn)} \lesssim
\begin{cases}
\|\rho\|_{\la}|\ln \|\rho\|_{\la}|^{-\frac{1}{2}}, \ \ \ \ \mbox{if} \ p=2, \\
\|\rho\|_{\la}^{1+\alpha_0}, \ \ \ \ \ \ \ \ \ \ \ \ \ \ \mbox{if} \ 1 < p < 2,
\end{cases}
\end{align*}
where $\alpha_0$ is a positive constant depending on $n,\la$ and $p.$
\end{lemma}

\begin{remark}
We expect that Lemma \ref{finallemma} should hold for all $\alpha_0 < \frac{2-p}{p},$ however, we do not know yet how to prove it. Given the Euclidean stability estimate of \cite{Weietal}, $\alpha_0 = \frac{2-p}{p}$ should be optimal. 
\end{remark}


\section{Proof of the counter examples}

In this section, we prove that $\|\rho\|_{\la} \approx \mbox{dist}(u, \mbox{sum of hyperbolic bubbles}),$ which will complete the proof of counter-example. First, we need some preparatory lemmas.

\begin{lemma}\label{derivative vu}
The followings hold
\begin{itemize}
\item[(a)] For fixed $w\in \bn,$ the following holds
\begin{align*}
\calu[z] - \calu[w]  + 2\calu^{\prime}(d(\cdot, w))\frac{\tau_{-w}(\cdot)}{|\tau_{-w}(\cdot)|}\cdot \tau_{-w}(z) = o(|\tau_{-w}(z)|), \ \ \mbox{in} \ H^1(\bn)
\end{align*}
as $d(z,w) \rightarrow 0.$  

\item[(b)] $\calu[w]$ is orthogonal to $2\calu^{\prime}(d(\cdot, w))\frac{\tau_{-w}(\cdot)}{|\tau_{-w}(\cdot)|}$ with respect to $\langle \cdot, \cdot \rangle_{\lambda}.$

\item[(c)] Moreover, denoting $u_j[w] = 2\calu^{\prime}(d(\cdot, w))\frac{(\tau_{-w}(\cdot))_j}{|\tau_{-w}(\cdot)|}, j = 1,2,\ldots,n$
we have  
\begin{align*}
\|u_j[w]\|_{L^{p+1}(\bn)} \approx 1 , \ \ \ \|u_j[w]\|_{\la} \approx 1
\end{align*}
uniformly with respect to $w,$ where $(\tau_{-w}(\cdot))_j$ denotes the $j$-th component of $\tau_{-w}(\cdot).$
\end{itemize}
\end{lemma}

\begin{proof}
\noindent
(a) We first consider the case $w=0.$ The general case then follows from this particular case.
\begin{align}\label{someeq0}
\calu[z](x) - \calu[0](x) &= \calu(d(x,z)) - \calu(d(x,0)) \notag\\
& = \int_0^1 \frac{d}{ds} \calu(d(x,sz)) \ ds \notag\\
&= \int_0^1 \calu^{\prime}(d(x,sz)) \frac{d}{ds}d(x,sz) \ ds \notag\\
&= \int_0^1 \left(\calu^{\prime}(d(x,sz)) - \calu^{\prime}(d(x,0))\right)\frac{d}{ds}d(x,sz) \ ds \notag \\
& \ \ \ \ \ \ \ \ \ \ \ \ \ \ \ \ \ \ \ \ + \calu^{\prime}(d(x,0))\left(d(x,z) - d(x,0)\right).
\end{align}
We recall, by definition
\begin{align} \label{12vj12}
V_j(\calu[0])(x) &= 2\calu^{\prime}(d(x,0)) \frac{x_j}{|x|}.
\end{align}
Multiplying \eqref{12vj12} by $z_j$ and summing over $j = 1,2,\ldots, n$ we get
\begin{align}\label{someeq4}
V[\calu[0]](x) \cdot z = 2\calu^{\prime}(d(x,0)) \frac{(x \cdot z)}{|x|}
\end{align}

We see from the formula $d(x,sz) = 2 \ln \left(\frac{\sqrt{1-2sx\cdot z + s^2|x|^2|z|^2} + |x-sz|}{\sqrt{(1-|x|^2)(1-s^2|z|^2)}}\right)$ that
\begin{align}
\frac{d}{ds}d(x,sz) = O\left(\frac{|z|}{|x|}\right), \ \ \ \mbox{as} \ |z|\rightarrow 0,
\end{align}
in the $C^1$ topology and it is easy to see from the same formula that 
\begin{align}
|\nabla_{\bn} d(x,z)|_{\bn} = \left(\frac{1-|x|^2}{2}\right)|\nabla \ d(x,z)| = O(1),
\end{align}
uniformly in $x$ as $z \approx 0.$
Now we estimate $d(x,z) - d(x,0)$ as $|z| \rightarrow 0.$ Using the same formula for the distance we get
\begin{align}\label{abc1}
e^{d(x,z)-d(x,0)} - 1 &= \frac{(\sqrt{1-2x\cdot z + |x|^2|z|^2} + |x-z|)^2}{(1+|x|)^2(1-|z|^2)} - 1 \notag\\
&=\frac{-4 x\cdot z -2|x| + 2|x-z|\sqrt{1-2x\cdot z + |x|^2|z|^2} + O(|z|^2)}{(1+|x|)^2(1-|z|^2)}
\end{align}

Now 
\begin{align}\label{abc2}
2|x-z|\sqrt{1-2x\cdot z + |x|^2|z|^2}&=2\sqrt{(|x|^2-2x\cdot z +|z|^2)(1-2x\cdot z + |x|^2|z|^2)}\\
&=2\sqrt{|x|^2-2(x\cdot z)(1+|x|^2) +O(|z|^2)}
\end{align}
On the other hand
\begin{align}\label{abc3}
\left(|x| - \frac{x \cdot z}{|x|}(1+|x|^2)\right)^2 = |x|^2 - 2(x \cdot z)(1+|x|^2) + O\left(\frac{|z|^2}{|x|}\right),
\end{align}
in the $C^1$-topology. Combining \eqref{abc1},\eqref{abc2} and \eqref{abc3} we get
\begin{align*}
e^{d(x,z)-d(x,0)} - 1 &= \frac{-4 x\cdot z  - 2\frac{x \cdot z}{|x|}(1+|x|^2) + O(|z|^2)}{(1+|x|)^2(1-|z|^2)} \\
&=-2\frac{\frac{x \cdot z}{|x|}(1 + |x|)^2 + O(|z|^2)}{(1+|x|)^2(1-|z|^2)} \\
& = -2\frac{x \cdot z}{|x|} + O(|z|^2).
\end{align*}
As a result, we deduce
\begin{align} \label{abc4}
d(x,z) - d(x,0) = \ln \left(1-2\frac{x \cdot z}{|x|} + O\left(\frac{|z|^2}{|x|}\right)\right) = -2\frac{x \cdot z}{|x|} + O\left(\frac{|z|^2}{|x|}\right).
\end{align}
We get from \eqref{someeq0}, \eqref{someeq4} and \eqref{abc4} that 
\begin{align}\label{abc5}
\calu[z](x) - \calu[0](x) + V[\calu[0]](x)\cdot z &= \int_0^1 \left(\calu^{\prime}(d(x,sz)) - \calu^{\prime}(d(x,0))\right)\frac{d}{ds}d(x,sz) \ ds \\
& \ \ \ \ \ \ \ \ \ \ \ + O\left(\calu^{\prime}(d(x,0))\frac{|z|^2}{|x|}\right),
\end{align}
in $H^1(\bn),$ which completes the proof of the lemma for $w=0$.

For general case, thanks to \eqref{abc4} we note that 
\begin{align*}
d(x,z) - d(x,w) &= d(\tau_{-w}(x), \tau_{-w}(z)) - d(\tau_{-w}(x),0) \\
&= -2\frac{\tau_{-w}(x) }{|\tau_{-w}(x)|}\cdot \tau_{-w}(z) + O\left(\frac{|\tau_{-w}(z)|^2}{|\tau_{-w}(x)|}\right).
\end{align*}
as $|\tau_{-w}(z)| \rightarrow 0,$ a consequence of $d(z,w) = d(\tau_{-w}(z), 0) = \ln \left(\frac{1 + |\tau_{-w}(z)|}{1-|\tau_{-w}(z)|}\right) \rightarrow 0.$ Now the proof can be completed as before.

\medskip

\noindent
(b) For simplicity we denote $u = 2\calu^{\prime}(d(\cdot, w))\frac{(\tau_{-w}(\cdot))_j}{|\tau_{-w}(\cdot)|},$ for some $j \in \{1,2,\ldots,n\},$ where $(\tau_{-w}(\cdot))_j$ denotes the $j$-th component of $\tau_{-w}(\cdot).$ Multiplying the equation satisfied by $\calu[w]:$ $-\Delta_{\bn} \calu[w] - \la \calu[w] = \calu[w]^p$ by $u$ and integrating by parts we get 
\begin{align*}
\langle \calu[w], u \rangle_{\la} &= \int_{\bn} \calu[w]^p u \ \dvg \\
&= 2\int_{\bn} \calu(d(x,w))^p\calu^{\prime}(d(x, w))\frac{(\tau_{-w}(x))_j}{|\tau_{-w}(\cdot)|} \ \dvg \\
&=2 \int_{\bn}\calu(d(x,0))^p \calu^{\prime}d(x,0)\frac{x_j}{|x|} \ \dvg \\
&=\int_{\bn} \calu[0]^p V_j(\calu[0]) \ \dvg = 0
\end{align*}
 where in the third line we used the change of variable formula.
 
 \medskip
 
 \noindent
 (c) Direct computation using the change of variable formula gives 
 \begin{align*}
 \int_{\bn} |u_j[w](x)|^{p+1} \ \dvg &\approx \int_{\bn} \left|\calu^{\prime}(d(x, w))\frac{(\tau_{-w}(x))_j}{|\tau_{-w}(x)|}
 \right|^{p+1} \ \dvg \\
 & = \int_{\bn} \left|\calu^{\prime}(d(x,0)) \frac{x_j}{|x|}\right|^{p+1} \ \dvg \approx \int_{\bn}|V_j(\calu[0])|^{p+1} \ \dvg \approx 1.
 \end{align*}
The  statement regarding the Sobolev norm now follows from the Poincar\'e-Sobolev inequality and the bound on $|\calu^{\prime \prime}| \lesssim \calu.$ 
\end{proof}

\begin{remark}
The expression of $u_j[w]$ is slightly different from  $V_j(\calu[w]).$ In the appendix section we will derive an explicit expression of $V_j(\calu[w])$ and show that $u_j[w]$ is a linear combination of $\{V_k(\calu[w])\}_{k=1}^n.$
\end{remark}


\begin{lemma} \label{compare norms}
There exists a constant $C_0$ such that for every $R \approx 1$ and  $\phi \in \mathcal{F},$ there holds
\begin{align*}
C_0^{-1}\|\phi\|_{L^2_{(U+V)^{p-1}}(\bn)} \leq \|\phi\|_{\la} \leq C_0\|\phi\|_{L^2_{(U+V)^{p-1}}(\bn)}.
\end{align*}
\end{lemma}

\begin{proof}
Without loss of generality we may assume that $\phi \in F_U,$ i.e., $\phi$ satisfies
\begin{align*}
-\Delta_{\bn} \phi - \la \phi = \mu U^{p-1}\phi, \ \ \ \ \mbox{where} \ \mu =1 \ \mbox{or} \ p.
\end{align*}
Multiplying the above equation by $\phi$ and integration by parts gives
\begin{align*}
\|\phi\|_{\la}^2 = \mu\int_{\bn}\phi^2U^{p-1} \ \dvg \leq \mu\int_{\bn}\phi^2(U+V)^{p-1} \ \dvg,
\end{align*}
which gives one side of the inequality. On the other side
\begin{align}\label{vbvbvb1}
\|\phi\|_{\la}^2 = \int_{\bn}\phi^2U^{p-1} \ \dvg = \int_{\bn}\phi^2(U+V)^{p-1} \ \dvg + \int_{\bn}\phi^2(U^{p-1} -(U+V)^{p-1}) \ \dvg.
\end{align}
We can estimate 
\begin{align}\label{vbvbvb}
\left|\int_{\bn}\phi^2(U^{p-1} -(U+V)^{p-1}) \ \dvg \right| &\leq \|\phi\|_{L^{p+1}(\bn)}^2\|U^{p-1}  - (U+V)^{p-1}\|_{L^{\frac{p+1}{p-1}}(\bn)}^{\frac{p-1}{p+1}} \notag\\
& \leq o(1) \|\phi\|_{\la}^2.
\end{align}
As a result \eqref{vbvbvb} and \eqref{vbvbvb1} together give the other side of the inequality, hence completing the proof of the lemma.
\end{proof}

\begin{lemma} \label{rho..distance}
Let $\rho$ be defined as in Lemma \ref{perturbed problem} and $u = U+V+\rho.$ If $R$
is sufficiently close to $1,$ then 
\begin{align*}
\|\rho\|_{\la} \lesssim \inf_{c_1,c_2 \in \R, z_1,z_2 \in \bn} \|u - c_1\calu[z_1] - c_2\calu[z_2]\|_{\la} 
\end{align*}
\end{lemma}

\begin{proof}
We define $\sigma  = U + V$ and $\sigma^{\prime} = c_1\calu[z_1] + c_2\calu[z_2],$ then 
$u - \sigma^{\prime} = \rho + (\sigma - \sigma^{\prime}).$
If $\|\sigma - \sigma^{\prime}\|_{\la} \geq 2 \|\rho\|_{\la},$ then $\|u - \sigma^{\prime}\|_{\la} \geq \|\rho\|_{\la}$ and hence there is nothing to prove. So we may assume $\|\sigma - \sigma^{\prime}\|_{\la} < 2 \|\rho\|_{\la}.$ Since $\|f\|_{L^2(\bn)} \rightarrow 0$ as $R\uparrow 1$ we have $\|\rho\|_{\la} \rightarrow 0$ and hence $\|\sigma - \sigma^{\prime}\|_{\la} \rightarrow 0$ as $R \uparrow 1.$ To begin with, we may without loss of generality assume that $\|U - c_1\calu[z_1]\|_{\la} = o(1)$ and $\|V - c_2\calu[z_2]\|_{\la} = o(1).$
We define 
\begin{align*}
\delta = |c_1 - 1|+|c_2-1| + |\tau_{Re_1}(z_1)|+|\tau_{-Re_1}(z_2)|,
\end{align*}
then $\delta = o(1)$ as $R\uparrow 1.$
We will show that $\|\sigma - \sigma^{\prime}\|_{\la} \approx \delta.$ We compute
\begin{align*}
\|\sigma - \sigma^{\prime}\|_{\la}^2 = \|U - c_1\calu[z_1]\|_{\la}^2 + \|V - c_2\calu[z_2]\|_{\la}^2 
+ 2\langle U - c_1\calu[z_1], V - c_2\calu[z_2] \rangle_{\la}.
\end{align*}
Note that $d(z_1, -Re_1) = d(z_1, \tau_{-Re_1}(0)) = d(\tau_{Re_1}(z_1),0) \approx |\tau_{Re_1}(z_1)|$ as 
$d(z_1,-Re_1) \rightarrow 0.$ Similarly for $z_2.$

Since $V_j(U)$ (respectively, $V_j(V)$) is $\langle \cdot,\cdot \rangle_{\la}$ orthogonal to $U$ (respectively, $V$) 
we deduce from Lemma \ref{derivative vu} 
\begin{align*}
\|U - c_1\calu[z_1]\|_{\la}^2 = \|(1 - c_1)U + c_1(U - \calu[z_1])\|_{\la}^2 
\approx (|c_1 - 1| + |\tau_{Re_1}(z_1)|)^2.
\end{align*}
 Similarly,
 \begin{align*}
 \|V - c_2\calu[z_2]\|_{\la}^2 = \|(1 - c_2)U + c_2(U - \calu[z_2])\|_{\la}^2 
\approx (|c_2 - 1| + |\tau_{-Re_1}(z_2)|)^2.
 \end{align*}
On the other hand, using again Lemma \ref{derivative vu} and fact that the interaction between $U$ and $V$ is small as $R \approx 1$ and $|V_j(U)| \lesssim U, \  |V_j(V)| \lesssim V$ we get
\begin{align}
\langle U -c_1\calu[z_1], V - c_2\calu[z_2]\rangle_{\la} \lesssim o(\delta^2) + \delta^2\|UV\|_{L^2(\bn)} = o(\delta^2),
\end{align} 
and hence 
\begin{align*}
\|\sigma - \sigma^{\prime}\|_{\la} \approx \delta.
\end{align*}
The next step is to show that 
\begin{align}\label{a09087}
|\langle \rho, \sigma - \sigma^{\prime}\rangle_{\la}| = o(\|\rho\|_{\la} \|\sigma - \sigma^{\prime}\|_{\la})
\end{align}
which would conclude that 
\begin{align*}
\|u - \sigma^{\prime}\|_{\la} = \|\rho + (\sigma - \sigma^{\prime})\|_{\la} \approx \|\rho\|_{\la} + \|\sigma - \sigma^{\prime}\|_{\la} \geq \|\rho\|_{\la},
\end{align*}
and complete the proof of the lemma. The remains of the proof of the lemma will be concentrated on the
 proof of \eqref{a09087}. Let $\rho^{\bot}$ be the orthogonal projection of $\rho$ on to $\mathcal{F}^{\bot}$ with respect to the $L^2_{(U+V)^{p-1}}(\bn)$ inner product. Recall that $\rho \in \mathcal{E}^{\bot}.$ Since the subspaces $\mathcal{E}$ and $\mathcal{F}$ are close and the distance between $\mathcal{E}$ and $\mathcal{F}$
is the same as the distance between $\mathcal{E}^{\bot}$ and $\mathcal{F}^{\bot},$ Lemma \ref{same dimension} and \eqref{generalized hdis} imply
\begin{align} \label{nmnmnm1}
\|\rho - \rho^{\bot}\|_{L^2_{(U+V)^{p-1}}(\bn)}\leq d(\mathcal{E},\mathcal{F})\|\rho\|_{L^2_{(U+V)^{p-1}}(\bn)} = o\left(\|\rho\|_{L^2_{(U+V)^{p-1}}(\bn)}\right).
\end{align}
 Moreover, from Lemma \ref{perturbed problem} and \eqref{rhofh-1} we deduce
 \begin{align}\label{nmnmnm2}
 \|\rho\|_{L^2_{(U+V)^{p-1}}(\bn)}^2 \leq \|\rho\|_{\la}^2 + \left|\int_{\bn} \tilde f \rho \ \dvg\right| \lesssim \|\rho\|_{\la}^2.
 \end{align}
On the other hand, since $\rho - \rho^{\bot} \in \mathcal{F},$ Lemma \ref{compare norms} gives
\begin{align}\label{nmnmnm3}
\|\rho - \rho^{\bot}\|_{L^2_{(U+V)^{p-1}}(\bn)} \approx \|\rho -  \rho^{\bot}\|_{\la}.
\end{align}
Hence \eqref{nmnmnm1}, \eqref{nmnmnm2}, \eqref{nmnmnm3} together gives
\begin{align} \label{mnbv1}
\|\rho -  \rho^{\bot}\|_{\la} = o(\|\rho\|_{\la}).
\end{align}

As a corollary of \eqref{mnbv1}  we deduce
\begin{align} \label{missing21}
|\langle \sigma - \sigma^{\prime}, \rho\rangle_{\la}| &\leq |\langle \sigma - \sigma^{\prime}, \rho^{\bot}\rangle_{\la}| + 
|\langle \sigma - \sigma^{\prime}, \rho - \rho^{\bot}\rangle_{\la}| \notag \\
& \leq |\langle \sigma - \sigma^{\prime}, \rho^{\bot}\rangle_{\la}| + \|\sigma - \sigma^{\prime}\|_{\la}\|\rho - \rho^{\bot}\|_{\la} \notag \\
& \leq |\langle \sigma - \sigma^{\prime}, \rho^{\bot}\rangle_{\la}| + o(\|\sigma - \sigma^{\prime}\|_{\la}\|\rho\|_{\la}),
\end{align}
 In order to estimate $|\langle \sigma - \sigma^{\prime}, \rho^{\bot}\rangle_{\la}|$ we decompose it as 
\begin{align} \label{mnbv2}
|\langle \sigma - \sigma^{\prime}, \rho^{\bot}\rangle_{\la}| \leq |\langle U - c_1\calu[z_1], \rho^{\bot}\rangle_{\la}|  + |\langle V- c_2\calu[z_2], \rho^{\bot}\rangle_{\la}|,
\end{align}
and noting that the two terms are identical in nature, it is enough to estimate the first term.

To proceed further, we use the following: Let us denote by $\mathcal{F}^{\bot}$ the $L^2_{(U+V)^{p-1}}(\bn)$-the orthogonal complement of $\mathcal{F}.$ Let $\psi \in \mathcal{F}^{\bot}\cap H^{1}(\bn)$ and $\phi \in \mathcal{F}.$ With out loss of generality assume that  $-\Delta_{\bn} \phi - \la \phi = \mu U^{p-1}\phi$ where $\mu = 1$ or $p.$ Then 
\begin{align} \label{mnbv3}
\left|\langle \psi , \phi\rangle_{\la} \right|=\left| \mu\int_{\bn} \psi  \phi U^{p-1} \ \dvg \right|&= \left| \mu\int_{\bn} \psi  \phi (U^{p-1} - (U+V)^{p-1}) \ \dvg \right| \notag\\
&\lesssim \|\psi\|_{L^{p+1}(\bn)}\|(U^{p-1} - (U+V)^{p-1})\|_{L^{\frac{p+1}{p-1}}(\bn)}\|\phi\|_{L^{p+1}(\bn)} \notag \\
&\lesssim o(1) \|\psi\|_{\la}\|\phi\|_{\la},
\end{align}
holds for every $\psi \in \mathcal{F}^{\bot} \cap H^1(\bn)$ and $\phi \in \mathcal{F}.$ Since $U, V_j(U) \in \mathcal{F}$ and $\rho^{\bot} \in \mathcal{F}^{\bot},$ using \eqref{mnbv3} we get
\begin{align*}
 |\langle (1 - c_1)U - c_1V(U) \cdot \tau_{Re_1}(z_1), \rho^{\bot}\rangle_{\la}| &\lesssim o(1)\|(1 - c_1)U - c_1V(U) \cdot \tau_{Re_1}(z_1)\|_{\la}\|\rho^{\bot}\|_{\la}\\
 &\lesssim o(1)(|c_1-1| + |\tau_{Re_1}(z_1)|)\|\rho\|_{\la} \\
 & \lesssim o(1)\delta \|\rho\|_{\la} \lesssim o(\|\sigma - \sigma^{\prime}\|_{\la}\|\rho\|_{\la}),
 \end{align*}
where in second last line we have used $\|\rho^{\bot}\|_{\la} \approx \|\rho\|_{\la}$ (thanks to \eqref{mnbv1}).  Therefore  
\begin{align}\label{mnbv5}
|\langle U - c_1\calu[z_1], \rho^{\bot}\rangle_{\la}| &\leq |\langle (1 - c_1)U - c_1V(U) \cdot \tau_{Re_1}(z_1), \rho^{\bot}\rangle_{\la}| \notag\\
& \ \ + |c_1\langle U + V(U) \cdot \tau_{Re_1}(z_1) - \calu[z_1], \rho^{\bot}\rangle_{\la}| \notag\\
&\leq o(\|\sigma - \sigma^{\prime}\|_{\la}\|\rho\|_{\la}) + \|U + V(U) \cdot \tau_{Re_1}(z_1) - \calu[z_1]\|_{\la}\|\rho^{\bot}\|_{\la} \notag\\
&\leq o(\|\sigma - \sigma^{\prime}\|_{\la}\|\rho\|_{\la}) + o(|\tau_{Re_1}(z_1)|)\|\rho\|_{\la}\notag\\
& \lesssim o(\|\sigma - \sigma^{\prime}\|_{\la}\|\rho\|_{\la}).
\end{align}
The desired bound \eqref{a09087} now follows from \eqref{mnbv1}, \eqref{missing21}, \eqref{mnbv2} and \eqref{mnbv5}.
\end{proof}

\medskip

\subsection{The function $u$ and $u^+$ are  counter examples}
In this section, we will complete the proof of counter-example. First, we prove Theorem \ref{main counter example}
and then we will show as in \cite{FG} that $u^{+} := \max\{u,0\}$ satisfies the same assertion of Theorem \ref{main counter example}. Before starting the proof let us briefly state and prove a lemma which would give us control on $\| u^{-}\|_{\lambda}.$

\begin{lemma}\label{nonlem1}
There holds 

\begin{equation}
\|u^-\|_{\lambda} \lesssim 
\begin{cases}
\|\rho\|_{\la}|\ln \|\rho\|_{\la}|^{-\frac{1}{4}}, \ \ \ \ \mbox{if} \ p=2, \\
\|\rho\|_{\la}^{1+ \frac{\alpha_0}{2}}, \ \ \ \ \ \ \ \ \ \ \ \ \ \ \mbox{if} \ 1 < p < 2,
\end{cases}
\end{equation}
where $u^- = - \mbox{min} \{ u, 0 \}.$

\end{lemma}

\begin{proof}
By applying divergence theorem to the vector field $u^- \nabla_{\bn} u$ we deduce 

\begin{equation}\label{eq1-nonlem1}
\int_{\bn} |\nabla_{\bn} u^{-}|^2 \, \dvg = \int_{ \bn \cap \{ u < 0 \}} u\, (-\Delta_{\bn} u) \, \dvg.
\end{equation}
We recall that $u = U + V + \rho,$ then $u$ satisfies the following equation

$$
-\Delta_{\bn} u \, = \, \lambda u + U^p + V^p +p (U + V)^{p-1} \rho \, + \, \tilde{f}.
$$
Moreover if $u < 0,$ then $U + V < |\rho|.$ Hence in the region $\{ u < 0 \}$ it holds $|u| \lesssim |\rho|$ and $f \lesssim |\rho|^p.$ Hence 

$$
\int_{ \bn \cap \{ u < 0 \}} u\, (-\Delta_{\bn} u) \, \dvg = \lambda \int_{ \bn \cap \{ u < 0 \}} (u^-)^2 \, \dvg \, + 
\, \int_{ \bn \cap \{ u < 0 \}} u(U^p + V^p +p (U + V)^{p-1} \rho \, + \, \tilde{f}) \, \dvg.
$$
Now substituting back in \eqref{eq1-nonlem1} we obtain

\begin{align*}
\| u^{-}\|_{\lambda}^2 \lesssim \int_{\bn} \rho^{p+1} \, + \, \int_{\bn} |\rho| |f - \tilde{f}| \, \dvg \lesssim \|\rho\|^{p+1}_{\lambda} \, + 
\, \| \rho\|_{\lambda} \| f - \tilde{f}\|_{L^{\frac{p+1}{p}}(\bn)}.
\end{align*}
Now the result follows from Lemma~\ref{ffh-1}.

\end{proof}

\medskip

\noindent
{\bf Proof of Theorem \ref{main counter example}.}

\begin{proof}
Let $u = U + V + \rho$ be as defined in \eqref{defurho} and set $\mathcal{D} = \inf_{c_1,c_2 \in \R, z_1,z_2 \in \bn} \|u - c_1\calu[z_1] - c_2\calu[z_2]\|_{\la}.$ Note that $u$ depends on $R.$ Then by \eqref{rhofh-1}
\begin{align*}
\mathcal{D} \leq \|u - U - V\|_{\la} = \|\rho\|_{\la} \lesssim \|f\|_{L^2(\bn)} \rightarrow 0,
\end{align*}
 as $R \uparrow 1.$ On the other hand by Lemma \ref{rho..distance}, $\|\rho\|_{\la} \lesssim \mathcal{D}$ and hence $\|\rho\|_{\la} \approx \mathcal{D}$ as $R \uparrow 1.$ Using this in Lemma \ref{finallemma} we infer that
 \begin{align*}
\|\Delta_{\bn} u + \la u + |u|^{p-1}u\|_{L^{\frac{p+1}{p}}(\bn)} \lesssim
\begin{cases}
\mathcal{D}|\ln\mathcal{D}|^{-\frac{1}{2}}, \ \ \ \ \mbox{if} \ p=2, \\
\mathcal{D}^{1+\alpha_0}, \ \ \ \ \ \ \ \ \ \  \mbox{if} \ 1 < p < 2,
\end{cases}
\end{align*}
completing the proof of part $(a)$ and part $(b)$ of the theorem.

\medskip 

Now we begin the proof for $u^{+},$ i.e., part $(c).$ We essentially follow the steps as in \cite[Section~4.6]{FG}. For the sake of brevity, we shall outline the minor modifications.
 We prove that $u^{+} = (u_{R})^{+}$ satisfies the requirements. One can see this immediately using Lemma~\ref{nonlem1}

$$
\left|   \| u^+ - U - V\|_{\lambda} - \| u - U -V\|_{\lambda} \right| \leq \| u^{-}\|_{\lambda}  \lesssim \circ(\| u - U - V\|_{\lambda}).
$$
The rest of the proof follows by realising the following estimates 

\begin{equation}\label{eq-1non}
\| u^+ - c_1 \calu[z_1] - c_2 \calu[z_2] \|_{\la} \approx \| u - c_1 \calu[z_1] - c_2 \calu[z_2] \|_{\la},
\end{equation}
and 

\begin{equation}\label{eq-2non}
\| \Delta_{\bn} u^+ + \lambda u^+ + (u^+)^p\|_{H^{-1}} \lesssim
\begin{cases}
\| \Delta_{\bn} u + \lambda u + u|u|^{p-1}\|_{H^{-1}} \, + \, \|\rho\|_{\la}|\ln \|\rho\|_{\la}|^{-\frac{1}{4}}, \ \ \ \ \mbox{if} \ p=2,\\
\| \Delta_{\bn} u + \lambda u + u|u|^{p-1}\|_{H^{-1}} \, + \, \|\rho\|_{\la}^{1+ \frac{\alpha_0}{2}}, \ \ \ \ \ \ \ \ \ \ \ \ \ \ \mbox{if} \ 1 < p < 2.
\end{cases}
\end{equation}
The proof of the above estimates follows along the line of \cite[Proof of Theorem~4.3]{FG}. Indeed, using triangle inequality we obtain 

\begin{align*}
\left| \| u^+ -  c_1 \calu[z_1] - c_2 \calu[z_2] \|_{\la} - \| u - c_1 \calu[z_1] - c_2 \calu[z_2]\|_{\la} \right| &\leq \| u^+ - u\|_{\lambda} = \| u^{-}\|_{\lambda} \\
& \lesssim\begin{cases}
\|\rho\|_{\la}|\ln \|\rho\|_{\la}|^{-\frac{1}{4}}, \ \ \ \ \mbox{if} \ p=2,\\
\|\rho\|_{\la}^{1+ \frac{\alpha_0}{2}}, \ \ \ \ \ \ \ \ \ \ \ \ \ \ \mbox{if} \ 1 < p < 2.\\
\end{cases}
\end{align*}
Now the estimate \eqref{eq-1non} follows by using Lemma~\ref{rho..distance} and the fact that $ \frac{t |\log t|^{-\frac{1}{4}}}{t} \rightarrow 0$ and $\frac{t^{1 + \frac{\alpha_0}{2}}}{t} \rightarrow 0$
when $t \rightarrow 0.$ For \eqref{eq-2non}, using triangle and Sobolev inequalities we deduce 

\begin{align*}
\left| \| \Delta_{\bn} u^+ + \lambda u^+ (u^+)^p\|_{H^{-1}} - \| \Delta_{\bn} + \lambda u + u|u|^{p-1}\| \right|_{H^{-1}} & \lesssim 
\| u^+ - u\|_{\lambda} + \| (u^+)^p - u |u|^{p-1}\|_{H^{-1}} \\
& \lesssim \| u^-\|_{\lambda} + \| (u^-)^{p}\|_{L^{\frac{p+1}{p}}(\bn)} \\
& \lesssim \| u^-\|_{\lambda} + \| u^-\|^p_{\lambda} \lesssim \| u^-\|_{\lambda}.
\end{align*}
The desired estimate now follows from \eqref{eq-1non}, \eqref{eq-2non} and Lemma~\ref{nonlem1}. This completes the proof of part $(c)$ of  Theorem~\ref{main counter example}.
\end{proof}

\medskip


\appendix \section{Non-degenaracy and spectral properties of the linearized operator} \label{Appendix}
\setcounter{equation}{0}

In this subsection, we collect a few key lemmas needed for the proof related to  the eigenvalues and eigenfunctions of the linearized operator 
$$\mathcal{L} := (-\Delta_{\bn} - \lambda)/\mathcal{U}[z]^{p-1}.$$
We know that if $\tau_b$ is a hyperbolic translation then $\mathcal{U}[z] \circ \tau_b,$ also solves \eqref{eq1} and hence the kernel of the linearized operator contains non-trivial elements. It was shown in \cite{DG2} that the degeneracy happens only along an $n$-dimensional subspace characterized by the vector fields 
\begin{align*}
V_j(x) := (1 + |x|^2) \frac{\partial}{\partial x_j} \, - \, 2x_j \sum_{l =1}^{n} x_l \, \frac{\partial}{\partial x_l}.
\end{align*}
 for $j = 1, \ldots, n.$ More precisely, we define 
\begin{align*}
 \Phi_j(x) := \frac{d}{dt} \Big|_{t=0} \mathcal{U}[z] \circ \tau_{t e_j} , \ \ 1 \leq i \leq n,
 \end{align*}
then $\Phi_{j}(x) = V_j(\calu[z]),$ $\Phi_j$ solves the eigenvalue problem 
\begin{align} \label{eigenvalue}
-\Delta_{\bn} \Phi_j - \lambda \Phi_j = p \,  \mathcal{U}[z]^{p-1} \, \Phi_j.
\end{align}

and the degeneracy in the solution space to \eqref{eigenvalue} can occur only along the directions $\Phi_j, 1 \leq j \leq n.$ 
 \begin{theorema}[\cite{DG2} ]
 Let $V_j$ be the vector fields in $\bn$ defined above and $\Phi_j = V_j(\calu[z]).$ Then $\{\Phi_j\}_{j=1}^n$ forms a basis for the kernel of $(-\Delta_{\bn} - \lambda - p\calu[z]^{p-1}).$
 \end{theorema}

 As a result, we obtain complete information on the first and the second eigenvalues and corresponding eigenspaces of the operator $\mathcal{L}.$ We recall the following result from our earlier work \cite{BGKM}

\begin{proposition} \label{eigen value lemma}
The first and the second eigenvalues of the operator $$\mathcal{L} := (-\Delta_{\bn} - \lambda)/\mathcal{U}[z]^{p - 1}$$ 
are respectively $1$ and $p.$ Moreover, the first eigenspace is one dimensional and spanned by $\mathcal{U}[z]$ and the second eigenspace is $n$-dimensional and spanned by $\{\Phi_i\}_{1 \leq i \leq n}.$
\end{proposition}

In this sequel, we also recall another relevant result used in this article whose proof can be found in \cite{DG2}.

\begin{lemma}\label{half ev}
The first Dirichlet eigenvalue of the operator $$\mathcal{L} := (-\Delta_{\bn} - \lambda)/\mathcal{U}[z]^{p - 1}$$ 
in the negative half $\bn_{-}:= \{ x = (x_1, \ldots, x_n) \in \bn \ : \ x_j <0 \}$ is  $p$ for every $j = 1,\ldots,n$ and the corresponding eigenspace is one dimensional spanned by $V_j(\calu[z]).$ 
\end{lemma}

\subsection{Some properties of $V_j(\calu[z])$}
The main purpose of this section is to derive an explicit expression of $V_j(\calu[z])$ where $z \in \bn$ and $1 \leq j \leq n$ and to show that several norms of $V_j(\calu[z])$ remains uniformly bounded below and above as $z \rightarrow \infty$ in $\bn.$ We divide the proof into several small lemmas.

\begin{lemma} \label{derivative of distance}
For every $z \in \bn$ and $1 \leq j \leq n,$
\begin{align*}
\sinh d(x,z) \ \frac{d}{dt} \Big|_{t = 0} d(\tau_{te_j}(x),z) = \frac{4[x_j(1+|z|^2) - z_j(1+|x|^2)]}{(1-|z|^2)(1-|x|^2)},
\end{align*}
for all $x \in \bn.$
\end{lemma}
\begin{proof}
Differentiating the hyperbolic translation $\tau_{te_j}(x) = \frac{(1-t^2)x + t(|x|^2 + 2tx_j + 1)e_j}{t^2|x|^2 + 2tx_j+1}$ we get
\begin{align*}
\frac{d}{dt}\Big|_{t=0} \tau_{te_j}(x) = (1+|x|^2)e_j - 2x_j x.
\end{align*}
Again differentiating the $\cosh d(\tau_{te_j}(x),z) = 1 + \frac{2|\tau_{te_j}(x)-z|^2}{(1-|z|^2)(1 - |\tau_{te_j}(x)|^2)}$
we get
\begin{align*}
\sinh d(x,z) \ \frac{d}{dt} \Big|_{t = 0} d(\tau_{te_j}(x),z) = &\frac{4(x-z)\cdot [(1+|x|^2)e_j - 2x_j x]}{(1-|z|^2)(1-|x|^2)} \\
&+ \frac{4|x-z|^2}{(1-|z|^2)(1-|x|^2)^2} \ x \cdot[(1+|x|^2)e_j - 2x_j x] \\
& = \frac{4[(1+|x|^2)e_j - 2x_j x]}{(1-|z|^2)(1-|x|^2)^2} \cdot  [(x-z)(1-|x|^2) + |x-z|^2x]
\end{align*}
Direct computations give the result. 
\end{proof}

As a consequence of the above lemma, we have the following 
\begin{corollary} The explicit expression of $V_j(\calu[z])$ is given by
\begin{align*}
V_j(\calu[z])(x) = 2\calu^{\prime}(d(x,z)) \frac{x_j(1+|z|^2) - z_j(1+|x|^2)}{|x-z|\sqrt{1 - 2x\cdot z + |x|^2|z|^2}}, \ \ 
\mbox{for} \ x \in \bn.
\end{align*}

\end{corollary}
\begin{proof}
By definition 
\begin{align}\label{exp1}
V_j(\calu[z])(x) &= \frac{d}{dt} \Big|_{t=0} \calu[z] \circ \tau_{te_j}(x) \notag\\
&=  \frac{d}{dt} \Big|_{t=0} \calu(d( \tau_{te_j}(x), z)) \notag\\
&= \calu^{\prime}(d(x,z))  \frac{d}{dt} \Big|_{t=0} d( \tau_{te_j}(x), z).
\end{align}
Again from the explicit expression of the hyperbolic distance
\begin{align}\label{exp2}
\sinh d(x,z) = 2 \sinh \frac{d(x,z)}{2} \cosh \frac{d(x,z)}{2} = 2 \frac{|x-z|\sqrt{1 - 2x\cdot z+ |x|^2|z|^2}}{(1-|z|^2)(1 - |x|^2)}.
\end{align}
The lemma now follows from \eqref{exp1}, \eqref{exp2} and Lemma \ref{derivative of distance}.
\end{proof}

Note that $V_j$ and $u_j$ from Lemma \ref{derivative vu} are not the same. Actually, $u_j$ are linear combinations of $V_k$s
as indicated in the next lemma.

\begin{lemma}
For $j \in \{1,2,\ldots,n\}$ we have
\begin{align*}
u_j[z] = \sum_{k=1}^n c_kV_k(\calu[z]),
\end{align*}
where 
\begin{align*}
c_k =
\begin{cases}
\frac{z_kz_j}{1 + |z|^2}, \ \ \ \ \ \ \ \ \mbox{if} \ k \neq j, \\ \\
\frac{1 - |z|^2 - 2z_j^2}{1 + |z|^2}, \ \ \ \mbox{if} \ k = j.
\end{cases}
\end{align*}
\end{lemma}
\begin{proof}
It is not difficult to check from the expression of $\tau_{-z}(x) = \frac{(1 - |z|^2)x - (|x|^2 - 2 x \cdot z + 1)z}{|z|^2|x|^2 - 2 x\cdot z + 1}$ that 
\begin{align*}
|\tau_{-z}(x)|^2 = \frac{|x-z|^2}{|z|^2|x|^2 - 2 x\cdot z + 1}.
\end{align*}
 Hence 
 \begin{align*}
 u_j[z](x) &= 2 \calu^{\prime}(d(x,z)) \frac{(\tau_{-z}(x))_j}{|\tau_{-z}(x)|} \\
 & = 2\calu^{\prime}(d(x,z)) \frac{(1 - |z|^2)x_j - (|x|^2 - 2 x \cdot z + 1)z_j}{|x-z|\sqrt{1 - 2x\cdot z + |x|^2|z|^2}}\\
 & = V_j(\calu[z])
 + 2\calu^{\prime}(d(x,z)) \frac{2((x \cdot z)z_j - |z|^2x_j)}{|x-z|\sqrt{1 - 2x\cdot z + |x|^2|z|^2}}
 \end{align*}
Therefore it is enough to express $2((x \cdot z)z_j - |z|^2x_j)$ in terms of $x_k(1+|z|^2) - z_k(1+|x|^2)$.
We put 
\begin{align*}
2((x \cdot z)z_j - |z|^2x_j) = \sum_{k=1}^n d_k(x_k(1+|z|^2) - z_k(1+|x|^2)),
\end{align*}
which should hold for all $x \in \bn.$ Putting $x=0$ we derive $\sum_{k=1}^nd_kz_k = 0$ and hence 
\begin{align*}
 \sum_{k=1}^n d_kx_k(1+|z|^2) = 2((x \cdot z)z_j - |z|^2x_j) = 2 \sum_{k \neq j} z_kz_j x_k + 2(z_j^2 - |z|^2)x_j.
\end{align*}
Comparing the coefficients we deduce $d_k = 2z_kz_j/(1 + |z|^2)$ for $k \neq j$ and $d_j = 2(z_j^2 - |z|^2)/(1 + |z|^2).$
It is now easy to see that $\sum_k d_kz_k = 0$ is also satisfied. Plugging it in the expression of $u_j$ we conclude the result.
\end{proof}

Though not needed for this article it is interesting to investigate the corresponding results of Lemma \ref{derivative vu} for 
$V_j$'s. We do it for the particular case $z = Re_1.$
For $j = 2, \ldots, n$

\begin{align}\label{derivative of sinh}
\sinh d(x,Re_1) \ \frac{d}{dt} \Big|_{t = 0} d(\tau_{te_j}(x),Re_1) =  \frac{4[x_j(1+R^2)]}{(1-R^2)(1-|x|^2)}
\end{align}

Again from the formula of the hyperbolic distance, we see that 
\begin{align*}
\sinh d(x,Re_1) &= 2 \sinh \frac{d(x,Re_1)}{2}\cosh \frac{d(x,Re_1)}{2} \\
&= \frac{2R|x-Re_1||x-R^{-1}e_1|}{(1-R^2)(1-|x|^2)} \lesssim \frac{1}{(1-R^2)(1-|x|^2)}.
\end{align*}
On the other hand, from \eqref{derivative of sinh}, we deduce that on the region $\{x \in \bn \ | \ |x_j| \geq \delta\}$
for some $\delta > 0$ small 
\begin{align} \label{der of hype}
 \left | \frac{d}{dt} \Big|_{t = 0} d(\tau_{te_j}(x),Re_1) \right| \gtrsim_{\delta} 1.
\end{align}

When $j=1,$ we have $x_1(1+R^2) - R(1+|x|^2) = R(x-Re_1)\cdot (x-R^{-1}e_1)$ and hence 
\begin{align}
 \frac{d}{dt} \Big|_{t = 0} d(\tau_{te_1}(x),Re_1) = \frac{f_R(x)}{\sinh (x,Re_1)},
\end{align}
where $f_R(x) = \frac{4R(x-Re_1)\cdot (x-R^{-1}e_1)}{(1-R^2)(1-|x|^2)}.$ In either case it is easy to see that 
\begin{align*}
\left| \frac{d}{dt} \Big|_{t = 0} d(\tau_{te_j}(x),Re_1)\right| \lesssim 1, \ \ \mbox{for all} \ j=1,2,\ldots,n.
\end{align*}
 Now we find an explicit expression of $V_j(\calu[Re_1]).$ By definition $\calu[Re_1](x) = \calu \circ \tau_{-Re_1}(x)$ and hence $\calu[Re_1] \circ \tau_{te_j}(x) = \calu(d(\tau_{-Re_1}(\tau_{te_j}(x)),0)) = \calu(d(\tau_{te_j},Re_1)).$ Hence
 \begin{align*}
 V_j(\calu[Re_1]) = \frac{d}{dt}\Big|_{t=0} \calu(d(\tau_{te_j},Re_1)) = \calu^{\prime}(d(x,Re_1)) \frac{d}{dt} \Big|_{t = 0} d(\tau_{te_j}(x),Re_1),
 \end{align*}
which together with $|\calu^{\prime}|\lesssim \calu$ gives
\begin{align}\label{calure1}
|V_j(\calu[Re_1])| \lesssim \calu[Re_1], \ \ \mbox{for all} \ j = 1,2,\ldots,n.
\end{align}

\medskip

With this preparation, we can now state the main result on $V_j(\calu[Re_1]).$

\begin{lemma} \label{h1normvu}
For every $1 \leq j \leq n$ and $R \sim 1$ 
\begin{align*}
\|V_j(\calu[Re_1])\|_{L^{2}(\bn)} \approx 1, \ \ \|V_j(\calu[Re_1])\|_{\la}\approx 1.
\end{align*}
\end{lemma}

 \begin{proof}
 The upper bound follows easily from the equation satisfied by $\Phi_j:= V_j(\calu[Re_1])$ and \eqref{calure1}.  Indeed, multiplying \eqref{eigenvalue} by $\Phi_j$, 
 \begin{align*}
 \|\Phi_j\|_{\la}^2 = p\int_{\bn}\Phi_j^2\calu[Re_1]^{p-1} \ \dvg \lesssim \int_{\bn} \calu[Re_1]^{p+1} \ \dvg \lesssim 1,
 \end{align*}
 and Poincar\'e inequality gives the required $L^2$-bound. For the lower bound when $j>1$ we use \eqref{der of hype} to get 
 \begin{align*}
 |\Phi_j(x)| \gtrsim_{\delta} |\calu^{\prime}(d(x,Re_1))|, \ \ \ \mbox{on} \ \ |x_j|\geq \delta.
 \end{align*}
 We fix $R_0$ such that the hyperbolic ball $B(Re_1, R_0)$ is entirely contained in the region $|x_j| \geq \delta.$ Then 
 \begin{align*}
 \int_{\bn} |\Phi_j(x)|^2 \ \dvg \gtrsim_{\delta} \int_{B(Re_1,R_0)}|\calu^{\prime}(d(x,Re_1))|^2 \ \dvg = \int_{B(0,R_0)}|\calu^{\prime}|^2 \ \dvg \approx 1.
 \end{align*}
Now for $j = 1,$ 
\begin{align*}
\int_{\bn} |\Phi_1(x)|^2 \ \dvg &= \int_{\bn} (\calu^{\prime}(d(x,Re_1)))^2 \frac{f_R(x)^2}{\sinh^2 d(x,Re_1)} \ dvg \\
&= \int_{\bn} (\calu^{\prime}(d(y,0)))^2 \frac{f_R(\tau_{Re_1}(y))^2}{\sinh^2 d(y,0)},
\end{align*}
where we have used the change of variable $x = \tau_{Re_1}(y).$ Let us denote $D = R^2|x|^2 + 2Ry_1+1.$ A few straightforward computations give
\begin{align*}
\tau_{Re_1}(y) - Re_1 &= (1-R^2) \ \frac{y+R|y|^2e_1}{D} \\
\tau_{Re_1}(y) - R^{-1}e_1 &= (1-R^2) \ \frac{y-2y_1e_1 - R^{-1}e_1}{D} \\
1 - |\tau_{Re_1}(y)|^2 &= \frac{(1-R^2)(1-|y|^2)}{D}.
\end{align*}
Hence from the definition of $f_R$, we get
\begin{align*}
f_R(\tau_{Re_1}(y)) &= \frac{4R(\tau_{Re_1}(y) - Re_1)\cdot (\tau_{Re_1}(y) - R^{-1}e_1)}{(1-R^2)(1 - |\tau_{Re_1}(y)|^2)}\\
&= 4R \frac{(y + R|y|^2e_1)\cdot(y - 2y_1e_1-e_1)}{D(1-|y|^2)} \\
&\rightarrow  \frac{4(y + |y|^2e_1)\cdot(y - 2y_1e_1-e_1)}{|y+e_1|^2(1-|y|^2)} = \frac{4g(y)}{(1-|y|^2)}, \ \ \mbox{as} \ \ R \uparrow 1.
\end{align*}
Clearly, $g$ is bounded above. Recalling $\sinh d(y,0) = \frac{2|y|}{(1-|y|^2)}$ we get
\begin{align*}
\liminf_{R \rightarrow 1} \int_{\bn} |\Phi_1(x)|^2 \ \dvg \geq \int_{\bn} (\calu^{\prime}(d(y,0)))^2 \frac{g(y)^2}{|y|^2} \ \dvg \approx 1,
\end{align*}
which completes the proof.
 \end{proof}

\subsection{Spectral properties of weighted Laplace-Beltrami operator}

In this subsection, we recall some well-known results on the spectral properties of $(-\Delta_{\bn} - \la)/w.$ The basic example of weights we consider is one of  $U^{p-1}, V^{p-1}, (U+V)^{p-1}.$ Hence we assume 
\begin{align} \label{assumption w}
0 < w \in L^{\frac{p+1}{p-1}}(\bn)\cap L^{\infty}(\bn) \ \ \mbox{and the embedding} \ H^1(\bn)  \hookrightarrow L^2_w(\bn) \ \ \mbox{is compact}.
\end{align}
As a result, $\left(\frac{-\Delta_{\bn} - \la}{w}\right)^{-1}: L^{2}_w(\bn) \to L^2_w(\bn)$ is a compact, self-adjoint and positive operator. By standard spectral theory, $L^2_w(\bn)$ admits a countable orthonormal basis $\{\psi_k\}_{k=1}^{\infty}$ consisting of eigenvectors of $\left(\frac{-\Delta_{\bn} - \la}{w}\right)^{-1}.$ Up to a rearrangement of indices we can write
\begin{align*}
-\Delta_{\bn}\psi_k - \la \psi_k = \mu_k w\psi_k, \ \ 0<\mu_1 < \mu_2\leq \mu_3\leq \cdots \leq \mu_k \rightarrow \infty, \ as \ k \rightarrow \infty.
\end{align*}
Moreover $\{\mu_k^{-\frac{1}{2}}\psi_k\}$ forms an orthonormal basis for $H^1(\bn)$ with respect to the bilinear form $\langle \cdot, \cdot\rangle_{\la}.$

The following theorem is a straightforward adaptation of \cite[Lemma A.4 and Lemma A.5]{FG}.

\begin{lemma} \label{equivalence of norms}
Let $n \geq3$ and $w$ be a weight satisfying \eqref{assumption w}. Let $\psi \in L^2_w(\bn)$ satisfies $-\Delta_{\bn}\psi - \la \psi - \mu w \psi = f.$ If we write $\psi = \sum_k \alpha_k \psi_k$ then
\begin{itemize}
\item[(a)] $\sum_k \alpha_k^2 \left(1 - \frac{\mu}{\la_k}\right)^2 = \|(-\Delta_{\bn} - \la)^{-1}f\|_{L^2_w(\bn)}^2.$
\item[(b)] Moreover, if we assume $\la_k \geq \frac{\mu}{1-\epsilon},$ whenever $\alpha_k \neq 0,$ where $\epsilon \in (0,1)$ then 
\begin{align*}
 \|\psi\|_{\la} \approx_{\epsilon} \|f\|_{H^{-1}(\bn)}.
\end{align*}
\end{itemize}
\end{lemma}

\begin{proof}
Plugging $\psi = \sum_k \alpha_k\psi_k$ in to $-\Delta_{\bn}\psi - \la \psi - \mu w \psi = f$ we get
\begin{align*}
\frac{f}{w} &= \frac{1}{w}\sum_k \alpha_k (-\Delta_{\bn}\psi_k - \la \psi_k - \mu w \psi_k) \\
&= \frac{1}{w}\sum_k \alpha_k(\la_kw\psi_k - \mu w \psi_k)\\
&= \sum_{k} \alpha_k(\la_k -\mu_k)\psi_k.
\end{align*}
Appling $\left(\frac{-\Delta_{\bn} - \la}{w}\right)^{-1}$ both sides and using $\left(\frac{-\Delta_{\bn} - \la}{w}\right)^{-1}(\psi_k) = \frac{\psi_k}{\la_k}$ we obtain
\begin{align}\label{z-1}
\sum_{k} \alpha_k \left(1 - \frac{\mu}{\la_k}\right)\psi_k = (-\Delta_{\bn} - \la)^{-1}f.
\end{align}
Since $\{\psi_k\}$ forms an orthonormal basis for $L^2_w(\bn)$, taking $L^2_w$-norm on both sides of \eqref{z-1}
we get $(a).$ 

For (b), we recall that $\{\la_k^{-\frac{1}{2}}\psi_k\}$  forms an orthonormal basis of $H^1(\bn)$ with respect to $\langle \cdot, \cdot \rangle_{\la},$ and hence 
\begin{align}
\|\psi \|_{\la}^2 = \sum_k \alpha_k^2 \la_k.
\end{align}
 On the other hand, taking the $\|\cdot\|_{\la}$-norm on both sides of \eqref{z-1} we get
 \begin{align*}
 \sum_{k} \alpha_k^2\left(1 - \frac{\mu}{\la_k}\right)^2 \la_k \ = \ \| (-\Delta_{\bn} - \la)^{-1}f \|_{\la}^2 \ \approx \ \| f \|_{H^{-1}(\bn)}^2.
 \end{align*}
 As a result, we get $\|\psi\|_{\la} \lesssim \|f\|_{H^{-1}(\bn)}.$ To prove the equivalence $\approx$ in last line we set $v = (-\Delta_{\bn} - \la)^{-1}f,$ then $(-\Delta_{\bn} - \la)v = f$ and hence 
 $\|v\|_{\la}^2 = \int_{\bn} f v \ \dvg \leq \|f\|_{H^{-1}(\bn)}\|v\|_{H^1(\bn)} \lesssim  \|f\|_{H^{-1}(\bn)}\|v\|_{\la},$
 which gives one side of the equivalence. For the other side, for every $\varphi \in C_c^{\infty}(\bn)$
 \begin{align*}
 \left| \int_{\bn} f \varphi \ \dvg \right|  = |\langle v, \varphi \rangle_{\la}| \leq \|v\|_{\la} \| \varphi \|_{\la} \lesssim \|v\|_{\la} \| \varphi \|_{H^{1}(\bn)}.
 \end{align*}
Taking supremum over all $\| \varphi \|_{H^{1}(\bn)} \leq 1$ proves the equivalence. 

\medskip

Now if $\la_k \geq \mu(1-\epsilon)^{-1}$ whenever $\alpha_k \neq 0,$ then $1- \frac{\mu}{\la_k} \geq \epsilon$ and hence
\begin{align*}
\|f\|_{H^{-1}(\bn)}^2 \approx  \sum_{k} \alpha_k^2\left(1 - \frac{\mu}{\la_k}\right)^2 \la_k \geq \epsilon^2 \sum_k \alpha_k^2\la_k^2 = \epsilon^2 \| \psi \|_{\la}^2.
\end{align*}
\end{proof}

We also used the following simple lemma.

\begin{lemma} \label{bounded operator}
Let $w$ be a weight satisfying \eqref{assumption w} (but not necessarily bounded), then $(-\Delta_{\bn} - \la)^{-1} : L^2(\bn) \to L^2_w(\bn)$ is a bounded operator.
\end{lemma}
\begin{proof}
For every $u \in C_c^{\infty}(\bn),$ we have 
\begin{align*}
\|(-\Delta_{\bn} - \la)u\|_{L^2(\bn)} &= \sup_{\|\varphi \|_{L^2(\bn)}= 1} \int_{\bn} (-\Delta_{\bn} - \la)u \varphi \ \dvg \\
&\geq \frac{1}{\|u\|_{L^2(\bn)}}\int_{\bn} u(-\Delta_{\bn} - \la)u \ \dvg \\
& = \frac{\| u\|_{\la}^2}{\|u\|_{L^2(\bn)}} \\
& \gtrsim \|u\|_{\la} \gtrsim \|u\|_{L^{p+1}(\bn)},
\end{align*}
where in the second last line we have used Poincar\'{e} inequality.  Since $w \in L^{\frac{p+1}{p-1}}(\bn)$ 
\begin{align*}
\int_{\bn} u^2w \ \dvg \leq \| u \|_{L^{p+1}(\bn)}^2 \|w\|_{L^{\frac{p+1}{p-1}}(\bn)} \lesssim \|(-\Delta_{\bn} - \la)u\|_{L^2(\bn)}^2
\end{align*}
and hence $\|(-\Delta_{\bn} - \la)^{-1} u\|_{L^{2}_w(\bn)} \lesssim \| u\|_{L^2(\bn)}.$

\end{proof}


\par\bigskip\noindent
\textbf{Acknowledgments.}
The research of M.~Bhakta is partially supported by the {\em SERB WEA grant (WEA/2020/000005)} and DST Swarnajaynti fellowship (SB/SJF/2021-22/09). D.~Ganguly is partially supported by the INSPIRE faculty fellowship (IFA17-MA98). D.~Karmakar acknowledges the support of the Department of Atomic Energy, Government of India, under project no. 12-R\&D-TFR-5.01-0520. S.~Mazumdar is partially supported by IIT Bombay SEED Grant RD/0519-IRCCSH0-025.

\end{document}